\documentclass[ejs, preprint]{imsart}

\startlocaldefs
\endlocaldefs

 \setattribute{journal}{name}{} 
 \setattribute{journal}{issn}{}	
 
\usepackage{epsfig,subfigure}
\usepackage[utf8]{inputenc}
\usepackage{amssymb}
\usepackage{amsfonts}
\usepackage{amsthm}
\usepackage{amsmath,multirow,rotating}
\usepackage[a4paper, layout=a4paper, marginratio={1:1, 2:3}]{geometry}
\usepackage[round]{natbib}
\usepackage{bm}
\usepackage[colorlinks=true,linkcolor=blue,citecolor=blue,pdfborder={0 0 0}]{hyperref}
\usepackage{color}
\usepackage{paralist}

\usepackage{mathrsfs}
\usepackage{dsfont}

\usepackage{enumitem}

\usepackage{graphicx}%

\usepackage{array}
\newcolumntype{C}[1]{>{\centering}m{#1}}

\newtheorem{theorem}{Theorem}[section]
\newtheorem{prop}[theorem]{Proposition}
\newtheorem{lem}[theorem]{Lemma}

\newtheorem{thm}[theorem]{Theorem}

\theoremstyle{definition}

\newtheorem{cond}[theorem]{Condition}

\theoremstyle{remark}

\newtheorem{ex}[theorem]{Example}
\newtheorem{rem}[theorem]{Remark}

\numberwithin{equation}{section} \numberwithin{theorem}{section}

\newcommand{\Exp}{\operatorname{E}}		
\newcommand{\Cov}{\operatorname{Cov}}

\newcommand{\Var}{\operatorname{Var}}

\newcommand{\R}{\mathbb R}
\newcommand{\N}{\mathbb N}
\newcommand{\Z}{\mathbb Z}

\newcommand{\eps}{\varepsilon}

\newcommand{\Bc}{\mathcal{B}}
\newcommand{\Fc}{\mathcal{F}}
\newcommand{\Jc}{\mathcal{J}}

\newcommand{\Ed}{\operatorname{Exp}} 		
\newcommand{\slb}{{\operatorname{sb}}} 		
\newcommand{\djb}{{\operatorname{db}}}		
\newcommand{\diff}{\, \mathrm{d}}
\newcommand{\Ei}{\operatorname{Ei}}			
\newcommand{\Beta}{\operatorname{B}}		
\newcommand{\CFG}{\operatorname{CFG}} 	
\newcommand{\MAD}{\operatorname{MAD}} 	
\newcommand{\ROOT}{\operatorname{R}} 	
\newcommand{\m}{\operatorname{m}} 			
\newcommand{\scs}{\scriptscriptstyle}			

\newcommand{\thetahat}[3]{\hat \theta^{ #3 }_{ #1 , #2 }}
\newcommand{\Thetahat}[2]{\hat \theta^{ #2 }_{ #1 }}

\newcommand{\ip}[1]{\lfloor #1 \rfloor}

\newcommand{\I}{\mathds{1}}
\newcommand{\ind}{\mathds{1}}

\newcommand{\wto}{\stackrel{d}{\longrightarrow}}

\newcommand{\Nor}{\mathcal{N}}
\newcommand{\op}{o_\Prob(1)}
\newcommand{\ho}{^}

\newcommand{\intne}{\int_{0}\ho{1}}
\newcommand{\intnu}{\int_{0}\ho{\infty}}
\newcommand{\intmn}{\int_{-\infty}\ho{0}}

\newcommand{\iu}{\I\Big(U_s > 1- \frac{x}{b_n}\Big)}
\newcommand{\iue}{\I\big(U_s^{\varepsilon} > 1- \frac{x}{b_n}\big)}
\newcommand{\Pro}{\mathbb{P}}
\newcommand{\Prob}{\mathbb{P}}

\newcommand{\wk}{\sqrt{k_n}}

\allowdisplaybreaks

\begin{document}

\begin{frontmatter}

\title{Method of moments estimators for the extremal index of a stationary time series}
\runtitle{Method of moments estimators for the extremal index}


\author{\fnms{Axel} \snm{Bücher}\ead[label=e1]{axel.buecher@hhu.de}}
\and
\author{\fnms{Tobias} \snm{Jennessen}\ead[label=e2]{tobias.jennessen@hhu.de}}

\address{Heinrich-Heine-Universität Düsseldorf \\ Universitätsstr.\, 1 \\ 40225 Düsseldorf \\
\printead{e1,e2}}

\runauthor{A. Bücher and T. Jennessen}
\affiliation{Heinrich-Heine Universität Düsseldorf}

\begin{abstract}
The extremal index $\theta$, a number in the interval $[0,1]$, is known to be a measure of primal importance for analyzing the extremes of a stationary time series.
New rank-based estimators for $\theta$ are proposed which rely on the construction of approximate samples from the exponential distribution with parameter $\theta$ that is then to be fitted via the method of moments. 
The new estimators are analyzed both theoretically as well as empirically through a large-scale simulation study. In specific scenarios, in particular for time series models with $\theta \approx 1$, they are found to be superior to recent competitors from the literature.
\end{abstract}



\tableofcontents

\end{frontmatter}

\section{Introduction} \label{sec:intro}

The statistical analysis of the extremal behavior of a stationary time series is important in many fields of application, such as in hydrology, meteorology,  finance or actuarial science \citep{BeiGoeSegTeu04}. Such an analysis typically consists of two steps: (1) assessing the tail of the marginal law and (2) assessing the serial dependence of the extremes, that is, the tendency that extreme observations occur in clusters. 
The present work is concerned with step (2). The most common and simplest mathematical object capturing  the serial dependence between the extremes  is provided by the extremal index $\theta \in [0,1]$. In a suitable asymptotic framework, the extremal index can be interpreted as the reciprocal of the expected size of a cluster of extreme observations. The underlying probabilistic theory was worked out in \cite{Lea83, LeaLinRoo83, Obr87, HsiHusLea88, LeaRoo88}.

Estimating the extremal index based on a finite stretch of observations from the time series has been extensively studied in the literature.  An early overview is provided in Section~10.3.4 in \cite{BeiGoeSegTeu04}, where the estimators are classified into three groups: estimators based on the blocks method, the runs method or the inter-exceedance time method. Respective references are \cite{Hsi93, SmiWei94, FerSeg03, Suv07, Rob09, Nor15, Fer18, FerFer18, Cai19}, among many others. The proposed estimators typically depend on two or, arguably preferable, one parameter to be chosen by the statistician. The present paper is on a class of method of moments estimators (based on the blocks method), which improves upon a recent estimator proposed in \cite{Nor15} and analyzed theoretically in \cite{BerBuc18}.

Some notations and assumptions are necessary for the motivation of the new class of estimators. Throughout the paper, $X_1,X_2,\ldots$ denotes a stationary sequence of real-valued random variables with continuous cumulative distribution function (c.d.f.) $F$.  The sequence is assumed to have an extremal index $\theta \in (0,1]$, i.e., for any $\tau >0$, there exists a sequence $u_b = u_b(\tau), b \in \N,$ such that $\lim_{b \to \infty} b \bar{F}(u_b) = \tau$ and 
\begin{align} \label{eq:ei}
\lim_{b \to \infty} \Pro(M_{1:b}\leq u_b) = e\ho{-\theta \tau},
\end{align}
where $\bar{F}=1-F$ and $M_{1:b} = \max\{X_1,\ldots,X_b\}$.
Next, define a sequence of  standard uniform random variables by $U_s = F(X_s)$ and let
\begin{align} \label{eq:ydef}
Y_{1:b} = -b\log(N_{1:b}), \qquad N_{1:b} = F(M_{1:b}) = \max\{U_1, \dots, U_b\}.
\end{align}
Since $b\bar F \{ F^{\leftarrow}(e^{-x/b}) \} = b(1-e^{-x/b}) \to x$ for $b\to\infty$, it follows from \eqref{eq:ei} that, for any $x>0$,
\begin{align} 
\Pro( Y_{1:b} \ge x)  \label{eq:yb}
&= 
\Pro( M_{1:b} \le F^{\leftarrow}(e^{-x/b}))   
\to 
e^{-\theta x},
\end{align}
where $F^{\leftarrow}(z)= \inf\{y \in \R: F(y) \geq z\}$ denotes the generalized inverse of $F$ evaluated at $z\in\R$.
In other words, for large block length $b$,  $Y_{1:b}$ approximately follows an exponential distribution with parameter~$\theta$, denoted by $\Ed(\theta)$ throughout. This inspired \cite{Nor15} and \cite{BerBuc18} to estimate $\theta$ by the maximum likelihood estimator for the exponential distribution; see Section~\ref{sec:def} below for details on how to arrive at an observable (rank-based) approximate sample from the $\Ed(\theta)$-distribution based on an observed stretch of length $n$ from the time series $(X_s)_{s\in \N}$.

The idea of transforming observations into a sample of exponentially distributed observations is actually not new within extreme value statistics: it is also, among many others, the main motivation for the Pickands estimator in multivariate extremes \citep{Pic81, GenSeg09}. More precisely, if $(X,Y)$ is a bivariate random vector from a multivariate extreme value distribution with Pickands function $A= (A(w))_{w \in [0,1]}$, then $\xi(w) = \min\{ - \log F_X(X) / (1-w), -\log F_Y(Y) / w\}$ is exponentially distributed with parameter $A(w)$. Given a sample of size $n$ from $(X,Y)$, we may replace $F_X$ and $F_Y$ by their empirical counterparts and arrive at an approximate sample of size $n$ from the $\Ed(A(w))$-distribution, to be, for instance, estimated by the maximum likelihood estimator. 

The present paper is now motivated by the following observation: while the maximum likelihood estimator is asymptotically efficient in the ideal situation of observing an i.i.d.\ sample from the exponential distribution, it was shown in \cite{GenSeg09} for rank-based estimators of the Pickands function that it is in fact more efficient to consider alternative estimators based on the method of moments, such as a rank-based version of the CFG-estimator \citep{CapFouGen97}. Given that Northrop's blocks estimator is also rank-based, the main motivation of this work is to consider CFG-type estimators for the extremal index $\theta$. Alongside, we will also investigate other moment-based estimators, including one that is closely connected to the madogram estimator in \cite{NavGuiCooDie09}. 
We will show that, depending on the true value of $\theta$, the new estimators may either exhibit a smaller or a larger asymptotic variance than Northrop's maximum likelihood estimator. In particular, we will show that the CFG-type estimator's variance is substantially smaller for  $\theta$ close to one, i.e., for time series with little clustering of extremes.

The remaining parts of this paper are organized as follows: in Section~\ref{sec:def}, we collect some results about certain useful moments of the exponential distribution and use those to introduce the new estimators for $\theta$. Regularity assumptions needed to prove asymptotic results are summarized and discussed in Section~\ref{sec:cond}. The paper's main results are then presented in Section~\ref{sec:asy}, alongside with a discussion of certain aspects of the derived asymptotic variance formulas. Section~\ref{sec:armax} is about a particular time series model, for which we show that all regularity conditions imposed in Section~\ref{sec:cond} are met. The finite-sample performance of the new estimators is investigated in a Monte-Carlo simulation study in Section~\ref{sec:sims}. Finally, all proofs are postponed to Section~\ref{sec:pmaster}.

\section{Definition of estimators} 
\label{sec:def}

Recall the definition of $Y_{1:b}$ in \eqref{eq:ydef}, where $b\in\N$. Similarly, let
\[
Z_{1:b}=b(1-N_{1:b}), \qquad N_{1:b} = F(M_{1:b}) = \max\{U_1, \dots, U_b\},
\]
and note that, as $b\to\infty$ and for any $x>0$,
\begin{align} \label{eq:zb}
\Pro( Z_{1:b} \ge x)  
= 
\Pro( M_{1:b} \le F^{\leftarrow}(1-x/b))   
\to 
 e^{-\theta x}
\end{align}
by similar arguments as for $Y_{1:b}$. The convergence relations in \eqref{eq:yb} and \eqref{eq:zb} serve as a basis for the method of moments estimators defined below. 

Subsequently, let $X_1, \dots, X_n$ denote a finite stretch of observations from the stationary sequence $(X_s)_{s \ge1}$.
Within Section~\ref{subsec:djbsample} and \ref{subsec:slbsample}, we start by using  \eqref{eq:yb} and \eqref{eq:zb} to derive some observable, approximate samples from the $\Ed(\theta)$-distribution. In Section~\ref{subsec:mom}, we collect some moment equations for the exponential distribution, which will then be used to motivate new estimators for the extremal index in Section~\ref{subsec:estimators}.

\subsection{Two approximate $\Ed(\theta)$-samples based on disjoint blocks maxima}
\label{subsec:djbsample}

Divide the sample $X_1, \dots, X_n$ into $k_n$ successive blocks of size $b_n$, and for simplicity assume that $n=b_nk_n$ (otherwise, the last block of less than $b_n$ observations should be deleted). For $i=1,\ldots,k_n$, let 
\[
M_{ni} = \max \{ X_{(i-1)b_n+1},\ldots,X_{ib_n} \}  
\] 
denote the maximum of the $X_s$ in the $i$th block of observations and let 
\[
Y_{ni} = -b_n \log N_{ni}, \qquad Z_{ni}=b_n (1-N_{ni}), \qquad N_{ni} = F(M_{ni}).
\] 
Due to relations \eqref{eq:yb} and \eqref{eq:zb}, if the block size $b=b_n$ is sufficiently large, the (unobservable) random variables $Y_{ni}$ and $Z_{ni}$ are approximately exponentially distributed with parameter $\theta$. Observable counterparts are obtained by replacing $F$ by  the (slightly adjusted) empirical c.d.f.\ $\hat{F}_{n}(x) = (n+1)^{-1} \sum_{s=1}^{n} \I(X_s \leq x)$, giving rise to the definitions 
\[
\hat Y_{ni} = -b_n \log \hat N_{ni}, \qquad  \hat Z_{ni}=b_n (1-\hat N_{ni}), \qquad \hat N_{ni} = \hat F_n(M_{ni}).
\] 
Both the samples $\mathcal Y^{\djb}_n = \{ \hat Y_{ni}: i=1, \dots, k_n\}$ and $\mathcal Z^{\djb}_n = \{\hat Z_{ni}: i=1, \dots, k_n\}$  will be used later to define \textit{disjoint blocks estimators for $\theta$} (note that both samples are dependent over $i$ due to the use of $\hat F_n$, which complicates the asymptotic analysis).

\subsection{Two approximate $\Ed(\theta)$-samples based on sliding blocks maxima}
\label{subsec:slbsample}

As in the previous paragraph, let $n$ denote the sample size and $b_n$ denote a block length parameter (the assumption that $k_n=n/b_n \in \N$ is not needed, no discarding is necessary). For $t=1, \dots, n-b_n+1$, let
\[
M_{nt}^{\slb} = M_{t:t+b_n-1} = \max \{ X_{t}, \dots, X_{t+b_n-1} \}  
\] 
denote the maximum of the $X_s$ in a block of length $b_n$ starting at observation $t$. Define
\begin{align*}
Y_{nt}^{\slb} &= -b_n \log N_{nt}^{\slb}, \qquad & Z_{nt}^{\slb} &=b_n (1-N_{nt}^{\slb}), \qquad & N_{nt}^{\slb} &= F(M_{nt}^{\slb}), \\
\hat Y_{nt}^{\slb} &= -b_n \log \hat N_{nt}^{\slb}, \qquad & \hat Z_{nt}^{\slb} &=b_n (1-\hat N_{nt}^{\slb}), \qquad & \hat N_{nt}^{\slb} &= \hat F_n(M_{nt}^{\slb}).
\end{align*}
By the same heuristics as before, the observable samples $\mathcal Y^{\slb}_n = \{ \hat Y_{nt}^{\slb}: t=1, \dots, n-b_n+1\}$ and $\mathcal Z^{\slb}_n = \{\hat Z_{nt}^{\slb}: t=1, \dots, n-b_n+1\}$ are approximate samples from the exponential distribution and
will be used later to define \textit{sliding blocks estimators for $\theta$} (both samples are heavily dependent over $i$ due to the use of $\hat F_n$ and the use of overlapping blocks).

\subsection{Preliminaries on the exponential distribution} 
\label{subsec:mom}

Some important moment equations, valid for a random variable $\xi$, which is $\Ed(\theta)$-distributed, are collected. 
First,
\begin{align} \tag{CFG} \label{eq:cfg}
\Exp[ \log \xi] =  - \log \theta - \gamma =: \varphi_{(\rm C)}(\theta) , 
\end{align}
where $\gamma =- \int_0^\infty \log (x) e^{-x} \diff x \approx 0.577$ denotes the Euler-Mascheroni-constant. Equation~\eqref{eq:cfg} is the basis for motivating the CFG-estimator, see \cite{CapFouGen97, GenSeg09} and the details in Section~\ref{sec:intro}. Next, note that
\begin{align} \tag{MAD} \label{eq:mad} 
\Exp[\exp(-\xi)] = \frac{\theta}{1+\theta}=: \varphi_{(\rm M)}(\theta),
\end{align} 
which serves as a basis for the madogram, see \cite{NavGuiCooDie09}. A further choice, including \eqref{eq:cfg} as a limit,  is provided by
\begin{align} \tag{ROOT} \label{eq:root}
\Exp[\xi^{1/p}] = \theta^{-1/p} \Gamma(1+1/p) =: \varphi_{{(\rm R)}, p}(\theta),
\end{align}
where $\Gamma(x) = \int_0^\infty t^{x-1} e^{-t} \diff t$ denotes the Gamma function and where $p>0$. The moment estimator in case of $p=1$ will turn out to coincide with Northrop's maximum likelihood estimator. Also note that the previous equation is equivalent to
\begin{align} \label{eq:root2}
\Exp\Big[\frac{\xi^{1/p}-1}{1/p}\Big] = \frac{\theta^{-1/p} \Gamma(1+1/p)-1 }{1/p} =: \tilde \varphi_{(\rm R),p}(\theta),
\end{align}
and taking the limits for $p\to\infty$ on both sides (interchanging the limit and the expectation on the left) exactly yields Equation~\eqref{eq:cfg}.

\subsection{Definition of the estimators}
\label{subsec:estimators}

Let $\chi_m= \{\xi_1, \dots, \xi_m\}$ denote a generic sample (not necessarily independent) from the $\Ed(\theta)$-distribution. Replacing the moments in Equations~\eqref{eq:cfg}, \eqref{eq:mad} and \eqref{eq:root} by their empirical counterparts and solving the equation for $\theta$, we obtain the following three estimators for $\theta$:
\begin{align*}
\hat \theta_{\CFG} (\chi_m) 
&= 
e^{-\gamma} \exp\Big\{-\frac{1}{m} \sum_{i=1}^{m} \log( \xi_i) \Big\},  \\
\hat \theta_{\MAD}(\chi_m) 
&= 
\frac{\frac{1}{m} \sum_{i=1}\ho{m} \exp(-\xi_i)}{1- \frac{1}{m} \sum_{i=1}\ho{m} \exp(-\xi_i)} , \\
\hat \theta_{{\ROOT},p}(\chi_m)
&= 
\Gamma(1+1/p)\ho{p} \Big( \frac{1}{m} \sum_{i=1}\ho{m} \xi_i\ho{1/p} \Big)\ho{-p},
\end{align*}
where $p>0$. It may be verified that 
$
\lim_{p\to\infty} \hat \theta_{{\ROOT},p}(\chi_m) = \hat \theta_{\CFG} (\chi_m),
$
see also \eqref{eq:root2} for another relationship between the two estimators.
Next, replacing $\chi_m$ by any of the four samples $\mathcal Y^{\djb}_n, \mathcal Z^{\djb}_n, \mathcal Y^{\slb}_n$ or $\mathcal Z^{\slb}_n$ defined in Sections~\ref{subsec:djbsample} and \ref{subsec:slbsample}, we finally arrive at 12 method of moments estimators for $\theta$. We use the suggestive notations \[
\thetahat{\djb}{\CFG}{y_n} = \hat \theta_{\CFG}(\mathcal Y^{\djb}_n),  \qquad
\thetahat{\slb}{\MAD}{z_n}  = \hat \theta_{\MAD}(\mathcal Z^{\slb}_n)
\]
to, e.g., denote the disjoint blocks CFG-estimator based on the $\hat Y_{ni}$ and the sliding blocks madogram-estimator based on the $\hat Z_{ni}$, respectively.  Note that the four estimators of the form 
$\thetahat{\m}{\ROOT, 1}{y_n}, \thetahat{\m}{\ROOT, 1}{z_n}, \m \in \{\djb, \slb\},$  are the (pseudo) maximum likelihood (PML) estimators considered in \cite{BerBuc18}.

\section{Mathematical preliminaries}  
\label{sec:cond}

Further mathematical details are necessary before we can state asymptotic results about the estimators defined in the previous section. The asymptotic framework and the conditions are mostly similar as in Section~2 in \cite{BerBuc18}, but will be repeated here for the sake of completeness.

The serial dependence of the time series $(X_s)_{s \in \N}$ will be controlled via mixing coefficients. For two sigma-fields $\Fc_1, \Fc_2$ on a probability space $(\Omega, \Fc, \Prob)$, let
\begin{align*}
\alpha({\Fc}_1,{\Fc}_2)
&=
\sup_{A \in {\Fc}_1, B\in {\Fc}_2} |\Prob(A\cap B)-\Prob(A)\Prob(B)| .
\end{align*}
In time series extremes, one usually imposes assumptions on the decay of the mixing coefficients between sigma-fields generated by $\{X_s \ind(X_s > F^\leftarrow (1-\eps_n)): s\le  \ell\}$ and  $\{X_s \ind(X_s > F^\leftarrow(1-\eps_n)): s \ge \ell+k \}$, where $\eps_n \to 0$ is some sequence reflecting the fact that only the dependence in the tail needs to be restricted (see, e.g., \citealp{Roo09}). As in \cite{BerBuc18}, we need a slightly stronger condition, that also controls the dependence between the smallest of all block maxima. More precisely, for $-\infty \le p < q \le \infty$ and $\eps\in(0,1]$, let $\Bc_{p:q}^\eps$  denote the sigma algebra generated by $U_s^\eps:=U_s \ind(U_s > 1- \eps)$ with $s\in \{p, \dots, q\}$ and define, for $\ell\ge 1$,
\[
\alpha_{\eps}(\ell) = \sup_{k \in \N} \alpha(\Bc_{1:k}^\eps, \Bc_{k+\ell:\infty}^\eps).
\] 
In Condition~\ref{cond:BerBuc}(iii) below, we will impose a condition on the decay of the mixing coefficients for small values of $\eps$. Note that the coefficients are bounded by the standard alpha-mixing coefficients of the sequence $U_s$, which can be retrieved for $\eps=1$.

The extremes of a time series may be conveniently described by the point process of normalized exceedances. The latter is defined, 
for a Borel set $A\subset E:= (0,1]$ and a number $x\in[0,\infty)$, by
\[
N_n^{(x)}(A) = \sum_{s=1}^n \ind(s/n \in A, U_s > 1-x/n).
\]
Note that $N_{n}^{(x)}(E)=0$ iff $N_{1:n} \le 1-x/n$; the probability of that event converging to $e^{-\theta x}$ under the assumption of the existence of the extremal index~$\theta$.

Fix $m\ge 1$ and $x_1> \dots > x_m>0$. For $1\le p < q \le n$, let
$\Fc_{p:q,n}^{(x_1, \dots, x_m)}$ denote the sigma-algebra generated by the events $\{U_i > 1- x_j/n\}$ for $p\le i \le q$ and  $1 \le j \le m$. For $1\le \ell \le n$, define
\begin{multline*}
\alpha_{n,\ell}(x_1, \dots, x_m)=\sup\{ |\Prob(A\cap B) - \Prob(A) \Prob(B)| : \\
A \in \Fc_{1:s,n}^{(x_1, \dots, x_m)}, B \in \Fc_{s+\ell:n,n}^{(x_1, \dots, x_m)}, 1 \le s \le n-\ell\}.
\end{multline*}
The condition $\Delta_n(\{u_n(x_j)\}_{1\le j \le m})$ is said to hold if there exists a sequence $(\ell_n)_n$ with $\ell_n=o(n)$ such that $\alpha_{n,\ell_n}(x_1, \dots, x_m) =o(1)$ as $n\to\infty$. 
A sequence $(q_n)_n$ with $q_n=o(n)$ is said to be $\Delta_n(\{u_n(x_j)\}_{1\le j \le m})$-separating if there exists a sequence $(\ell_n)_n$ with $\ell_n=o(q_n)$ such that $nq_n^{-1} \alpha_{n,\ell_n}(x_1, \dots,x_m) =o(1)$ as $n\to\infty$. 
If $\Delta_n(\{u_n(x_j)\}_{1\le j \le m})$ is met, then such a sequence always exists, simply take $q_n=\ip{\max\{ n\alpha_{n,\ell_n}^{\scs 1/2}, (n\ell_n)^{\scs 1/2}\}}.$ 
 
 By Theorems 4.1 and 4.2 in \cite{HsiHusLea88},
if the extremal index exists and the $\Delta(u_n(x))$-condition is met ($m=1$), then a necessary and sufficient condition for weak convergence of $N_n^{\scriptscriptstyle(x)}$ is convergence of the conditional distribution of $N_{n}^{\scriptscriptstyle(x)}(B_n)$ with $B_n=(0,q_n/n]$ given that there is at least one exceedance of $1-x/n$ in $\{1, \dots, q_n\}$ to a probability distribution $\pi$ on $\N$, that is,
\[
\lim_{n\to\infty} \Prob(N_n^{(x)} (B_n) = j \mid N_{n}^{(x)}(B_n)>0) = \pi(j)  \qquad \forall \, j\ge 1,
\]
where $q_n$ is some $\Delta(u_n(x))$-separating sequence. Moreover, in that case, the convergence in the last display holds for any $\Delta(u_n(x))$-separating sequence $q_n$, and the weak limit of $N_n^{\scriptscriptstyle(x)}$ is a compound poisson process $\mathrm{CP}(\theta x, \pi)$. If the $\Delta(u_n(x))$-condition holds for any $x>0$, then $\pi$ does not depend on $x$ (\citealp{HsiHusLea88}, Theorem~5.1). 
 
A multivariate version of the latter results is stated in \cite{Per94}, see also the summary in \cite{Rob09}, page 278, and the thesis \cite{Hsi84}. Suppose that the extremal index exists and that  the $\Delta(u_n(x_1), u_n(x_2))$-condition is met for any $x_1\ge x_2\ge0, x_1 \ne0$. Moreover, assume that there exists a family of probability measures $\{\pi_2^{\scs(\sigma)}: \sigma\in [0,1]\}$ on $\Jc = \{(i,j): i \ge j \ge 0, i \ge 1\}$, such that, for all $(i,j) \in \Jc$,
\[
\lim_{n\to\infty} \Prob(N_n^{(x_1)} (B_n) = i, N_n^{(x_2)} (B_n) = j \mid N_{n}^{(x_1)}(B_n) >0) = \pi_2^{(x_2/x_1)}(i,j),
\]
where $q_n$ is some $\Delta(u_n(x_1), u_n(x_2))$-separating sequence.  In that case, the two-level point process $\bm N_n^{\scriptscriptstyle(x_1,x_2)}=(N_n^{\scriptscriptstyle(x_1)}, N_{n}^{\scriptscriptstyle(x_2)})$ converges in distribution to a point process with characterizing Laplace transform explicitly stated in \cite{Rob09} on top of page 278. Note that
\[
\pi_2^{(1)}(i,j)=\pi(i) \ind(i=j), \qquad \pi_2^{(0)}(i,j) = \pi(i) \ind(j=0).
\]

Finally, we will need the tail empirical pocess
\begin{align} \label{eq:tep}
	e_n(x)= \frac{1}{\sqrt{k_n}} \sum_{s=1}^{n} 
	\left\{ \I \bigg(U_s > 1-\frac{x}{b_n}\bigg) -\frac{x}{b_n} \right\} , \qquad x \ge 0,
\end{align}
where $U_s=F(X_s)$, see, e.g., \cite{Dre00, Roo09}.

The following set of conditions will be imposed to establish asymptotic normality of the estimators.

	\begin{cond} \label{cond:BerBuc}~
	
	\begin{enumerate}
	 \item[(i)] The stationary time series $(X_s)_{s\in \N}$ has an extremal index $\theta \in (0,1]$ and the above assumptions guaranteeing convergence of the one- and two-level point process of exceedances are satisfied.
	 \item[(ii)] There exists  $\delta >0$ such that, for any $m >0$, there exists a constant $\tilde{C}_m$ such that
	 \[ \ \ \ \ \ \ \ \ \Exp\big[|N_n\ho{(x_1)}(E)-N_n\ho{(x_2)}(E)|\ho{2+\delta}\big] \leq \tilde{C}_m(x_2-x_1) \ \textrm{ for all } 0 \leq x_1 
	 \leq x_2 \leq m,  n \in \N. \]
	 \item[(iii)] There exist constants $c_2 \in (0,1)$ and $C_2 >0$ such that \[ \alpha_{c_2}(m) \leq C_2m\ho{-\eta} \]
	 for some $\eta \geq 3(2+\delta)/(\delta - \mu)>3,$ where $0<\mu<\min(\delta,1/2)$ and $\delta >0$ is from Condition (ii).
	 The block size $b_n$ converges to infinity and satisfies 
	 \[ 
	 k_n = o(b_n\ho{2}), \ \ n \to \infty. 
	 \] 
	 Further, there exists a sequence 
	 $\ell_n \to \infty$ with $\ell_n =o(b_n\ho{2/(2+\delta)})$ and $k_n\alpha_{c_2}(\ell_n) =o(1)$ as $n \to\infty$.
	 \item[(iv)] There exist constants $c_1 \in (0,1)$ and $C_1 >0$ such that, for any $y \in (0,c_1)$ and $n \in \N$,
	 \[ \Var \left\{ \sum_{s=1}\ho{n} \I(U_s >1-y) \right\} \leq C_1 (ny+n\ho{2}y\ho{2}). \]
	 \item[(v)] For any $c \in (0,1)$, one has 
	 \[ \lim_{n \to \infty} \Pro \Big( \min_{i=1,\ldots,2k_n} N_{ni}' \leq c \Big) =0, \]
	 where $N_{ni}'=\max \{ U_s, \ s \in [(i-1)b_n/2+1,\ldots,ib_n/2]\}$ for $i=1,\ldots,2k_n$.
	 \item[(vi)] For any $x>0$, 
	 \[ \lim_{m \to \infty} \limsup_{n \to \infty} \Pro \left( N_{m:b_n} > 1- \frac{x}{n} \,\Big|\, U_1 \geq 1-\frac{x}{n} \right) =0. \]
	\end{enumerate}
	\end{cond}

	\begin{cond}[Integrability] \label{cond2}~
	\begin{enumerate}
	 \item[(i)] With $\delta >0$ from Condition \ref{cond:BerBuc}(ii), one has \[ \limsup_{n \to \infty} \Exp \big[|\log(Z_{1:n})|\ho{2+\delta}\big] < \infty. \]
	 \item[(ii)] Fix $p>0$. With $\delta >0$ from Condition \ref{cond:BerBuc}(ii), one has 
		 \[ 
		 \limsup_{n \to \infty} \Exp \big[Z_{1:n}\ho{(2+\delta)/p}\big] < \infty. 
		 \]
	\end{enumerate}
	\end{cond}
	
	\begin{cond}[Bias Condition] \label{cond3} Recall $\varphi_{(\rm C)}, \varphi_{(\rm M)}$ and $\varphi_{(\rm R),p}$ defined in \eqref{eq:cfg}, \eqref{eq:mad} and \eqref{eq:root}, respectively. 
	\begin{enumerate}
	 \item[(i)] As $ n \to \infty$, $ \Exp[\log(Z_{1:b_n})] = \varphi_{(\rm C)}(\theta) +o\big( k_n\ho{-1/2} \big)$.	 
	 \item[(ii)] As $ n \to \infty$, $ \Exp[\exp(-Z_{1:b_n})] = \varphi_{(\rm M)}(\theta) +o\big( k_n\ho{-1/2} \big)$.
	 \item[(iii)] Fix $p>0$. As $ n \to \infty$, $ \Exp\big[Z_{1:b_n}\ho{1/p}\big] = \varphi_{(\rm R),p}(\theta)+o\big( k_n\ho{-1/2} \big)$.
	 	\end{enumerate}
	\end{cond}

         \begin{cond}[Technical Condition for the CFG-type estimator] \label{cond_CFG}~
         \begin{enumerate}[label= \upshape (\roman*), ref= \thecond (\roman*)]
         \item \label{cond:bnkn2}
         For some $q>1/2$, we have $b_n = O(k_n^q)$ as $n\to\infty$.
         \item \label{cond_gewKonv}
         For some  $\tau \in (0,1/2)$, we have, as $n\to\infty$,           
         \[ 
         \left\{ \frac{e_n(x)}{x\ho{\tau}} \right\}_{x \in [0,1]} \wto \left\{ \frac{e(x)}{x\ho{\tau}} 
                                   \right\}_{x \in [0,1]}
                                  \quad \textrm{ in } D([0,1]),
                                \]
                                the c\`agl\`ad space of functions on $[0,1]$,
                                                                where $e_n$ 
                               denotes the tail empirical process defined in \eqref{eq:tep} and where $e$ is a centered Gaussian process with continuous sample paths and covariance as given in Lemma \ref{lem:cfg-fidis}.
        \item \label{cond_max} 
	 For any $c>0$,  we have, as $n \to \infty$,
         \[ 
         \max_{Z_{ni} \geq c} \left| \frac{e_n(Z_{ni})}{Z_{ni}\sqrt{k_n}} \right|
         = o_\Prob(1).
         \]                
         \item \label{cond_ewerte}
       For any $c>0$, there exists $\mu=\mu_c \in (1/2,1/\{2(1-\tau)\})$ with $\tau$ 
         from (ii) such that,  as $n\to\infty$, 
         \[ 
         \Prob(  Z_{n1} < c k_n\ho{-\mu} )
         - \Prob(  \xi < c k_n\ho{-\mu} )= o\big(\log(n)\ho{-1} k_n^{-1/2}\big), \ \textrm{ where } \xi \sim \Ed(\theta). \]  
         \end{enumerate}                      
         \end{cond}

	The items of Condition \ref{cond:BerBuc} are the same as Condition 2.1(i)-(v) and (2.2) in \cite{BerBuc18} and are discussed in great detail in that reference. 
	Condition~\ref{cond2} is needed for uniform integrability of the sequences $Z_{n1}^{\scs 2/p}$ and $\log^2 Z_{n1}$, respectively. It implies 
	\[
	 \lim_{n\to\infty}\Var(Z_{n1}^{1/p}) = \Var(\xi^{1/p}), \qquad  \lim_{n\to\infty}\Var(\log Z_{n1}) = \Var(\log \xi),
	 \] 
	 respectively, where $\xi$ denotes an exponentially distributed random variable with parameter $\theta$. 
	Condition~\ref{cond3} is a bias condition requiring the approximation of the first moment of $f(Z_{n1})$  by $\Exp[f(\xi)]$  to be sufficiently accurate, where $f(x) \in \{x^{1/p}, \exp(-x), \log x\}$.
	
	Condition~\ref{cond_CFG} is a technical condition which is only needed for deriving the asymptotics of the CFG-estimator. The Condition~\ref{cond:bnkn2} requires $b$ to be not too large.
 Sufficient conditions for Condition \ref{cond_gewKonv} in terms of beta mixing coefficients can be found in \cite{Dre00}. 	A sufficient condition for Condition~\ref{cond_max} is for instance strong mixing with polynomial rate $\alpha_1(n) = O(n\ho{-(1+\surd 2)-\varepsilon}), n\to\infty,$ for some $\varepsilon>0$, together with Condition~\ref{cond:bnkn2} being met with $q<1/(\sqrt2-1) \approx 2.41$. Indeed, 
  for any $x \geq c$ and $\eta>0$, one can write
         \begin{align}
          \frac{e_n(x)}{x} &= \frac{1}{\sqrt{k_n}} \sum_{s=1}\ho{n} \left\{ \iu - \frac{x}{b_n} \right\} \frac{1}{x} 
          = - b_n\ho{1/2-\eta} \ \mathbb{U}_{n, \eta}\Big( 1- \frac{x}{b_n} \Big) \frac{1}{x\ho{1-\eta}}, \nonumber
         \end{align}
         where \[ \mathbb{U}_{n, \eta}(u) 
         = \frac{\frac{1}{\sqrt{n}} \sum_{s=1}\ho{n} \left\{ \I(U_s \leq u)-u \right\}}{(1-u)\ho{\eta}} \ind_{(0,1)}(u).
         \]
         By Theorem 2.2 in \cite{ShaYu96}, we have $\sup_{x\ge 0} | \mathbb{U}_{n, \eta}(1-x/b_n) |  =O_\Prob(1)$ for all $\eta \leq 1-2\ho{-1/2} \approx 0.29$. Hence,  by Condition~\ref{cond:bnkn2},
         \[
          \max_{Z_{ni} \geq c} \left| \frac{e_n(Z_{ni})}{Z_{ni}\sqrt{k_n}} \right| = O_\Prob\Big( \frac{b_n^{1/2-\eta}}{\sqrt{k_n}} \Big) = 
          O_\Prob\Big( k_n^{q(1/2-\eta) - 1/2} \Big).
         \]
        The expression on the right-hand side is $o_\Prob(1)$ if we choose $\eta\in(1/2-1/\{2q\}, 1-2^{-1/2}]$; note that the latter interval is non-empty since $q<1/(\sqrt 2-1)$.
        Finally, Condition~\ref{cond_ewerte} is another technical condition requiring the approximation of the law of $Z_{n1}$ by the exponential distribution to be sufficiently accurate in the lower tail.

\section{Asymptotic Results}
\label{sec:asy}

We present asymptotic results on all estimators defined in Section~\ref{sec:def}. For simplicity, all results are stated and proved for the $\hat Z_{ni}$-versions only. As in Theorem 3.1 in \cite{BerBuc18}, it may be verified that the respective versions based on $\hat Y_{ni}$ show the same asymptotic behavior as the $\hat Z_{ni}$-versions.  
Throughout, for $z\in(0,1)$, let $(\xi_1^{(z)},\xi_2^{(z)}) \sim \pi_2^{(z)}$.

\begin{thm} 
\label{theo:cfg}
Under Condition \ref{cond:BerBuc}, \ref{cond2}(i), \ref{cond3}(i) and \ref{cond_CFG},
we have 
\[
\sqrt{k_n} ( \thetahat{\m}{\CFG}{z_n}-\theta) 
\wto 
\Nor(0, \sigma^2_{\rm m, \rm C})
\] 
 for ${\rm m} \in \{\djb, \slb\}$ and as $n\to\infty$, where
\begin{align*}
\sigma^2_{\djb, \rm C} 
& =
2 \theta^3 \int_{0}^{1} \frac{\theta  \Exp[ \xi_1^{(z)} \xi_2^{(z)}] -  \Exp[ \xi_1^{(z)} \I(\xi_2^{(z)}>0)]}{z(1+z)} \diff z 
 + \big\{ \pi^2/6 - 2\log(2) \} \theta^2, \\
\sigma_{\slb, \rm C}\ho{2} 
&=  
\sigma^2_{\djb, \rm C}  - \{ \pi\ho{2}/{6} - 8\log(2) + 4\} \theta^2. 
\end{align*}
\end{thm}

\begin{thm} 
\label{theo:mad}
Under Condition \ref{cond:BerBuc} and \ref{cond3}(ii), we have 
\[
\sqrt{k_n} (\thetahat{\m}{\MAD}{z_n}-\theta) 
\wto 
\Nor(0, \sigma^2_{m, \rm M})
\] 
 for ${\rm m} \in \{\djb, \slb\}$ and as $n\to\infty$, where
\begin{align*}
\sigma_{\djb, \rm M}\ho{2} 
&= 4 \theta^2(1+\theta) \intne \frac{ \theta \Exp[\xi_1\ho{(z)} \xi_2\ho{(z)}] - \Exp[\xi_1\ho{(z)} \I(\xi_2\ho{(z)}>0)]}{(1+z)\ho{3}}
 \ \mathrm{d}z 
  + \frac{\theta^2(1+\theta)}{2(2+\theta)} \\ 
\sigma_{\slb, \rm M}\ho{2} &= \sigma_{\djb, \rm M}\ho{2} 
- \frac{3 \theta^2 + 4 \theta - 4(1+\theta)(2+\theta) \log\{2(1+\theta)/(2+\theta)\} }{\theta(2+\theta)(1+\theta)^2}.
\end{align*}
\end{thm}

\begin{thm} 
\label{theo:root}
Fix $p>0$. Under Condition \ref{cond:BerBuc}, \ref{cond2}(ii) and \ref{cond3}(iii),  
\[
\sqrt{k_n} (\thetahat{\m}{\ROOT, p}{z_n}-\theta) 
\wto 
\Nor(0, \sigma^2_{{\rm m}, p})
\] 
 for ${\rm m} \in \{\djb, \slb\}$ and as $n\to\infty$, where
\begin{align*}
\sigma_{\djb, p}\ho{2}
&=
\frac{4p\theta^3}{\Beta(1/p, 1/p)}   \intne
\frac{\theta \Exp[\xi_1\ho{(z)} \xi_2\ho{(z)}] + \Exp[\xi_1\ho{(z)} \I(\xi_2\ho{(z)}=0) ] z\ho{\frac{1}{p}-1}}{(1+z)\ho{ 1+ \frac{2}{p} }} \ \mathrm{d}z \nonumber\\
&\hspace{6cm} + \Big\{ \frac{2p^3}{\Beta(1/p, 1/p)} - p^2- 2p \Big\} \theta^2, \nonumber\\
\sigma_{\slb, p}\ho{2}
&=
\sigma_{\djb, p}\ho{2}  -  \bigg[ p\ho{2} + \frac{2p\ho{3}}{\Beta(1/p,1/p)} - \frac{4p}{\Gamma(1/p)\ho{2}} \intnu (1-e\ho{-z}) z\ho{1/p-2} \Gamma(1/p,z) \ \mathrm{d}z \bigg]  \theta^2 , \nonumber
\end{align*}
 where $\Beta(x,y) = \int_0^1 t^{x-1} (1-t)^{y-1} \diff t$ is the Beta-function and $\Gamma(x,z)=\int_x^\infty t^{z-1} e^{-t} \diff t$ is the incomplete gamma function.
\end{thm}

It is worthwhile to mention that the imposed conditions in each theorem are exactly the same for the disjoint and the sliding blocks version. Furthermore, apart from the different bias conditions, the conditions regarding $k_n$ are exactly the same in Theorem~\ref{theo:mad} and \ref{theo:root}, and slightly stronger for Theorem~\ref{theo:cfg} in that the additional technical Condition~\ref{cond_CFG} is imposed.

The proofs are provided in Section~\ref{sec:pmaster} and  bear some similarities with  the one of Theorem~3.2 in \cite{BerBuc18}. In particular, they rely on the delta method, Wichura's theorem and empirical process theory to adequately handle the asymptotic contribution of the rank transformation. The most sophisticated proof is the one of Theorem~\ref{theo:cfg}, which is essentially due to the fact that $\Exp[\log \xi]=\int_0^\infty \log(t) \theta e^{-\theta t} \diff t$ is an improper integral both at zero and at infinity (see also \cite{GenSeg09} for similar technical difficulties with the CFG-estimator for the Pickands dependence function in multivariate extremes).

 \begin{figure}
  \begin{center}
   \includegraphics[width = 0.75 \textwidth]{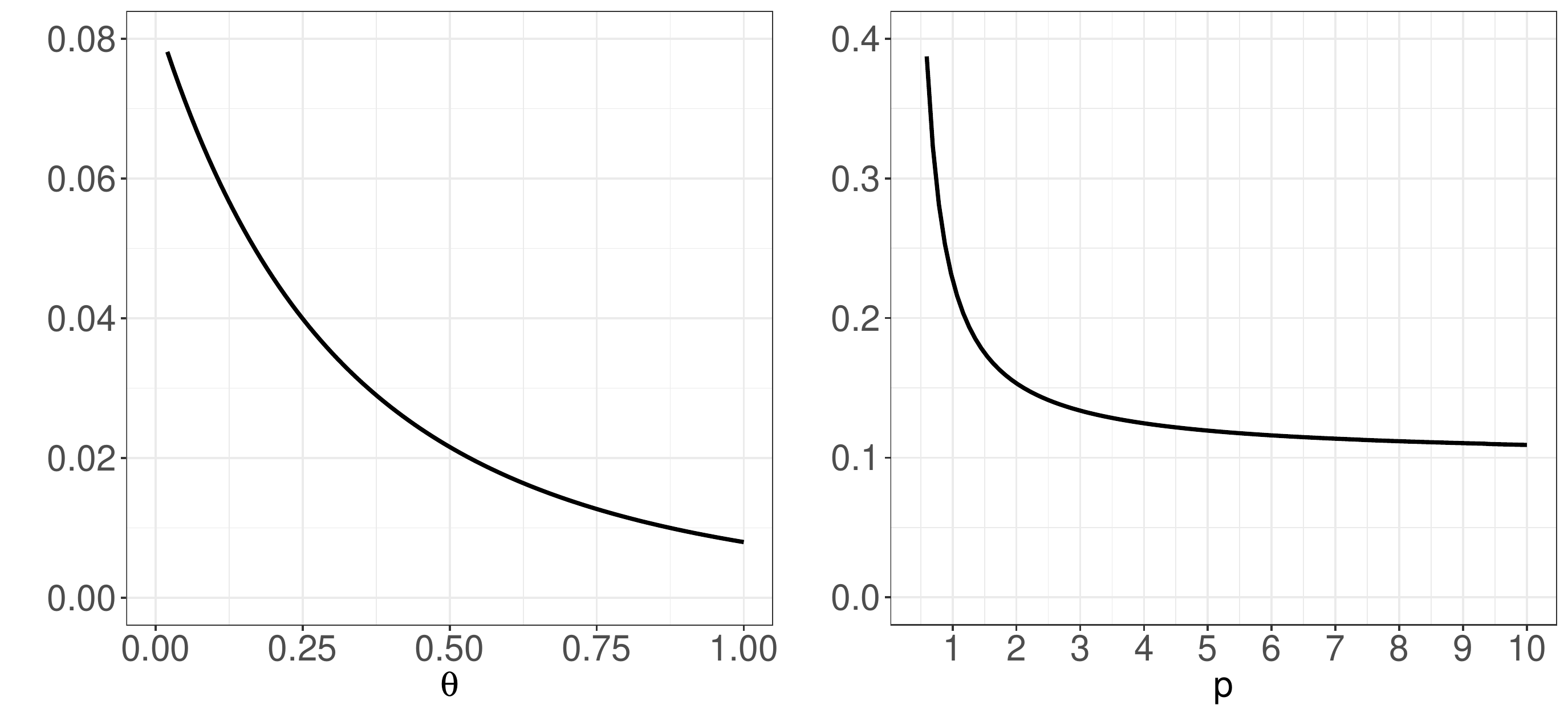}
   \caption{Graph of the functions $\theta \mapsto (\sigma_{\djb, \rm M}\ho{2} -  \sigma_{\slb, \rm M}\ho{2})/\theta^2$ (left) and $p \mapsto (\sigma_{\djb, \rm p}\ho{2} -  \sigma_{\slb, \rm p}\ho{2})/\theta^2$ (right).}
   \label{fig:Mado}
  \end{center}

 \end{figure}

It is worth to mention that the difference 
\[
\mathrm{AsyVar}(\sqrt{k_n} \thetahat{\djb}{\CFG}{z_n} / \theta) - \mathrm{AsyVar}(\sqrt{k_n}\thetahat{\slb}{\CFG}{z_n} / \theta)
:= (\sigma^2_{\djb, \rm C} - \sigma_{\slb, \rm C}\ho{2})/\theta^2 \approx 0.0977
\]
is a universal constant independent of any properties  of the observed time series.
The same holds true for the Root-estimator with a constant depending in a complicated way on the parameter $p$ (the graph of $p \mapsto (\sigma_{\djb, p}\ho{2} -  \sigma_{\slb, p}\ho{2})/\theta^2$ is depicted in Figure~\ref{fig:Mado}, with a value of approximately 0.2274 for the PML-estimator). For the Madogram-estimator, this difference depends on $\theta$ (see Figure~\ref{fig:Mado} for the graph of $\theta \mapsto (\sigma^2_{\djb, \rm M} - \sigma_{\slb, \rm M}\ho{2})/\theta^2$); it is non-negative and decreasing with value  $1/12\approx 0.083$ for $\theta\to0$ and approximately $0.0079$ for $\theta=1$. In that regard, the use of sliding blocks over disjoint blocks is least beneficial for the  Madogram-estimator.

\begin{ex} \label{ex:indep}
In the case that the time series is serially independent, the cluster size distributions are given by $\pi(i)=\I(i=1)$ and $\pi_2^{\scs (z)}(i,j)=(1-z)\I(i=1,j=0)+ z \I(i=1,j=1)$, which implies 
\[ 
\theta =1, \quad E[\xi_1^{(z)} \xi_2^{(z)}] = z 
\quad \textrm{and} \quad E[\xi_1^{(z)}\I(\xi_2^{(z)}=0)] = 1 - z. 
\]
It can be seen that these formulas hold true whenever $\theta = 1$. Consequently, the limiting variances in Theorem~\ref{theo:cfg} and \ref{theo:mad} are equal to
\begin{align*}
\sigma^2_{\djb, \rm C} &= \frac{\pi^2}{6} - 2 \log(2) \approx 0.2586,  
& \sigma^2_{\slb, \rm C} &= 6\log(2) -4 \approx 0.1588, \\
\sigma_{\djb, \rm M}\ho{2} &= 1/3, 
& \sigma_{\slb, \rm M}\ho{2} &\approx 0.32536. 
\end{align*}
It is remarkable that the asymptotic variances are substantially smaller than those of the maximum likelihood estimator,
see Example~3.1 in \cite{BerBuc18}, which are equal to $1/2$ and $0.2726$ for the disjoint and sliding blocks version, respectively. 

The limiting variance in the case of the Root-estimator  is given by
	\begin{align*}
	\sigma\ho{2}_{\djb, p} & = \frac{2p}{\Beta(\tfrac1p, \tfrac1p)}  \left[ p^2 + 2\ho{-2/p}p \right] - p\ho{2}-p, \\
	\sigma\ho{2}_{\slb, p} & = \sigma_{\djb, p}\ho{2}  -  \bigg[ p\ho{2} + \frac{2p\ho{3}}{\Beta(1/p,1/p)} - \frac{4p}{\Gamma(1/p)\ho{2}} \intnu (1-e\ho{-z}) z\ho{1/p-2} \Gamma(1/p,z) \ \mathrm{d}z \bigg].
	\end{align*}
	Some values are 
	\begin{align*}
	\sigma\ho{2}_{\djb,1/2}  &= \frac{15}{16}, &
	 \sigma\ho{2}_{\djb,1}  &= \frac12, &
	\sigma\ho{2}_{\djb,2}  
	&\approx 0.3662, \\
	 \sigma\ho{2}_{\slb,1/2}  &= \frac{7}{16}, &
	 \sigma\ho{2}_{\slb,1}  
	 &\approx 0.2726, &
	\sigma\ho{2}_{\slb,2}  &\approx 0.212909.
	\end{align*}
	It can further be shown that $\lim_{p\to \infty}\sigma\ho{2}_{{\rm m}, p} = \sigma^2_{{\rm m}, \rm C}$ for ${\rm m}\in\{\djb, \slb\}$.
\end{ex}

	\begin{rem} \label{rem:bias_reduction}
			Instead of working with $\hat{F}_n$ in the definition of $\hat{Z}_{ni} = b_n \{ 1- \hat{F}_n(M_{ni})\}$, one may alternatively use the empirical c.d.f.\ of $(X_s)_{s \notin I_i}$ multiplied by $(n-b_n)/(n-b_n+1)$ for $I_i = \{(i-1)b_n+1,\ldots,ib_n\}$, denoted by $\hat{F}_{n,-i}$, and define $\tilde{Z}_{ni} = b_n\{ 1-\hat{F}_{n,-i}(M_{ni})\}$ and $\tilde \theta= \hat \theta(\tilde{Z}_{n1}, \dots, \tilde{Z}_{nk_n})$. This modification has been motivated as a bias reduction scheme in \cite{Nor15}. 
			Since  
			\[ 
			\tilde Z_{ni} = b_n\{1- \hat{F}_{n,-i}(M_{ni})  \} = b_n \{ 1-\hat{F}_{n}(M_{ni})\} \frac{n+1}{n-b_n+1} =  \hat Z_{ni} \frac{n+1}{n-b_n+1} ,
			\]  
			some simple calculations show that, for instance for the CFG-estimator, 	
			\[  
			e^{-\gamma} \exp\Big\{ -\frac{1}{k_n} \sum_{i=1}^{k_n} \log(\tilde{Z}_{ni}) \Big\} 
			= 
			\frac{n-b_n+1}{n+1} \thetahat{\djb}{\CFG}{z_n},
			 \]
			 showing that the modification is asymptotically negligible. It is however beneficial in finite-sample situations, whence it has been applied throughout the finite-sample situations considered in Section~\ref{sec:sims}.
			Obviously, similar adaptions can be applied to the sliding blocks version and the other moment based estimators.
		\end{rem}

	\section{Example: max-autoregressive process}
	\label{sec:armax}

 In this section, we exemplarily discuss the new estimators when applied to a max-autoregressive process, defined by the recursion
		\[ 
		X_s = \max \left\{ \alpha X_{s-1}, (1-\alpha)Z_s \right\}, \quad s \in \Z, 
		\] 
		where $\alpha \in [0,1)$ and where $(Z_s)_{s \in \Z}$ is an i.i.d.\ sequence
		of Fr\'{e}chet(1)-distributed random variables.   A stationary solution of the above recursion is
		\[ 
		X_s = \max_{j \geq 0} \ (1-\alpha) \alpha\ho{j} Z_{s-j}, 
		\] such that the stationary solution 
		is again Fr\'{e}chet(1)-distributed.  Note that a model with an arbitrary stationary c.d.f. $F$ may be obtained by considering $\tilde X_s=F^{\leftarrow}\{ \exp(-1/X_s) \}$ and that all subsequent results are also valid for $(\tilde X_s)_s$.
		
		We start by explicitly calculating the asymptotic variances of the estimators in Section~\ref{subsec:armax_var}, and then show in Section~\ref{subsec:armax_cond} that all regularity conditions from Section~\ref{sec:cond} are met.
	
	\subsection{Asymptotic variances for the ARMAX-model}
	\label{subsec:armax_var}
	
		Recall that the ARMAX-model has extremal index $\theta =1-\alpha$ and 
		that the corresponding cluster size distribution is geometric, that is, $\pi(j) = \alpha^{j-1}(1-\alpha), j \ge 1$, see, e.g., Chapter 10 in \cite{BeiGoeSegTeu04}.  From Example~6.1 in  \cite{BerBuc18},
		one further has
		\[ 
		\Exp[\xi_1^{(z)} \xi_2^{(z)}] = 
		\frac{\alpha\ho{w+1}+z + z w (1-\alpha)}{(1-\alpha)\ho{2}}, 
		\quad
		\Exp[\xi_1^{(z)}\I(\xi_2^{(z)}=0)] = \frac{1-\alpha\ho{w+1}}{1-\alpha} - z (w+1),\]
		where $w=\lfloor  \log(z)/\log(\alpha) \rfloor$ and $(\xi_1\ho{(z)},\xi_2\ho{(z)}) \sim \pi_2\ho{(z)}$. 		
		This allows to calculate the limiting variances in Theorem \ref{theo:cfg}--\ref{theo:root}   explicitly.
	For the CFG-type estimator, some tedious but straightforward calculations imply
		\begin{align}
		& \frac{\sigma_{\djb, \rm C}^2}{\theta^2} = \frac{\pi\ho{2}}{6}+2\log(2)(\alpha-1) 
		\ \ \textrm{ and } \ \ 
		\frac{\sigma_{\slb, \rm C}^2}{\theta^2} = 2\log(2) (3+\alpha)-4, \nonumber
		\end{align}
		see also Figure~\ref{fig:ARMAXroots} for a picture of the graph of these functions.
		Next, we compare these variances with the disjoint and sliding blocks variances of the PML-estimator in \cite{BerBuc18}, which are given by $\sigma^2_{\djb,1}$ and $\sigma^2_{\slb,1}$ and satisfy 
		\begin{align*}
		 \frac{\sigma^2_{\djb,1}}{\theta^2} = \frac{1}{2}  (1+\alpha) \ \ \textrm{ and } \ \ 
		\frac{\sigma^2_{\slb,1}}{\theta^2} = \frac{8 \log(2)-5+\alpha}{2},
		\end{align*}
		respectively. 
		Thus, 
		$\sigma_{\djb, \rm C}^2 \le \sigma_{\djb, 1}^2$ iff $\alpha \leq \{ 1+4\log(2) - \pi^2/3\}/\{4\log(2)-1\} \approx 0.2723$ and $\sigma_{\slb, \rm C}^2 \le \sigma_{\slb, 1}^2$ iff $\alpha \leq \{3-4\log(2)\}/\{4\log(2)-1\} \approx 0.128$.
		Further comparisons can be drawn from Figure \ref{fig:ARMAXroots}, where the asymptotic variances of $\sqrt{k_n}(\hat{\theta}_n/\theta-1)$ are additionally illustrated for the Madogram- and the Root-estimators. 
  
             \begin{figure} 
                \begin{center}
                \includegraphics[width=0.8\textwidth]{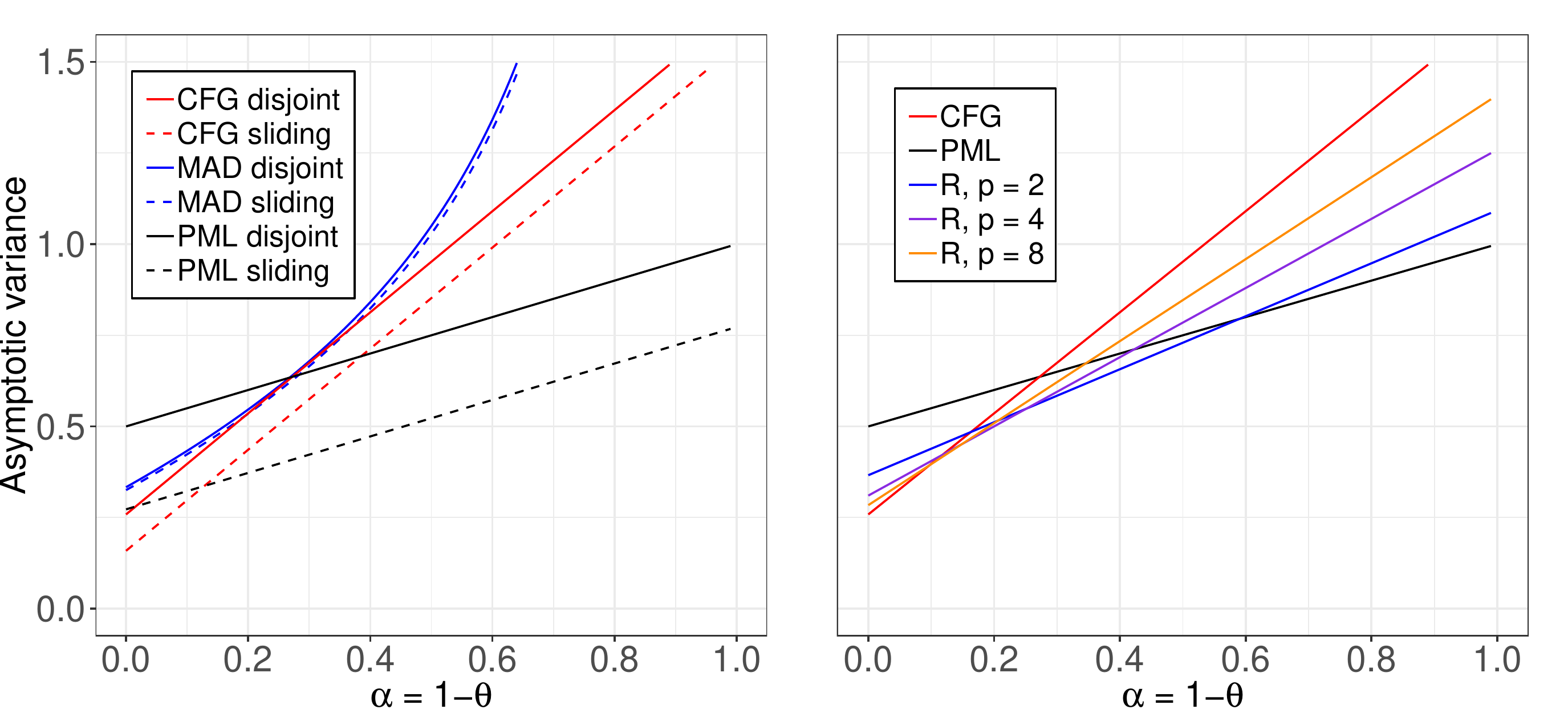}
                \caption{Asymptotic variance of $\sqrt{k_n}(\hat{\theta}_n/\theta-1)$ in the ARMAX-model. The estimators in the right figure rely on disjoint blocks.} \label{fig:ARMAXroots}
                \end{center}
        \end{figure}

	
	\subsection{Regularity Conditions for the ARMAX-model}
	\label{subsec:armax_cond}
	
	   Recall that $X_s$ is Fr\'{e}chet(1)-distributed, i.e., the stationary c.d.f. $F$ is given by
	$F(x) = \exp(-1/x), x > 0,$  with inverse $F\ho{-1}(x) = -\log(x)\ho{-1}$.
	
	The assumptions in Condition \ref{cond:BerBuc} are satisfied as shown in \cite{BerBuc18}, page\ 2322, provided $b_n$ and $k_n$ are chosen to satisfy the conditions in item (iii). Next, by induction,
		\[ 
		\Pro\big( \max_{s=1,\ldots,b} X_s \leq x \big)
		= F(x)\ho{1+\theta(b-1)},\]
		which implies that the c.d.f. of $Z_{1:b} = b \{ 1- F(M_{1:b})\}$ is given by
		\begin{equation} \label{distr_Z_1n}
		\Pro (Z_{1:b} \leq x) = 1- \Pro \Big( \max_{s=1,\ldots,b} X_s \leq F\ho{-1}(1-x/b) \Big)
		=
		\begin{cases}
		1, \ & x \geq b, \\
		1- \big( 1- \frac{x}{b} \big)\ho{1+\theta (b-1)}, \ & x \in [0,b],\\
		0, \ & b \leq 0.
		\end{cases}
		\end{equation}
		A  tedious but straightforward calculation then shows that the assumptions in Condition~\ref{cond2} and \ref{cond3} are met, provided $k_n/b_n\ho{2} = o(1)$, cf.\ Condition \ref{cond:BerBuc}(iii).  Condition~\ref{cond:bnkn2} is a condition on the choice of $b_n$, that is under the control of the statistician.
		Conditions~~\ref{cond_gewKonv} and \ref{cond_max} are consequences of mixing properties of $(X_s)_s$ as argued at the end of Section~\ref{sec:cond}. It remains to show that Condition \ref{cond_ewerte} is satisfied. By \eqref{distr_Z_1n} and with $\xi \sim \Ed(\theta)$, we have
		\begin{align}
	 \Pro(Z_{n1} < c k_n\ho{-\mu}) - \Pro(\xi< c k_n\ho{-\mu}) 
		&= \exp(-\theta c k_n\ho{-\mu})- \Big( 1- \frac{c k_n\ho{-\mu}}{b_n} \Big)\ho{1+\theta(b_n-1)} \nonumber\\
		&= o(k_n^{-1/2} (\log n)^{-1}), \qquad n\to\infty,  \nonumber
		\end{align}
for any $\mu>1/2$, where the final estimate follows from Taylor's theorem and Condition~\ref{cond:bnkn2}.

\section{Finite-sample results} 
\label{sec:sims}

A Monte-Carlo simulation study was performed to assess the finite-sample performance of the introduced estimators and to compare them with competing estimators from the literature. The data is simulated from the following four time series models that were also investigated in \cite{BerBuc18}:
\begin{itemize}
\item 
The \textbf{ARMAX-model}:
\[
X_s = \max\{ \alpha X_{s-1}, (1-\alpha) Z_s \}, \qquad s \in \Z,
\]
where $\alpha \in [0,1)$ and where $(Z_s)_s$ is an i.i.d.\ sequence of standard Fr\'echet random variables. We consider $\alpha=0,0.25,0.5,0.75$ resulting in $\theta=1,0.75,0.5,0.25$.
\item 
The \textbf{squared ARCH-model}:
\[
X_s = (2 \times 10^{-5} + \lambda X_{s-1}) Z_s^2,  \qquad s \in \Z,
\]
where $\lambda\in(0,1)$  and where $(Z_s)_s$ denotes an i.i.d.\ sequence of standard normal random variables. We consider $\lambda=0.1, 0.5, 0.9, 0.99$ for which the simulated values $\theta =0.997, 0.727, 0.460, 0.422$ were obtained, respectively; see Table 3.1 in \cite{DehResRooVri89}.
\item 
The \textbf{ARCH-model}:
\[
X_s = (2 \times 10^{-5} + \lambda X_{s-1}^2) ^{1/2}Z_s,  \qquad s \in \Z,
\]
where $\lambda\in(0,1)$  and where $(Z_s)_s$ denotes an i.i.d.\ sequence of standard normal random variables. We consider $\lambda=0.1, 0.5, 0.7, 0.99$ for which the simulated values $\theta =0.999, 0.835, 0.721, 0.571$ were obtained, respectively; see Table 3.2 in \cite{DehResRooVri89}.

\item 
The \textbf{Markovian Copula-model} (\cite{DarNguOls92}):
\[
X_s =F^\leftarrow(U_s), \quad (U_s, U_{s-1}) \sim C_\vartheta, \qquad s \in \Z.
\]
Here, $F^\leftarrow$ is the left-continuous quantile function of some arbitrary continuous c.d.f. $F$, 
$(U_s)_s$ is a stationary Markovian time series of order~1 and $C_\vartheta$ denotes the Survival Clayton Copula with parameter $\vartheta>0$. 
We consider choices $\vartheta=0.23, 0.41, 0.68, 1.06, 1.90$  such that (approximately) 
$\theta = 0.95, 0.8, 0.6, 0.4, 0.2$
 \citep{BerBuc18} and fix $F$ as the standard uniform c.d.f. (the results are independent of this choice, as the estimators are rank-based). Algorithm 2 in \cite{RemPapSou12} allows to simulate from this model. 
\end{itemize}

In each case, the sample size is fixed to $n = 2^{13}=8192$ and the block size is chosen from $b=b_n\in\{2^2,\ldots,2^{9}\}$. The performance is assessed based on $N=3000$ simulation runs each. 

\subsection{Comparison of the introduced estimators} \label{CompNewEst}
We start by comparing the finite-sample properties of the proposed sliding blocks estimators $ \thetahat{\rm{m} }{\CFG}{x}$, $\thetahat{\rm{m} }{\MAD}{x}$ and $\thetahat{\rm{m} }{\ROOT,p}{x}$ for $p \in \{ 0.5, 0.75, 1,2,4,8,16\}$ for $x \in \{z_n,y_n\}$ and for $\rm{m} \in \{\slb, \djb\}$.

As the simulation results are, to a large extent, similar among the different models and estimators, they are only partially reported, with a particular view on highlighting selected interesting qualitative features. 
We begin by a detailed investigation of the variance, the squared bias  and the mean squared error (MSE) as a function of the block size parameter $b$. 
In Figure~\ref{DisSl_fig1}, we present results for the disjoint and sliding blocks version of the CFG- and the PML-estimator in a representative ARMAX-model with $\theta=0.75$.  
Similarly as in \cite{BerBuc18} and as to be expected from the asymptotic results, the bias of the disjoint and the sliding blocks version are almost identical, while the variance is uniformly smaller for the sliding blocks version (in particular for large values of $b_n$). Since this qualitative behavior holds uniformly over all models and estimators, we omit the disjoint blocks estimator from the subsequent discussions and write $\Thetahat{\CFG}{x } = \thetahat{\slb}{\CFG}{x}$ etc.\ for simplicity.

		\begin{figure}
			\begin{center} \vspace{-.3cm}
				\includegraphics[width=0.8\textwidth]{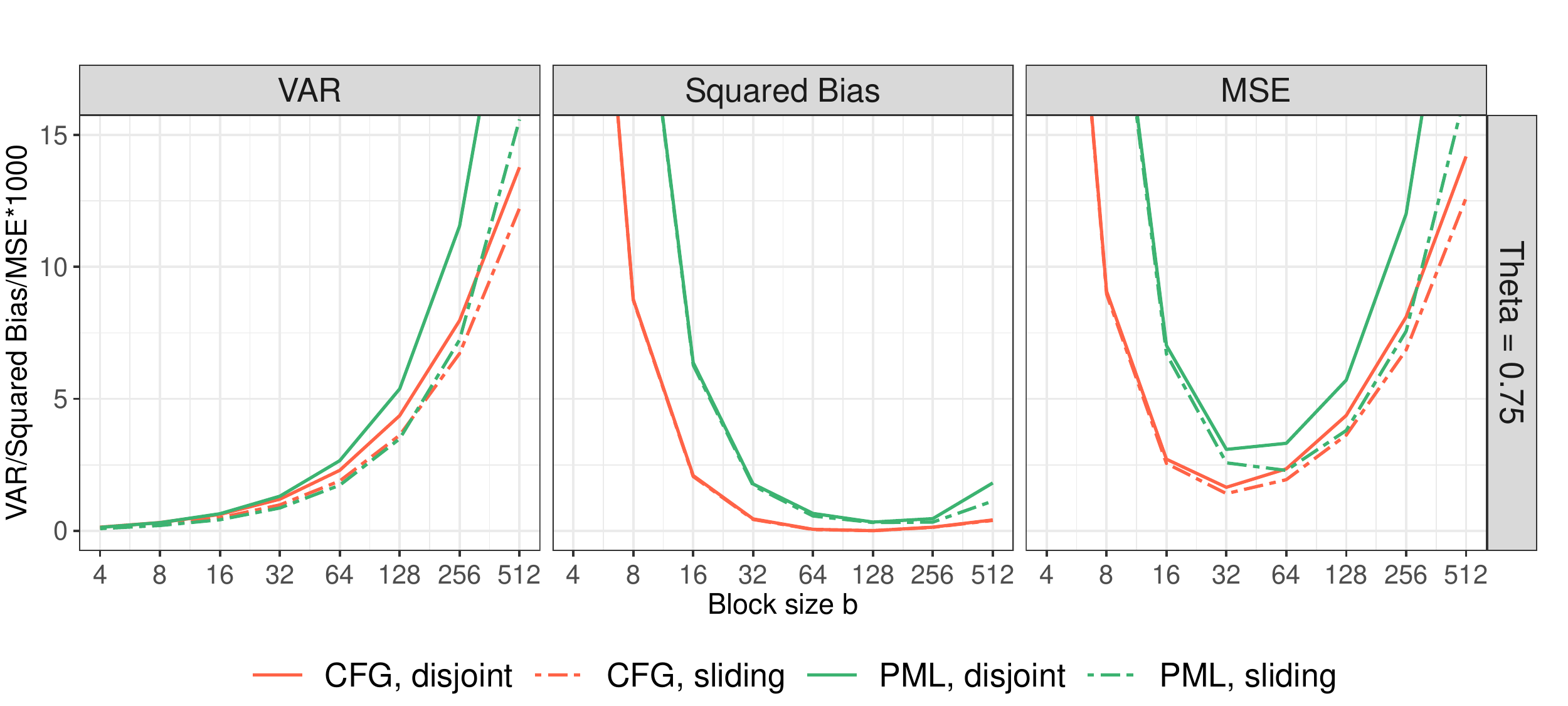}
				\caption{Comparison of variance, squared bias and MSE, multiplied by $10^3$, of the disjoint and sliding blocks CFG- and PML-estimator in the ARMAX-model. }   \vspace{-.4cm}
				\label{DisSl_fig1}
			\end{center}
		\end{figure}

Next, we compare the different moment estimators. For illustrative purposes, we begin by restricting the presentation to the $z_{n}$-versions  and the ARCH-model. The corresponding results are depicted in Figure \ref{Fig:ARCH_all} (for the CFG-, the Madogram- and three selected Root-estimators). In general, as to be expected from the underlying theory, the variance curves are increasing in $b$, while the squared bias curves are (mostly) decreasing in $b$, resulting in a typical U-shape for the MSE curves.
 The hierarchy of the estimators with regard to the considered performance measures is similar among the considered values of $\theta$. In terms of the MSE, up to an intermediate block size, the CFG- and Madogram-estimator are superior to the other estimators (especially to the PML-estimator), while for large block sizes the Madogram-estimator has a relatively high MSE, but the CFG-estimator partly remains superior. The Root-estimators are, as expected, ordered in $p$ and located between the PML- and CFG-estimator.

\begin{figure} [ht!]
 \begin{center}
 \vspace{-.3cm}
  \includegraphics[width=0.95\textwidth]{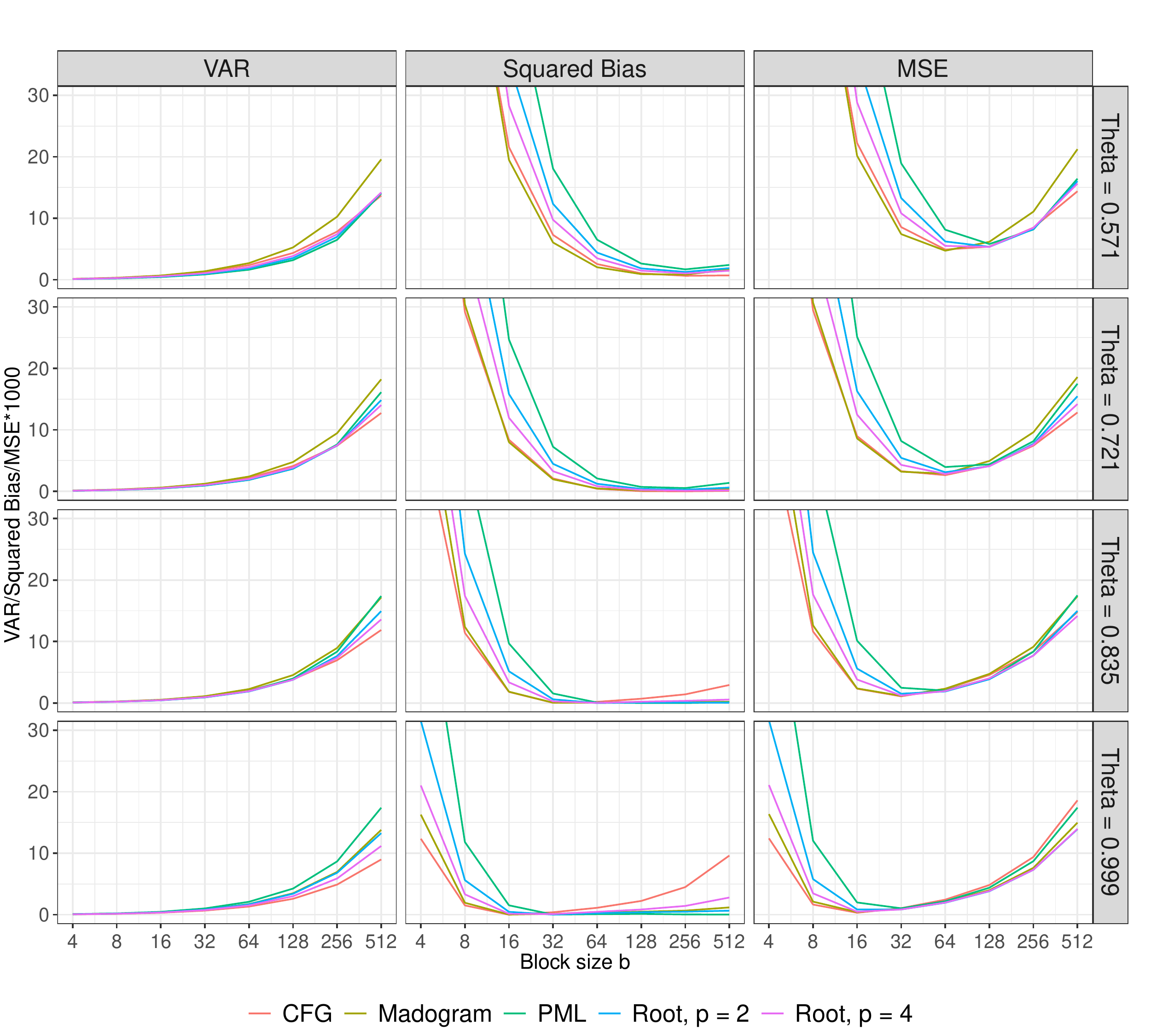}
  \caption{Variance, squared bias and MSE, multiplied by $10^3$, for the estimation of $\theta$ within the ARCH-model for four values of $\theta$.} \vspace{-.3cm}
  \label{Fig:ARCH_all}
 \end{center}
\end{figure}
 
 Next, a comparison between the $z_n$- and $y_n$-versions of the estimators is drawn in Figure~\ref{Fig:ZniVsYni}; for illustrative purposes, attention is restricted to six different models and two estimators. Remarkably, there are many models, especially for smaller values of $\theta$, in which the MSE-curves of the $y_n$-versions lie uniformly below the ones of the $z_n$-versions. In the remaining models, neither version can be said to be strictly preferable. Furthermore, it is remarkable that, for $\theta$ close to one, the MSE-curves of  the $y_n$-versions are often no longer U-shaped, but increasing in the block size instead. 
 The latter behavior may be explained by the proximity to the i.i.d.\ case, since in that case, we have 
\[
	\Pro (Y_{1:b} \geq y) = \Pro( N_{1:b} \leq e^{-y/b}) = \Pro(U_1 \leq e^{-y/b})^b = e^{-y}
\] 
for all $b\in\N$, such that there is real equality in relation (\ref{eq:yb}), resulting in a vanishing bias.

\begin{figure}[t]
 \begin{center}
  \vspace{-.3cm}
  \includegraphics[width=0.9\textwidth]{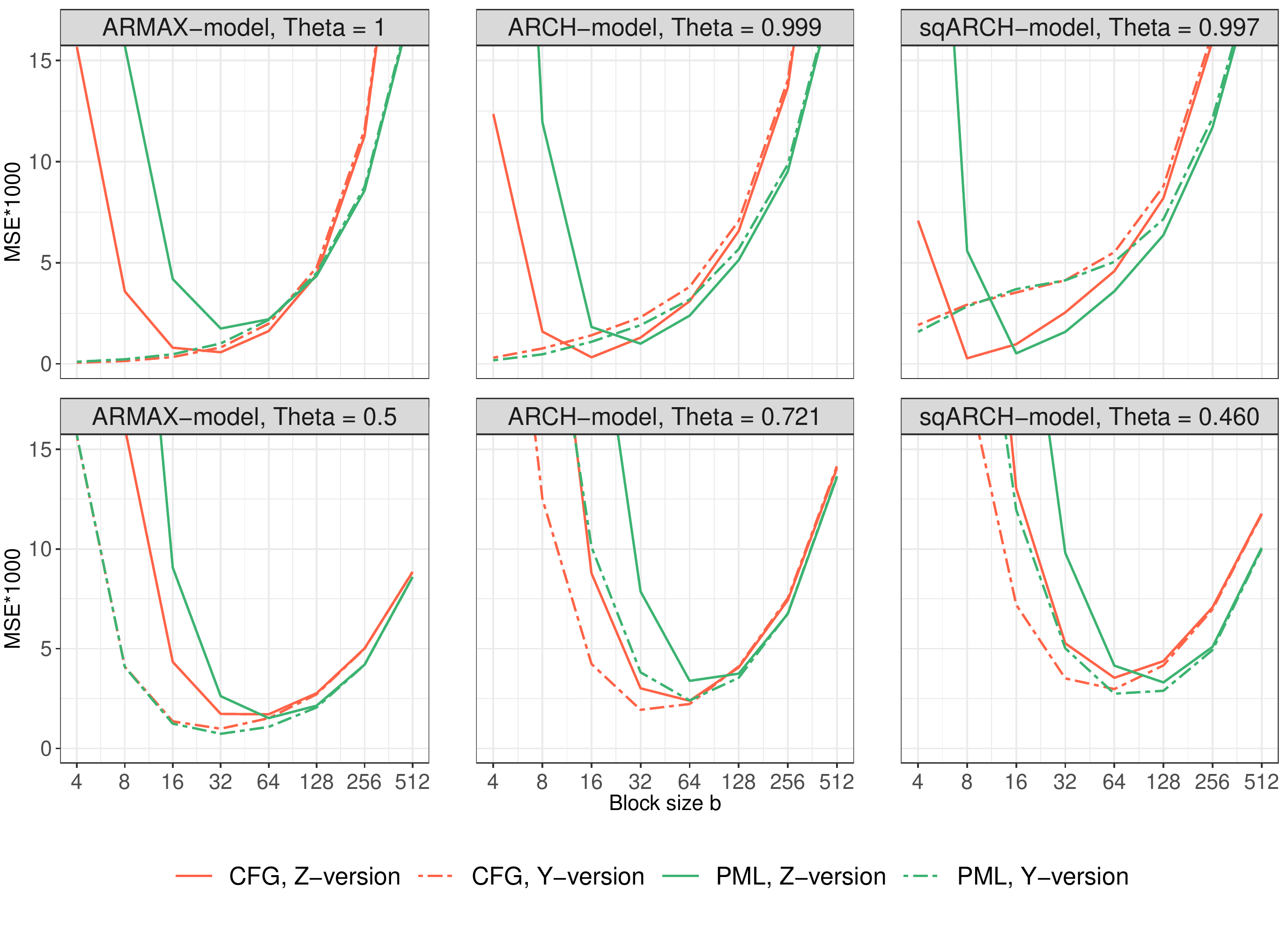}  \vspace{-.4cm}
  \caption{Comparison of the MSE multiplied by $10^3$ of the $z_n$- and $y_n$-versions of the estimators.}   \vspace{-.4cm}
  \label{Fig:ZniVsYni}
 \end{center}
\end{figure}

Next, we investigate the dependence of the performance of the Root-estimators on the parameter $p$; recall that $p=1$ yields the PML-estimator, while `$p=\infty$' yields the CFG-estimator. In Figure \ref{Fig:Wurzeln}, the MSE-curves are depicted as a function of $p$ for various fixed block sizes and for three selected models. It can be seen that choices of $p<1$ lead to a poor behavior of the corresponding estimators. At the same time, the results do not allow to identify some `optimal' choice of $p \geq 1$ which is valid uniformly over all models. A similar conclusion can be drawn from Table~\ref{Tab:Roots}, which presents, for the ARCH- and ARMAX-model and every block size $b$, the value of $p$ for which the Root-estimator attains the minimal MSE ($p=\infty$ corresponds to the CFG-estimator). One can see that most values of $p$ are represented, with $p=\infty$ appearing most often, but that there is no optimal choice of $p$ universally over all models.

 \begin{figure}[t]
 	\begin{center}
 		\includegraphics[width=0.8\textwidth]{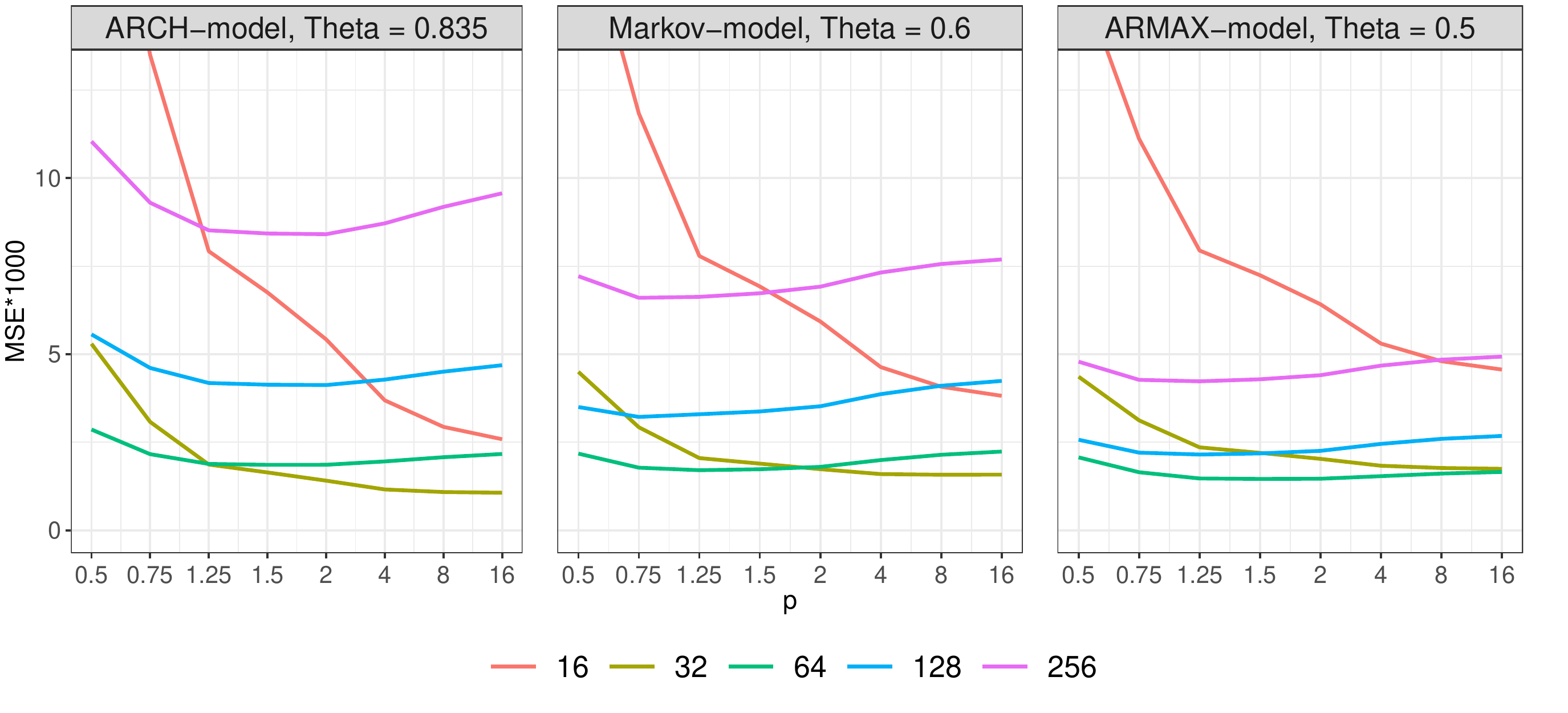}
 		\caption{Mean Squared Error multiplied by $10^3$ of the Root-estimators as a function of the parameter $p$ for block sizes $b \in \{ 16,32,64,128,256 \}$ and three different models.}
 		\vspace{-.3cm}
 		\label{Fig:Wurzeln}
 	\end{center}
 \end{figure}

\begin{table}[t!]
\centering
{\footnotesize
\begin{tabular}{ r | C{.8cm} C{.8cm}  C{.8cm} C{.8cm} |  C{.8cm} C{.8cm} C{.8cm} c }
\hline \hline
 Model & \multicolumn{4}{c|}{ARCH} & \multicolumn{4}{c}{ARMAX} \\
 \hline
 Theta & 0.999 & 0.835 & 0.721 & 0.571 & 1 & 0.75 & 0.5 & 0.25   \\
  \hline
b = 4   & $\infty$&  $\infty$ & $\infty$ & $\infty$   & $\infty$ & $\infty$ & $\infty$   & $\infty$  \\ 
 8   &    $\infty$&  $\infty$ & $\infty$ & $\infty$   & $\infty$ & $\infty$ & $\infty$   & $\infty$  \\ 
 16  &    $\infty$&  $\infty$ & $\infty$ & $\infty$   & $\infty$ & $\infty$ & $\infty$   & $\infty$  \\ 
 32  &    2 &  $\infty$ & $\infty$ & $\infty$    & $\infty$ & $\infty$ & $\infty$   & $\infty$  \\ 
 64  &    2 &  2  & $\infty$ & $\infty$  & 16 & 8  & 1.5 & 2   \\ 
 128 &    2 &  1.5 & 4  & 4    & 8  & 4  & 1 & 1   \\ 
 256 &    2 &  4  & $\infty$ & 1.25 & 4  & 8  & 1 & 0.75 \\ 
 512 &    2 &  8  & $\infty$ & $\infty$ & 4  & $\infty$ & 1 & 0.75   \\ \hline
 $\min_b$  & $\infty$ & $\infty$ & $\infty$ & $\infty$ & $\infty$ & $\infty$ & 1.5 & 1 \\
 \hline\hline
\end{tabular}
\vspace{8pt}
}
\caption{Identification of the Root-estimator $p$ with the minimum MSE for the ARCH- and ARMAX-model and every considered block size $b$. The $p$ with the minimum MSE over all blocksizes is presented in the last line.}  \vspace{-.3cm}
\label{Tab:Roots}
\end{table}

\subsection{Comparison with other estimators for the extremal index}
In this section, we compare the performance of the introduced new estimators with the following estimators: the bias-reduced sliding blocks estimator from \cite{RobSegFer09} (with a data-driven choice of the threshold as outlined in Section 7.1 of that paper), the integrated version of the blocks estimator from \cite{Rob09}, the intervals estimator from \cite{FerSeg03} and the ML-estimator from \cite{Suv07}. The parameters $\sigma$ and $\phi$ for the Robert-estimator (cf. page 276 of \citealp{Rob09}) are chosen as $\sigma = 0.7$ and $\phi= 1.3$. In the case of the intervals- and Süveges-estimator, the choice of a threshold $u$ is required, which is here chosen as the $1-1/b_n$ empirical quantile of the observed data. With regard to our estimators, we present results for the sliding-blocks, bias-reduced and $z_n$-versions, if not indicated otherwise.

\begin{figure} [t!]
 \begin{center}
  \includegraphics[width=0.8\textwidth]{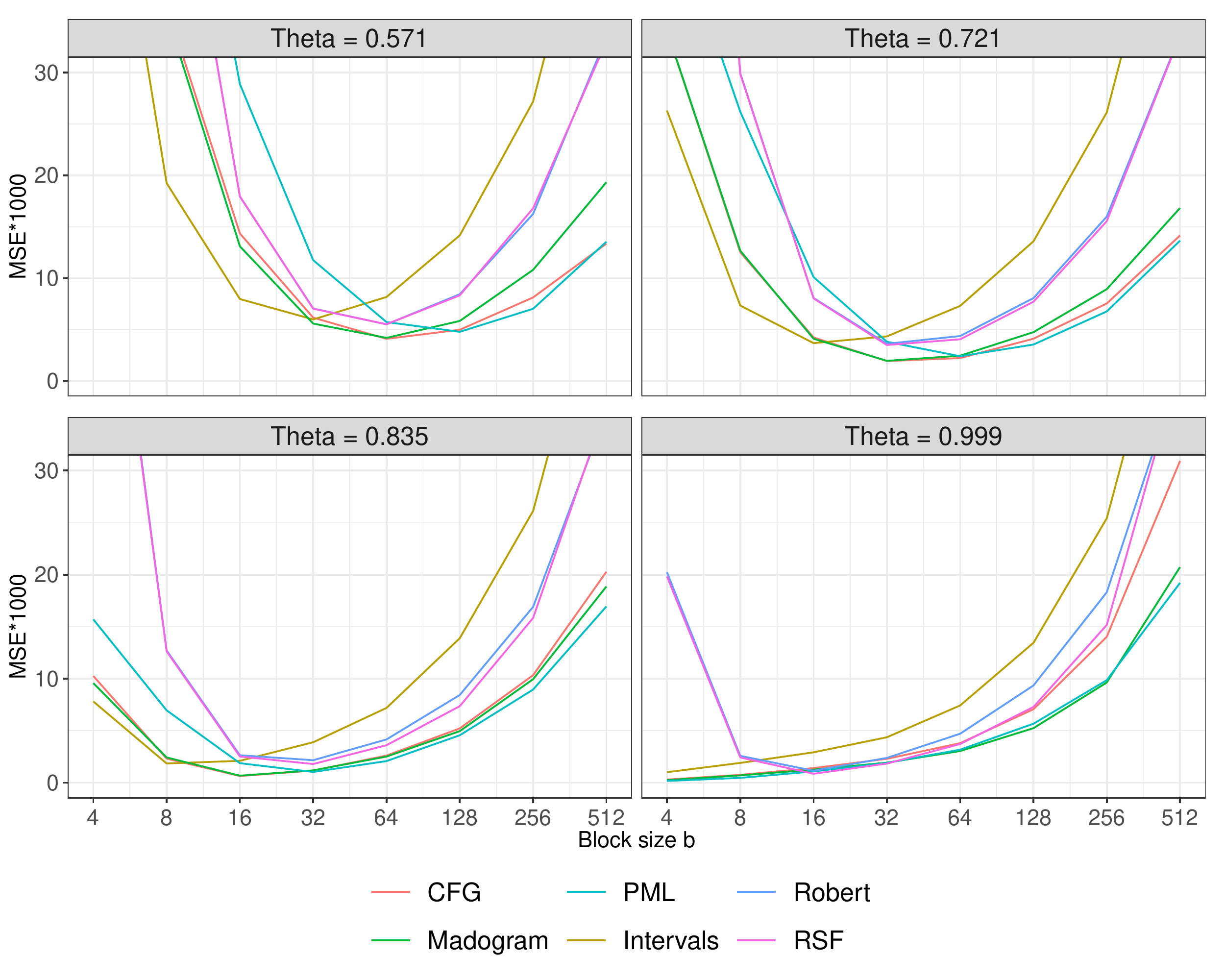} \vspace{-.3cm}
  \caption{Mean Squared Error multiplied by $10^3$ for the estimation of $\theta$ within the ARCH-model for four values of~$\theta$.}  \vspace{-.3cm}
  \label{Fig:AllEstimARCH}
 \end{center}
\end{figure}

In Figure \ref{Fig:AllEstimARCH}, we depict the MSE as a function of the block size $b$. For most models, the MSE-curves of the estimators from the literature are again U-shaped due to the bias-variance tradeoff already described in section \ref{CompNewEst}. It can further be seen that no estimator is uniformly best in any model under consideration. The method-of-moment estimators do however compare quite well to the competitors. 

The minimum values of the MSE-curves in Figure \ref{Fig:AllEstimARCH} are of particular interest. 
Due to the large amount of estimators and models under consideration (in total 26 estimators and 17 models) we try to simplify possible comparisons by the following aggregation, summarized in Table \ref{Tab:minMSE}. First, in the first four columns of the table, we calculate for each time series model and each estimator under consideration, the sum (sum over all values of $\theta$ considered for the specific model) of the minimum MSE-values (minimum over $b$). Second, in the last four columns of the table, we present the sum of the minimum MSE-values (minimum over $b$) over all models, for which the extremal index $\theta$ lies in the interval $(0,0.3], (0.3,0.6], (0.6,0.8]$ or $(0.8,1]$, respectively. It can be seen that the CFG-estimator wins thrice, the Madogram- and PML-estimator wins twice, the Süveges and the Intervals-estimator wins once, and that the remaining smallest values are covered by a version of the Root-estimator.  Also note that for large values of $\theta \in (0.8,1]$ (last column), the CFG-estimator and the Root-estimator for $p \in \{8,16\}$ are the best performing estimators. As a final interesting observation, note that the $y$-versions of the moment estimators mostly outperform the $z$-version, except for the column corresponding to $\theta\in(0.8,1]$ and some entries in the columns `Markov' and  `sqARCH'. A more refined analysis showed that these differences were almost exclusively attributable to the two specific models `Markov($\theta=0.95$)' and `sqARCH($\theta=0.997$)', which appear to be rather difficult to estimate for all estimators under consideration.

\begin{table}[t!]
	\centering
	{\footnotesize
		\begin{tabular}{l | rrrr | rrrr}
			\hline \hline
			 Estimator & ARMAX & ARCH & sqARCH & Markov & $(0,0.3]$ & $(0.3,0.6]$ & $(0.6,0.8]$ & $(0.8,1]$ \\
			\hline
			CFG, Z & 4.80 & 8.54 & 8.46 & 11.19 & 5.84 & 19.08 & 5.46 & \bf 2.61 \\
			 CFG, Y & 2.56 & \bf 6.98 & 8.41 & 12.63 & 5.08 & 15.45 & \bf 3.56 & 6.49 \\
			Madogram, Z & 5.17 & 8.87 & 7.92 & 10.77 & 5.66 & 18.12 & 5.68 & 3.27 \\
			Madogram, Y & 3.00 & \bf 7.08 & 8.62 & 12.65 & 5.10 & 15.72 & \bf 3.59 & 6.94 \\
			 PML, Z & 6.18 & 11.74 & 7.99 & 10.89 & 4.44 & 18.62 & 7.37 & 6.38 \\
			 PML, Y & \bf 1.96 & 8.40 & 7.45 & 10.99 & \bf 3.73 & 14.83 & 4.04 & 6.21 \\
			 Root, p = 0.5, Z & 9.64 & 17.37 & 11.57 & 12.11 & 4.90 & 24.18 & 11.25 & 10.35 \\
			 Root, p = 0.5, Y & 2.33 & 11.99 & 8.49 & 10.94 & 3.90 & 18.14 & 5.66 & 6.05 \\
			 Root, p = 0.75, Z & 7.08 & 13.33 & 8.83 & 10.99 & 4.44 & 19.80 & 8.79 & 7.20 \\
			 Root, p = 0.75, Y & 2.03 & 9.26 & 7.63 & 10.74 & \bf 3.66 & 15.53 & 4.41 & 6.06 \\
			 Root, p = 1.25, Z & 5.77 & 11.02 & 7.82 & 10.80 & 4.56 & 18.33 & 6.61 & 5.89 \\
			 Root, p = 1.25, Y & \bf 1.96 & 8.06 & \bf 7.37 & 11.04 & 3.74 & \bf 14.47 & 3.90 & 6.32 \\
			 Root, p = 1.5, Z & 5.54 & 10.48 & 7.86 & \bf 10.47 & 4.72 & 18.38 & 6.21 & 5.04 \\
			 Root, p = 1.5, Y & 1.98 & 7.93 & \bf 7.32 & 11.10 & 3.76 & \bf 14.32 & 3.84 & 6.40 \\
			 Root, p = 2, Z & 5.22 & 9.82 & 8.11 & \bf 10.22 & 4.84 & 18.67 & 5.76 & 4.10 \\
			 Root, p = 2, Y & 2.03 & 7.88 & \bf 7.34 & 11.16 & 3.84 & \bf 14.34 & 3.72 & 6.51 \\
			 Root, p = 4, Z & 4.84 & 9.10 & 8.40 & \bf 10.14 & 5.07 & 18.81 & 5.39 & 3.20 \\
			 Root, p = 4, Y & 2.20 & 7.52 & 7.64 & 11.58 & 4.21 & 14.53 & 3.67 & 6.52 \\
			 Root, p = 8, Z & 4.76 & 8.88 & 8.42 & 10.48 & 5.37 & 18.96 & 5.36 & \bf 2.85 \\
			 Root, p = 8, Y & 2.35 & 7.31 & 7.95 & 12.02 & 4.56 & 14.91 & 3.68 & 6.48 \\
			 Root, p = 16, Z & 4.76 & 8.69 & 8.41 & 10.78 & 5.58 & 18.99 & 5.39 & \bf 2.68 \\
			 Root, p = 16, Y & 2.45 & \bf 7.14 & 8.16 & 12.32 & 4.80 & 15.18 & \bf 3.61 & 6.47 \\
			 Intervals & 3.49 & 12.53 & 11.72 & 21.86 & \bf 3.60 & 15.55 & 11.46 & 18.98 \\
			 ML Süveges & \bf 1.90 & 22.67 & 8.70 & 25.20 & 14.93 & 30.46 & 4.95 & 8.13 \\
			 Robert & 8.54 & 12.45 & 9.97 & 13.61 & 6.46 & 22.42 & 8.34 & 7.34 \\
			 RSF & 8.09 & 11.68 & 9.77 & 15.85 & 7.28 & 23.52 & 7.52 & 7.06 \\
			\hline \hline
		\end{tabular}
		\vspace{8pt}
	}
	\caption{Sum of minimal Mean Squared Error multiplied by $10^3$ over different models and $\theta_1 \in (0,0.3], \theta_2 \in (0.3,0.6], \theta_3 \in (0.6,0.8]$ and $\theta_4 \in (0.8,1]$. The three smallest values per column are in boldface.}  \vspace{-.3cm}
	\label{Tab:minMSE}
\end{table}

\section{Conclusion}
Estimating the extremal index is a classical problem in extreme value analysis for univariate stationary time series, with many ad-hoc solutions based on diverse motivations. This paper considers a new approach that is based on certain rescaled samples of ranks of  block maxima and the method of moment principle. The underlying samples have also been used by \cite{Nor15} and \cite{BerBuc18} to define explicit (pseudo) maximum likelihood estimators for the extremal index. Using the method of moment principle instead results in a large variety of alternative estimators. Studying their properties was initially motivated by the fact that a similar approach in multivariate extremes (the rank-based CFG-estimator for the Pickands function) was found to yield a more efficient estimator than the (pseudo) maximum likelihood method \citep{GenSeg09}.

The method of moment principle being a rather universal principle, the present paper goes far beyond only  considering a CFG-type estimator. In fact, based on natural moment equations for the exponential distribution (see Section~\ref{subsec:mom}), three classes of method of moment estimators were considered, which may each be based on (1) either disjoint or sliding block maxima, and (2) on certain $y$- or $z$-transformations of the block maxima. The sliding blocks version was always found to be more efficient than the disjoint blocks version. The $y$- and $z$-version share a similar behavior in terms of their asymptotic variances, but their bias may differ substantially depending on the underlying data generating process. The initial conjecture derived from \cite{GenSeg09} was partially confirmed: for $\theta$ in an explicit neighbourhood of $1$, the asymptotic variance of the CFG-type estimator is always smaller than the one of the ML-type estimator. A comparison between the various method of moment estimators is more cumbersome, with no universal answer, neither theoretically nor in terms of simulated finite sample results. If one were to come up with a single proposal, then the simulation study overall suggests to  use the sliding blocks $y$-version of the root-estimator with an intermediate choice of $p$, say, $p=1.25$.

In comparison with many other estimators for the extremal index, the proposed estimators have the advantage of being based on only one parameter to be chosen by the statistician, namely the block size $b$. Moreover, the estimators perform equally well or even better in some typical finite sample situations. 

Finally, this work leaves some interesting questions for future research: (1) what is the minimal asymptotic variance that can be achieved by estimators based on the considered rank-based samples? (2) More generally,  are there estimators for the extremal index that are semiparametrically efficient?
(3) Can the sliding blocks method be used to derive more efficient estimators for the cluster size distribution, for instance by generalizing the disjoint blocks versions in \citealp{Rob09}?

 \section*{Acknowledgements}
This work has been supported by the Collaborative Research Center ``Statistical modeling of nonlinear dynamic processes'' (SFB 823) of the German Research Foundation, which is gratefully acknowledged. The authors are grateful to two referees and an associate editor for their constructive comments on an earlier version of this article which lead to a substantial improvements.

\appendix

\section{Proofs of Theorems~\ref{theo:cfg}--\ref{theo:root}}
\label{sec:pmaster}

The proofs of Theorems~\ref{theo:cfg}--\ref{theo:root} are actually quite similar in that each proof will be decomposed into a sequence of similar intermediate lemmas. 
Occasionally, those lemmas will be hardest to prove for  Theorem~\ref{theo:cfg} and easiest to prove for Theorem~\ref{theo:mad}; this is also reflected by the larger number of conditions required for the proof of Theorem~\ref{theo:cfg}. The proof of Theorem~\ref{theo:root} in turn is quite similar to the one in \cite{BerBuc18}, and of intermediate difficulty. 
For the above reasons, we will carry out the proof of Theorem~\ref{theo:cfg} in great detail (Section~\ref{subsec:pcfg}), and skip parts of the technical arguments needed for Theorem~\ref{theo:mad} and \ref{theo:root} where possible (Sections~\ref{subsec:pmad} and \ref{subsec:proot}). Intermediate, but less central results for the proof of Theorem~\ref{theo:cfg} are given in Sections~\ref{subsec:pdb}, \ref{subsec:psb} and \ref{subsec:pa}.

All convergences are for $n\to\infty$ if not stated otherwise.

\subsection{Proof of Theorem~\ref{theo:cfg}}
\label{subsec:pcfg}

The following notations will be used throughout: 
\begin{align*}
\hat{S}_n &= \frac{1}{k_n} \sum_{i=1}^{k_n} \log(\hat {Z}_{ni}), & 
S_n &= \frac{1}{k_n} \sum_{i=1}^{k_n} \log(Z_{ni}), \\
\hat{S}_n^\slb &= \frac{1}{n-b_n+1} \sum_{i=1}^{n-b_n+1} \log(\hat {Z}_{ni}^\slb), & 
S_n^\slb &= \frac{1}{n-b_n+1} \sum_{i=1}^{n-b_n+1} \log(Z_{ni}^\slb). 
\end{align*}
Note that $\thetahat{\djb}{\CFG}{z_n} = \varphi_{\rm (C)}^{-1}(\hat S_n)$ and $\thetahat{\slb}{\CFG}{z_n} = \varphi_{\rm (C)}^{-1}(\hat S_n^\slb)$, where $\varphi_{\rm (C)}^{-1}(x)=\exp\{-(x+\gamma)\}$. Observing that $(\varphi_{\rm (C)}^{-1})'\{\varphi_{\rm (C)}(\theta)\}=\theta$, the two assertions of the theorem are a consequence of the delta-method and Proposition~\ref{prop:cfg_snhat} and Proposition~\ref{prop:cfg_snhat2}, respectively. \qed

\begin{prop} \label{prop:cfg_snhat}
Under Condition \ref{cond:BerBuc}, \ref{cond2}(i), \ref{cond3}(i) and \ref{cond_CFG}, we have
\[ 
	\sqrt{k_n} \{ \hat{S}_n - \varphi_{\rm (C)}(\theta)\} \wto \Nor(0,\sigma^2_{\djb, \rm C} / \theta^2)  \quad \textrm{ as } n \to \infty. 
\]
\end{prop}

\begin{proof}
We  may decompose
	\[
	\sqrt{k_n} \{ \hat{S}_n-\varphi_{\rm (C)}(\theta) \} = A_n+B_n+C_n ,
	\]
	where 
	\[
	A_n = \sqrt{k_n} \{\hat{S}_n-S_n\}, \quad 
	B_n = \sqrt{k_n}\{ S_n-\Exp(S_n) \} ,\quad
	C_n = \sqrt{k_n}  \{ \Exp(S_n)-\varphi_{\rm (C)}(\theta) \}. 
	\]  
	We have $C_n=o(1)$ by Condition \ref{cond3}(i). For the treatment of $A_n$, recall the tail empirical process defined in \eqref{eq:tep}.
	Further, let $\tilde N_{ni} = (n+1)/n  \times \hat N_{ni}$, and note that
	\begin{align}
		1-\tilde {N}_{ni} = 
		\frac{1}{n} \sum_{s=1}^{n} \I(X_s > M_{ni})  \nonumber 
		&=   \frac{1}{n} \sum_{s=1}^{n} \I \Big(U_s > 1-\frac{Z_{ni}}{b_n}\Big)  \nonumber\\
		&=  \frac{\sqrt{k_n}}{n} \frac{1}{\sqrt{k_n}} \sum_{s=1}^{n}  \Big\{ \I \Big(U_s > 1-\frac{Z_{ni}}{b_n}\Big) - 
		\frac{Z_{ni}}{b_n} \Big\} + \frac{Z_{ni}}{b_n}  \nonumber\\
		&= \frac{\sqrt{k_n}}{n} e_n(Z_{ni}) + \frac{Z_{ni}}{b_n}.  \label{eqnhut}
	\end{align} 
Finally, let 
\begin{align} \label{eq:hkn}
\hat{H}_{k_n}(x):= \frac{1}{k_n} \sum_{i=1}^{k_n} \I(Z_{ni} \leq x) 
\end{align}
denote the empirical c.d.f.\ of $Z_{n1},\ldots,Z_{nk_n}$. By Equation \eqref{eqnhut}, we obtain
	\begin{align}
		A_n 
		&= 
		\frac{1}{\wk} \sum_{i=1}^{k_n} \log(1-\hat {N}_{ni}) - \log\left( Z_{ni}  b_n\ho{-1} \right) \nonumber\\
		&= 
		\frac{1}{\wk} \sum_{i=1}^{k_n} \log\left\{ \frac{n}{n+1} \left( \frac{1}{n}+1-\tilde {N}_{ni} \right) \right\}
		-  \log\left( \frac{Z_{ni}}{ b_n} \right) \nonumber\\
		&=  
		\frac{1}{\wk} \sum_{i=1}^{k_n} \left[ \log \left\{ \frac{1}{n} 
		+ \frac{\wk}{n} e_n(Z_{ni}) + 
		\frac{Z_{ni}}{b_n} \right\} - \log\left( \frac{Z_{ni}}{ b_n} \right)  + \log\left(\frac{n}{n+1} \right) \right] \nonumber\\
		&= 
		\frac{1}{\wk} \sum_{i=1}^{k_n} \log\left\{ 1+ 
		\frac{\sqrt{k_n}b_n}{n} \cdot  \frac{e_n(Z_{ni})}{Z_{ni}} + \frac{b_n}{nZ_{ni}} \right\}
		+ \wk \log\Big(\frac{n}{n+1}\Big) \nonumber\\
		&= \int_{0}^{\infty} W_n(x) \, \mathrm{d}\hat{H}_{k_n}(x) +o(1), \label{Andisj}
	\end{align}  
	where 
	\[
	W_n(x) = \wk \log \left\{ 1 + \frac{1}{\wk}
		\left( \frac{e_n(x)}{x} +\frac{1}{\sqrt{k_n} x}\right) \right\}. 
	\]
Heuristically, $\hat H_{k_n}(x) \approx 1-\exp(-\theta x)$ and $W_n(x) \approx e(x)/x$ (where $e$ denotes the limit of the tail empirical process), whence the tentative limit of $A_n$ should be 
\[
A = \int_{0}^{\infty} \frac{e(x)}x \theta e^{-\theta x} \diff x.
\]  

For a rigorous treatment of $A_n + B_n$, let
	\begin{align*}
	E_n &= \int_{0}^{\infty} W_n(x) \, \mathrm{d}\hat{H}_{k_n}(x), & 
		E_{n,m} &= \int_{1/m}^{m} W_n(x) \ \mathrm{d}\hat{H}_{k_n}(x), &
		E_m' &= \int_{1/m}^{m} \ \frac{e(x)}{x} \ \theta e^{-\theta x} \ \mathrm{d}x
	\end{align*}	
	and let $B$ be defined as in Lemma~\ref{lem:cfg-fidis} below.
	As shown above, $A_n = E_n + o(1)$. The proposition is hence a consequence of Wichura's theorem (\cite{Bil79}, Theorem 25.5) and the following items:
	\smallskip
	\begin{compactenum}[(i)]
	\item For all $m \in \N$: $E_{n,m} + B_n \wto  E'_m + B$ as $n\to\infty$.
	\item  $E'_m + B  \wto  A + B \sim \Nor(0, \sigma^2_{\djb, {\rm C}}/\theta^2)$ as $m\to\infty$.
	\item For all $\delta>0$: $\lim_{m\to\infty} \limsup_{n\to\infty} \Prob(|E_{n} - E_{n,m} |> \delta) = 0$.
	\end{compactenum}
	\smallskip
	The assertion in (i) is proven in Lemma~\ref{lem:cfg-enmbn}.
	The assertion in (ii) follows from the fact that $E'_m + B$ is normally distributed with variance $\tau_m^2$ as specified in Lemma \ref{lem:cfg-enmbn}, and the fact that $\tau_m^2  \to \sigma^2_{\djb, {\rm C}}/\theta^2$ as $m\to\infty$ by Lemma~\ref{lem:cfg-sigcon}. Finally, Lemma~\ref{lem:cfg-limsup} proves (iii).
\end{proof}

\begin{prop} \label{prop:cfg_snhat2}
Under Condition \ref{cond:BerBuc}, \ref{cond2}(i), \ref{cond3}(i) and \ref{cond_CFG}, we have
\[ 
	\sqrt{k_n} \{ \hat{S}_n^\slb - \varphi_{\rm (C)}(\theta)\} \wto \Nor(0,\sigma^2_{\slb, \rm C} / \theta^2)  \quad \textrm{ as } n \to \infty. 
\]
\end{prop}

\begin{proof} The proof is very similar to the proof of Proposition~\ref{prop:cfg_snhat}. Decompose
\[ 
\sqrt{k_n} \{ \hat{S}^{\slb}_n-g(\theta) \} 
= A_n^{\slb}+B_n^{\slb}+C_n^{\slb}, \]
where 
\[
A_n^{\slb} := \sqrt{k_n} \{\hat{S}^{\slb}_n-S^{\slb}_n\}, \quad 
B_n^{\slb}:= \sqrt{k_n}\{ S^{\slb}_n-E[S^{\slb}_n] \},  \quad 
C_n^{\slb}:= \sqrt{k_n} \{E[S^{\slb}_n]-\varphi_{\rm (C)}(\theta)\}.
\]
Again, we have $C_n^\slb=o(1)$ by Condition \ref{cond3}(i).  A similar calculation as in (\ref{Andisj}) in the case of the disjoint blocks shows that $A_n^{\slb}$ can be written in the following way
\begin{equation}
	A_n^{\slb} = 
	 \int_{0}^{\infty} W_n(x) \ \mathrm{d}\hat{H}_{n}^{\slb}(x) +o(1), \nonumber
\end{equation}
where 
\[ 
\hat{H}_{n}^{\slb} (x) = \frac{1}{n-b_n+1} \sum_{t=1}^{n-b_n+1} \I(Z_{nt}^{\slb} \leq x) 
\] 
denotes the empirical c.d.f. of $Z_{n1}^{\slb},\ldots, Z_{n,n-b_n+1}^{\slb}$.  We may now treat $A_{n}^\slb + B_n^{\slb}$  exactly as $A_n + B_n$ in the proof of Proposition~\ref{prop:cfg_snhat}, with $E_n, E_{n,m}$ and Lemma 
\ref{lem:cfg-enmbn}, \ref{lem:cfg-sigcon} and \ref{lem:cfg-limsup} replaced by 
	\begin{align*}
	E_n^\slb &= \int_{0}^{\infty} W_n(x) \, \mathrm{d}\hat{H}_{n}^\slb(x), & 
		E_{n,m}^\slb &= \int_{1/m}^{m} W_n(x) \ \mathrm{d}\hat{H}_{n}^\slb(x), 
	\end{align*}	
	 and Lemma \ref{lem:cfg-enmbn2}, \ref{lem:cfg-sigcon2} and \ref{lem:cfg-limsup2}, respectively.
\end{proof}


\subsection{Proof of Theorem \ref{theo:mad}} 
\label{subsec:pmad} 

The following notation will be used throughout:
\begin{align*}
 \hat{S}_n & = \frac{1}{k_n} \sum_{i=1}\ho{k_n} \exp(-\hat{Z}_{ni}), &S_n& = \frac{1}{k_n} \sum_{i=1}\ho{k_n} \exp(-Z_{ni}),\\
 \hat{S}_n\ho{\slb} & = \frac{1}{n-b_n+1} \sum_{i=1}\ho{n-b_n+1} \exp(-\hat{Z}_{ni}\ho{\slb}),&  S_n\ho{\slb} & = \frac{1}{n-b_n+1} \sum_{i=1}\ho{n-b_n+1} \exp(-Z_{ni}\ho{\slb}).
\end{align*}
Note that $\thetahat{\djb}{\MAD}{z_n} = \varphi_{(\rm M)}^{-1}(\hat S_n)$ and $\thetahat{\slb}{\MAD}{z_n} = \varphi_{(\rm M)}^{-1}(\hat S_n^{\slb})$, where $\varphi_{(\rm M)}(x) = x/(1+x)$. The assertion follows from the delta-method and Proposition \ref{Prop_Mad_dj} and \ref{Prop_Mad_sl}.  \qed

\begin{prop} \label{Prop_Mad_dj}
 Under Condition \ref{cond:BerBuc} and \ref{cond3}(ii), we have 
 \[ \sqrt{k_n} \{ \hat S_n - \varphi_{(\rm M)}(\theta)\} \wto \Nor(0,\sigma_{\djb, \rm M}^2/(1+\theta)^4) \ \textrm{ as }  n \to \infty. \]
\end{prop}
\begin{proof}
 Write $\sqrt{k_n} \{ \hat{S}_n - \varphi_{\rm (M)}(\theta)\} = A_n + B_n + C_n,$ where 
\[ 
A_n = \wk \{\hat{S}_n-S_n\}, \ B_n = \wk \{S_n - E[S_n]\}, \ C_n = \wk \{E[S_n]-\varphi_{\rm (M)}(\theta)\}. 
\]
The term $C_n$ is asymptotically negligible by Condition \ref{cond3}(ii). A straightforward calculation shows that the summand $A_n$  can be written in terms of the tail empirical process $e_n$ as 
\begin{align}
 A_n 
 &= \intnu W_n(x) \ \mathrm{d}\hat{H}_{k_n}(x), \qquad 
 W_n(x) = \wk e\ho{-x} \left[  \exp(-e_n(x) k_n\ho{-1/2}) -1 \right], 
 \nonumber \label{AnMado}
\end{align}
where $\hat H_{k_n}$ is the empirical c.d.f. of $Z_{n1},\ldots,Z_{nk_n}$, see \eqref{eq:hkn}.
The asymptotic normality of $A_n+B_n$ can now be shown  as in the proof of Proposition~\ref{prop:cfg_snhat}. The corresponding key result is given by Lemma \ref{MadodjLem1};  whose  proof is similar (but easier) as for the CFG-estimator (Lemma~\ref{lem:cfg-fidis}) and is omitted for the sake of brevity.
\end{proof}

\begin{lem} \label{MadodjLem1}
 (a) For any $x_1,\ldots,x_m \in [0,\infty)$, as $ n \to \infty$, 
 \[ (e_n(x_1),\ldots,e_n(x_m),B_n) \wto (e(x_1),\ldots,e(x_m),B) \sim \Nor_{m+1}(0,\Sigma_{m+1}), \]
 with \[ \Sigma_{m+1} = \begin{pmatrix}
                          r(x_1,x_1) & \dots & r(x_1,x_m) & f(x_1) \\
                          \vdots & \ddots & \vdots & \vdots \\
                          r(x_m,x_1) & \dots & r(x_m,x_m) & f(x_m)\\
                          f(x_1) & \dots & f(x_m) & \frac{\theta}{\theta+2} - \frac{\theta\ho{2}}{(\theta + 1)\ho{2}}
                         \end{pmatrix},
 \] where the covariance function $r$ is given as in Lemma \ref{lem:cfg-fidis} and
 \begin{align}
  f(x) &= \sum_{i=1}\ho{\infty} i \intne p\ho{(x)}(i)- p_2\ho{(x,-\log(y))}(i,0) \I(x \geq -\log(y)) \ \mathrm{d}y - x \varphi_{\rm (M)}(\theta).
  \nonumber
 \end{align}
(b) For any $x_1,\ldots,x_m \in [0,\infty)$, as $ n \to \infty$,
\[ (W_n(x_1),\ldots,W_n(x_m),B_n) \wto (-e\ho{-x_1}e(x_1),\ldots,-e\ho{-x_m} e(x_m),B). \]
\end{lem}

\begin{prop} \label{Prop_Mad_sl}
 Under Condition \ref{cond:BerBuc} and \ref{cond3}(ii), we have
 \[ \sqrt{k_n} \{ \hat S_n^{\slb} - \varphi_{(\rm M)}(\theta)\} \wto \Nor(0,\sigma_{\slb, \rm M}^2/(1+\theta)^4) \ \textrm{ as }  n \to \infty.\]
\end{prop}
\begin{proof}
 The proof is similar to the proof of Proposition \ref{Prop_Mad_dj}. Decompose 
$\wk \{\hat{S}_n\ho{\slb} - \varphi_{\rm (M)}(\theta)\} = A_n\ho{\slb} + B_n\ho{\slb} +C_n\ho{\slb},$ where 
\[ A_n\ho{\slb} = \wk \{ \hat{S}\ho{\slb}_n-S\ho{\slb}_n\}, \ B\ho{\slb}_n = \wk \{S\ho{\slb}_n - E[S\ho{\slb}_n]\}, 
\ C\ho{\slb}_n = \wk \{E[S\ho{\slb}_n]-\varphi_{\rm (M)}(\theta)\}. \]
Again, we have $C_n\ho{\slb}=o(1)$ by Condition \ref{cond3}(ii) and \[A_n\ho{\slb} = \intnu W_n(x) \ \mathrm{d}\hat{H}_{n}\ho{\slb}(x),\] where $\hat H_{n}^{\slb}$ denotes the empirical c.d.f. of $Z_{n1}^{\slb},\ldots,Z_{n,n-b_n+1}^{\slb}$. The sum $A_n^{\slb}+B_n^{\slb}$ can now be treated as in proof of Proposition~\ref{prop:cfg_snhat2}.
The corresponding key result, Lemma~\ref{lem:cfg-fidis2},  needs to be replaced by Lemma \ref{MadslLem1}; whose proof is again omitted for the sake of brevity.
\end{proof}

\begin{lem} \label{MadslLem1}
 (a) For any $x_1,\ldots,x_m \in [0,\infty)$, as $ n \to \infty$, 
 \[ (e_n(x_1),\ldots,e_n(x_m),B\ho{\slb}_n) \wto (e(x_1),\ldots,e(x_m),B\ho{\slb}) \sim \Nor_{m+1}(0,\Sigma\ho{\slb}_{m+1}), \]
where all entries of $\Sigma_{m+1}^{\slb}$ are the same as those of $\Sigma_{m+1}$ in Lemma~\ref{MadodjLem1} except for the entry at position $(m+1,m+1)$, which needs to be replaced by
\[ v(\theta) = 2- \frac{4}{\theta+1} + 4 \ \frac{\log(\theta+1)- \log(\theta+2)+ \log(2)}{\theta (\theta +1)} - \frac{2 \theta\ho{2}}{(\theta+1)\ho{2}}. \]
(b) For any $x_1,\ldots,x_m \in [0,\infty)$, as $ n \to \infty$,
\[ (W_n(x_1),\ldots,W_n(x_m),B\ho{\slb}_n) \wto (-e\ho{-x_1}e(x_1),\ldots,-e\ho{-x_m} e(x_m),B\ho{\slb}). \]
\end{lem}


\subsection{Proof of Theorem \ref{theo:root}} 
\label{subsec:proot} 

For fixed $p>0$, define 
\begin{align*}
 \hat S_n &= \frac{1}{k_n} \sum_{i=1}^{k_n} \hat Z_{ni}^{1/p}, & S_n &= \frac{1}{k_n} \sum_{i=1}^{k_n} Z_{ni}^{1/p}, \\
 \hat S_n^{\slb} &= \frac{1}{n-b_n+1} \sum_{i=1}^{n-b_n+1} \hat Z_{ni}^{1/p}, & S_n^{\slb} &= \frac{1}{n-b_n+1} \sum_{i=1}^{n-b_n+1} Z_{ni}^{1/p}. 
\end{align*}
Note that $\thetahat{\djb}{\ROOT, p}{z_n} = \varphi_{\rm (R),p}^{-1}(\hat S_n)$ and $\thetahat{\slb}{\ROOT, p}{z_n} = \varphi_{\rm (R),p}^{-1}(\hat S_n^{\slb})$, where $\varphi_{\rm (R),p}(x) = x^{-1/p}\Gamma(1+1/p)$. By the delta-method, the assertion follows from Proposition \ref{Prop_Root_dj} and \ref{Prop_Root_sl}. \qed

\begin{prop} \label{Prop_Root_dj}
 Under Condition \ref{cond:BerBuc}, \ref{cond2}(ii) and \ref{cond3}(iii), we have
 \[ \sqrt{k_n} \{ \hat S_n - \varphi_{\rm (R),p}( \theta) \} \wto \Nor(0,\sigma_{\djb,p}^2 \psi_p(\theta) ) \textrm{ as } n \to \infty,\] where $\psi_p(\theta) = \Gamma(1+1/p)^2p^{-2} \theta^{-(2+2/p)}$.
\end{prop}

\begin{proof}
 Decompose $\sqrt{k_n} \{ \hat{S}_n - 
\varphi_{\rm (R),p}(\theta)\} = A_n + B_n + C_n,$ where 
\[ A_n = \wk \{\hat{S}_n-S_n\}, \ B_n = \wk \{S_n - E[S_n]\} \ \textrm{ and } C_n = \wk \{E[S_n]-\varphi_{\rm (R),p}(\theta)\}.   \]
By Condition \ref{cond3}(iii), the term $C_n$ converges to zero. A straightforward calculation shows that the term $A_n$ can be written as  
\begin{align}
 A_n 
 = \intnu W_n(x) \ \mathrm{d}\hat{H}_{k_n}(x), \nonumber \qquad
 W_n(x) = \wk \left\{ \left[ \frac{e_n(x)}{\wk} + x \right]\ho{1/p} - x\ho{1/p} \right\}.
\end{align}
The asymptotic normality of $A_n+B_n$ can be shown as in the proof of Proposition \ref{prop:cfg_snhat} by an application of Wichura's theorem. Here, Lemma \ref{lem:cfg-fidis} needs to be replaced by Lemma \ref{Root_dj_Lem1}, whose proof is similar but easier and therefore omitted for the sake of brevity.
\end{proof}

\begin{lem} \label{Root_dj_Lem1}
 (a) For any $x_1,\ldots,x_m \in (0,\infty)$, as $ n \to \infty$, 
 \[ (e_n(x_1),\ldots,e_n(x_m),B_n) \wto (e(x_1),\ldots,e(x_m),B) \sim \Nor_{m+1}(0,\Sigma_{m+1}) \]
 with \[ \Sigma_{m+1} = \begin{pmatrix}
                          r(x_1,x_1) & \dots & r(x_1,x_m) & f_p(x_1) \\
                          \vdots & \ddots & \vdots & \vdots \\
                          r(x_m,x_1) & \dots & r(x_m,x_m) & f_p(x_m)\\
                          f_p(x_1) & \dots & f_p(x_m) & v_p(\theta)
                         \end{pmatrix},
 \] where the covariance function $r$ is defined as in Lemma \ref{lem:cfg-fidis} and 
 \begin{align}
  f_p(x) &= \sum_{i=1}\ho{\infty} i \intnu p_2\ho{(x,y\ho{p})}(i,0) \I(x \geq y\ho{p}) \ \mathrm{d}y - x \varphi_{\rm (R),p}(\theta),
  \nonumber\\
  v_p(\theta) &= \theta\ho{\frac{-2}{p}} \left\{ \Gamma(1+2/p)-\Gamma(1+1/p)\ho{2} \right\}. \nonumber
 \end{align}
 
(b) For any $x_1,\ldots,x_m \in (0,\infty)$, as $ n \to \infty$,
\[ (W_n(x_1),\ldots,W_n(x_m),B_n) \wto \big(e(x_1)x_1\ho{\frac{1}{p}-1}p\ho{-1},\ldots,e(x_m)x_m\ho{\frac{1}{p}-1}p\ho{-1},B\big). \]
\end{lem}

\begin{prop} \label{Prop_Root_sl}
 Under Condition \ref{cond:BerBuc}, \ref{cond2}(ii) and \ref{cond3}(iii), we have
 \[ \sqrt{k_n} \{ \hat S_n^{\slb} - \varphi_{\rm (R),p}( \theta) \} \wto \Nor(0,\sigma_{\slb,p}^2 \psi_p(\theta) ) \textrm{ as } n \to \infty,\] where $\psi_p(\theta) = \Gamma(1+1/p)^2p^{-2} \theta^{-(2+2/p)}$.
\end{prop}
\begin{proof}
The proof is similar to the proof of Proposition \ref{Prop_Root_dj}. Write $\sqrt{k_n} \{ \hat{S}_n^{\slb} - 
\varphi_{\rm (R),p}(\theta)\} = A_n^{\slb} + B_n^{\slb} + C_n^{\slb}$, where 
\[ A_n^{\slb} = \wk \{\hat{S}_n^{\slb}-S_n^{\slb}\}, \ B_n^{\slb} = \wk \{S_n^{\slb} - E[S_n^{\slb}]\} \ \textrm{ and } C_n^{\slb} = \wk \{E[S_n^{\slb}]-\varphi_{\rm (R),p}(\theta)\}.   \]
By Condition \ref{cond3}(iii), $C_n^{\slb} = o(1)$, and a straightforward calculation yields \[A_n\ho{\slb} = \intnu W_n(x) \ \mathrm{d}\hat{H}_{n}\ho{\slb}(x),\] where $\hat H_{n}^{\slb}$ denotes the empirical c.d.f. of $Z_{n1}^{\slb},\ldots,Z_{n,n-b_n+1}^{\slb}$. The sum $A_n^{\slb}+B_n^{\slb}$ can be treated as in the proof of Proposition \ref{prop:cfg_snhat2}, where the main result, Lemma \ref{lem:cfg-fidis2}, needs to be replaced by Lemma \ref{RootSlLem1}, whose proof is omitted for the sake of brevity. 
\end{proof}

\begin{lem} \label{RootSlLem1}
 (a) For any $x_1,\ldots,x_m \in (0,\infty)$, as $ n \to \infty$, 
 \[ (e_n(x_1),\ldots,e_n(x_m),B_n^{\slb}) \wto (e(x_1),\ldots,e(x_m),B^{\slb}) \sim \Nor_{m+1}(0,\Sigma_{m+1}^{\slb}), \]
 where all entries of $\Sigma_{m+1}^{\slb}$ are the same as those of $\Sigma_{m+1}$ in Lemma \ref{Root_dj_Lem1} except for the entry at position $(m+1,m+1)$, which needs to be replaced by
 \begin{align*}
v_p^{\slb}(\theta) &= 4p^{-2}\theta^{-2/p} \intnu (1-e^{-z})z^{1/p-2} \Gamma(1/p,z) \ \mathrm{d}z - 2 \theta^{-2/p} \Gamma(1+1/p)^2.
 \end{align*}
 
(b) For any $x_1,\ldots,x_m \in (0,\infty)$, as $ n \to \infty$,
\[ (W_n(x_1),\ldots,W_n(x_m),B_n^{\slb}) \wto \big(e(x_1)x_1\ho{\frac{1}{p}-1}p\ho{-1},\ldots,e(x_m)x_m\ho{\frac{1}{p}-1}p\ho{-1},B^{\slb} \big). \]
\end{lem}

\section{Auxiliary results for the proof of Theorem~\ref{theo:cfg}}

\subsection{Auxiliary lemmas -- Disjoint blocks}
\label{subsec:pdb}

Throughout this section, we assume that Condition \ref{cond:BerBuc}, \ref{cond2}(i) and \ref{cond3}(i) are met.
	
	\begin{lem} \label{lem:cfg-fidis}
	For any $x_1,\ldots,x_m \in [0,\infty)$ and $m \in \N$, we have
	 \[
	 (e_n(x_1),\ldots,e_n(x_m),B_n)' 
			\wto (e(x_1),\ldots,e(x_m),B)', 
			\] where $ (e(x_1),\ldots,e(x_m),B)' \sim \Nor_{m+1} (0,\Sigma_{m+1})$ with 
			\[
			\Sigma_{m+1}= 
			\begin{pmatrix}
			r(x_1,x_1) & \dots & r(x_1,x_m) & f(x_1) \\
			\vdots & \ddots& \vdots & \vdots \\
			r(x_m,x_1) & \dots & r(x_m,x_m) & f(x_m) \\
			f(x_1) & \dots & f(x_m) & \pi\ho{2}/6
			\end{pmatrix}.
			\] 
			Here, $r(0,0)=0$ and, for $x \geq y \geq 0$ with $x \neq 0$,
			\begin{align*}
			r(x,y) &= \theta x \sum_{i=1}\ho{\infty} \sum_{j=0}\ho{i} ij \pi_2\ho{(y/x)}(i,j),  \qquad
			f(x) = h(x)-x\varphi_{(\rm C)}(\theta), \\
			h(x) &= \sum_{i=1}\ho{\infty} i \ \left[ \int_{0}\ho{\infty} \I(e\ho{y} \leq x) 
			p_2\ho{(x,e\ho{y})}(i,0) \ \mathrm{d}y - \int_{-\infty}\ho{0} p\ho{(x)}(i) -
			\I(e\ho{y} \leq x)   p_2\ho{(x,e\ho{y})} (i,0)  \ \mathrm{d}y \right] \nonumber
			\end{align*}
			and where, for $i \geq j \geq 0,\ i \geq 1$,
			\[ 
			p_2\ho{(x,y)}(i,j)= \Pro\big(\bm N_E\ho{(x,y)}=(i,j)\big), \ \ \ \bm N_E\ho{(x,y)} = \sum_{i=1}\ho{\eta}
			\big(\xi_{i1}\ho{(y/x)}, \xi_{i2}\ho{(y/x)}\big) 
			\]
			with $\eta \sim \mathrm{Poisson}(\theta x)$ independent of i.i.d.\ random vectors 
			$\big(\xi_{i1}\ho{(y/x)}, \xi_{i2}\ho{(y/x)}\big) \sim \pi_2\ho{(y/x)}, i \in \N$ and 
			\[ p\ho{(x)}(i) = \Pro(N_E\ho{(x)} = i), \ \ \ N_E\ho{(x)} = \sum_{i=1}\ho{\eta_2} \xi_i\]
			with $\eta_2 \sim \mathrm{Poisson}(\theta x)$ independent of i.i.d.\ random variables $\xi_i \sim \pi,\ i \in \N$. 		
	\end{lem}
	
	\begin{lem}\label{lem:cfg-process}
		For any $m \in \N$, we have
		\[ 
		\left\{ (W_n(x),B_n)' \right\}_{x \in [1/m,m]}  
		\wto \left\{ \left( \frac{e(x)}{x},B \right)' \right\}_{x \in [1/m,m]} \ \textrm{ in } D([1/m,m]) \times \R, 
		\]
		where $(e,B)'$ is a centered Gaussian process with continuous sample paths and with covariance
		functional as specified in Lemma \ref{lem:cfg-fidis}. 
	\end{lem}

		\begin{lem} \label{lem:cfg-EE'}
		For any $m\in\N$, we have
		\[ 
		E_{n,m} = E_{n,m}' + o_\Prob(1)  \quad \text{ as } n \to \infty, 
		\] 
		where 
		$
		E_{n,m}' = \int_{1/m}^{m} W_n(x) \theta e^{-\theta x} \ \mathrm{d}x.
		$ 
	\end{lem}

	\begin{lem} \label{lem:cfg-enmbn}
		For any $m \in \N$,  we have
		\[ 
		E_{n,m} + B_n \wto E_m'+B \sim \Nor(0,\tau_m^2) \ \textrm{ as } n \to \infty, 
		\] where, with $r$ and $f$ defined as in Lemma~\ref{lem:cfg-fidis}, 
		\[ 
		\tau_m^2= \theta\ho{2} \int_{1/m}\ho{m} 
		\int_{1/m}\ho{m} r(x,y) \frac{1}{xy} e\ho{-\theta(x+y)} \ \mathrm{d}x \mathrm{d}y + 2 \theta \int_{1/m}\ho{m} f(x) \frac{1}{x} e\ho{-\theta x} \ \mathrm{d}x + \frac{\pi\ho{2}}{6}. \]
	\end{lem}
	\begin{lem} \label{lem:cfg-sigcon}
		As $m \to \infty$, $\tau_m^2 \to \sigma_{\djb, (\rm C)}^2/\theta^2$, where $ \sigma_{\djb, (\rm C)}^2$ is specified in Theorem \ref{theo:cfg}.	
		\end{lem}

	\begin{lem} \label{lem:cfg-limsup} 
		If, in addition to Condition \ref{cond:BerBuc}, \ref{cond2}(i) and \ref{cond3}(i), Condition \ref{cond_CFG} holds, then, for all $\delta > 0$, 
		\[ 
		\lim_{m \to \infty} \limsup_{n \to \infty} P \! \left( |E_{n,m}-E_n | > \delta \right) = 0. 
		\]
	\end{lem}

\begin{proof}[Proof of Lemma \ref{lem:cfg-fidis}] We proceed similarly as in the proof of Lemma 9.3 in \cite{BerBuc18}.
	Weak convergence of the vector $(e_n(x_1),\ldots,e_n(x_m))'$ is a consequence of Theorem 4.1 in \cite{Rob09}. 
	For the treatment of the joint convergence with $B_n$, we only consider the case $m=1$ and set $x_1=x$; the general case can be treated analogously.  For $i=1, \dots, k_n$, we decompose a block $I_i=\{(i-1)b_n+1, \dots, ib_n\}$ into a big block $I_i\ho{+}$ and a small block $I_i\ho{-}$, where, recalling $\ell_n$ from Condition~\ref{cond:BerBuc}(iii),
	\[ 
	I_i^{+} = \{(i-1)b_n+1,\ldots,ib_n-\ell_n\}, \qquad I_i^{-} = \{ ib_n-\ell_n+1,\ldots,ib_n \}, 
	\] and set 
	\begin{align*}
		e_n^{+}(x) &= \frac{1}{\sqrt{k_n}} \sum_{i=1}^{k_n} \sum_{s \in I_i^{+}} \Big\{ \I \Big( U_s > 1-\frac{x}{b_n} \Big) -\frac{x}{b_n} \Big\},  \qquad
		B_n^{+} = \frac{1}{\sqrt{k_n}} \sum_{i=1}^{k_n} \Big\{ \log(Z_{ni}^{+}) - \Exp[\log(Z_{ni}^{+})] \Big\},
	\end{align*}
	where $Z_{ni}^{+} = b_n (1-N_{ni}^{+}), N_{ni}^{+} = \max_{s \in I_i^{+}} U_s$.
	Next, according to Lemma 6.6 in \cite{Rob09}, 
	\[ 
	e_n^{-}(x) := e_n(x) - e_n^{+}(x) = o_\Prob(1). 
	\] 
	It can further be shown by the same arguments as in the proof of Lemma~9.3 in \cite{BerBuc18} that
	\[
	B_n^{-} := B_n - B_n^{+} = o_\Prob(1).
	\] 	
	Finally, for $\varepsilon \in (0,c_1 \wedge c_2)$, define $A_n\ho{+} = \{\min_{i=1,\ldots,k_n} N_{ni}\ho{+} > 1- \varepsilon \}$, and note that $\Prob(A_n\ho{+}) \to 1$ by Condition \ref{cond:BerBuc}(v).  As a consequence of the previous three statements, it suffices to show that, using the Cram\'{e}r-Wold device, 
	\begin{align} \label{eq:cw1}
	\{ \lambda_1 e_n\ho{+}(x)+ \lambda_2 B_n\ho{+} \} \I_{A_n\ho{+}} \wto \lambda_1 e(x) + \lambda_2 B,
     \end{align}
     for any $\lambda_1,\lambda_2 \in \mathbb{R}$. 
     
     Now, the left-hand side of \eqref{eq:cw1} can be written as 
        \begin{align*}
         &\phantom{=}~  \{ \lambda_1 e_n\ho{+}(x)+ \lambda_2 B_n\ho{+} \} \I_{A_n\ho{+}} 
         = 
         \frac{1}{\sqrt{k_n}} \sum_{i=1}\ho{k_n} \tilde{g}_{i,n} + o_\Prob(1), \nonumber
        \end{align*}
        where $\tilde{g}_{i,n} = g_{i,n} \I(Z_{ni}\ho{+} < \varepsilon b_n)$ and  where 
        \[ 
        g_{i,n} = \lambda_1 \sum\nolimits_{s \in I_i\ho{+}} \Big\{ \I\Big( U_s > 1- \frac{x}{b_n} \Big) - \frac{x}{b_n} \Big\} + \lambda_2 \big\{  \log(Z_{ni}\ho{+})- \Exp [ \log(Z_{ni}\ho{+})] \big\}.  
        \] 
        Note, that $\tilde{g}_{i,n}$ only depends on the block $I_i\ho{+}$ and is
        $\mathcal{B}_{\scs (i-1)b_n+1:ib_n-\ell_n}\ho{\varepsilon}$-measurable. In particular, the $(\tilde{g}_{i,n})_{i=1,\ldots,k_n}$ are each separated by a small block of length $\ell_n$.  A standard argument based on characteristic functions and the assumption on alpha mixing may then be used to show that the weak limit of $k_{n}^{\scs -1/2}   \sum_{i=1}\ho{k_n} \tilde{g}_{i,n} $ is the same as if the $\tilde g_{i,n}$ were independent.
        
        Next, we show that  Ljapunov's condition (\cite{Bil79}, Theorem 27.3) is satisfied. By Minkowski's inequality, for any $p \in (2,2+\delta)$, 
         $C_\infty = \sup_{n \in \N} \Exp [ |\tilde{g}_{1,n} |^p ] < \infty$  by Condition \ref{cond:BerBuc}(ii) and \ref{cond2}(i). 
         Further, by stationarity and independence, we get 
        \begin{align*}
         \frac{\sum_{i=1}\ho{k_n} \Exp [ |\tilde{g}_{i,n} |^p ]}{\Var\big( \sum_{i=1}\ho{k_n} \tilde{g}_{i,n} \big)\ho{p/2}} 
         &= 
          k_n^{1-p/2} \frac{  \Exp [ |\tilde{g}_{1,n} |^p ]}{\big(\Exp [ \tilde{g}_{1,n}\ho{2} ] \big)\ho{p/2}}
         \leq 
        C_\infty  \times  k_n\ho{1-p/2} \Exp [ \tilde{g}_{1,n}\ho{2} ]\ho{-p/2}. \nonumber 
        \end{align*}
        Hence, provided $\lim_{n \to \infty} \Exp [ \tilde{g}_{1,n}\ho{2} ]$ exists, the last expression converges to $0$ and hence Ljapunov's condition is met. As a consequence, $k_{n}^{\scs -1/2}   \sum_{i=1}\ho{k_n} \tilde{g}_{i,n}$ weakly converges to a centered normal distribution with variance $\lim_{n \to \infty} \Exp [ \tilde{g}_{1,n}\ho{2} ]$.
        
        Finally, since $\lim_{n \to \infty} \Exp [ \tilde{g}_{1,n}\ho{2} ] = \lim_{n \to \infty} \Exp [ {g}_{1,n}\ho{2} ]$, it remains to be shown that
        \[
        \lim_{n \to \infty} \Exp [ {g}_{1,n}\ho{2} ] = \lambda_1^2 r(x,x) +  2\lambda_1\lambda_2 h(x) + \lambda_2^2 \pi^2/6.
        \]
        Since similar arguments as in the proof of $B_n^-=\op$ and $e_n^-=\op$ allow us to replace $I_1\ho{+}$ by $I_1$ and then $b_n$ by $n$, this in turn is a consequence of 
        \begin{align} 
        \label{eq:v1}
        & \lim_{n \to \infty} \Var \big( N_n\ho{(x)}(E)  \big) = r(x,x), \\
         \label{eq:v2}
        & \lim_{n \to \infty} \Cov \big\{ N_n\ho{(x)}(E) , \log(Z_{1:n})  \big\} = f(x), \\
         \label{eq:v3}
        &\lim_{n \to \infty} \Var \{ \log(Z_{1:n}) \} = {\pi\ho{2}}/ {6}. 
        \end{align}
        The assertion in \eqref{eq:v1} follows from Theorem 4.1 in \cite{Rob09}.  Further, since  $Z_{1:n} \wto \xi \sim \mathrm{Exp}(\theta)$ and since since  $| \log(Z_{1:n}) |\ho{2}$ is uniformly integrable by Condition \ref{cond2}(i), we have        
        \[
         \lim_{n \to \infty} \Var \{ \log(Z_{1:n}) \} = \Var \{ \log(\xi) \} = \frac{\pi\ho{2}}{6}, \nonumber 
       \]
       which is \eqref{eq:v3}. Finally, 
note that $\Exp [ N_n\ho{\scs (x)}(E)] =x$ and $\Exp [ \log(Z_{1:n}) ] \to \varphi_{(\rm C)}(\theta)$ by similar arguments as given above. As a consequence, \eqref{eq:v2} follows from $\lim_{n\to\infty} \Exp\big[ N_n\ho{\scs (x)}(E)  \log(Z_{1:n}) \big] = h(x)$. The latter in turn can be seen as follows: first,
        \begin{align} \label{eq:enz}
           \Exp\big[ N_n\ho{(x)}(E)  \log(Z_{1:n}) \big] 
          &= 
          \sum_{i=1}\ho{n} i \, \Exp\big[ \I( N_n\ho{(x)}(E)=i) \log(Z_{1:n}) \big].  
         \end{align}
         The expected value on the right-hand side can be written as
         \begin{align*}
          &\phantom{=}~ \int_{0}\ho{\infty} \Pro\big(\I(N_n\ho{(x)}(E)=i)  \log(Z_{1:n}) > y\big) \ \mathrm{d}y - \int_{-\infty}\ho{0} 1- \Pro(\I(N_n\ho{(x)}(E)=i)  \log(Z_{1:n}) > y) \ \mathrm{d}y  \nonumber\\
          &=  \int_{0}\ho{\infty} \Pro(N_n\ho{(x)}(E)=i, Z_{1:n} > e^y) \ \mathrm{d}y - \int_{-\infty}\ho{0} \Pro(N_n\ho{(x)}(E)=i) - \Pro(N_n\ho{(x)}(E)=i,  Z_{1:n} > e^y) \ \mathrm{d}y. \label{mixedEW}
        \end{align*}
        Now, 
        \begin{align*}
         \Pro(N_n\ho{(x)}(E)=i, \ Z_{1:n} > e\ho{y}) 
         &= \Pro(N_n\ho{(x)}(E)=i, \ N_n\ho{(e\ho{y})}(E)=0) \to 
         \begin{cases} 
          p_2\ho{(x,e\ho{y})}(i,0)&, \ x \geq e\ho{y} \geq 0, \\
          0 &, e\ho{y}>x \geq 0
         \end{cases} \nonumber
        \end{align*}
        and $\Pro(N_n\ho{(x)}(E)=i) \to p\ho{(x)}(i)$, see \cite{Per94} and \cite{Rob09}. By uniform integrability we obtain that the expected value on the right-hand side of \eqref{eq:enz} converges to 
        \begin{align*}  \sum_{i=1}\ho{\infty} i \ \left[ \int_{0}\ho{\infty} \I(e\ho{y} \leq x) p_2\ho{(x,e\ho{y})}(i,0) \ \mathrm{d}y - \int_{-\infty}\ho{0} p\ho{(x)}(i) - \I(e\ho{y} \leq x) p_2\ho{(x,e\ho{y})} (i,0) \ \mathrm{d}y \right]  
         = h(x).
         \end{align*}
         The proof is finished.  
\end{proof}
        
       \begin{proof}[Proof of Lemma \ref{lem:cfg-process}] For fixed $x >0$, consider the function
        \[ 
        f_n : \R \to \R, \ f_n(z) = \wk \log \Big\{ 1 + \frac{1}{\sqrt{k_n}}\Big( \frac{z}{x}+\frac{1}{\wk x} \Big) \Big\}. 
        \]
        For $z_n \to z$, one has $f_n(z_n) \to e(z)/z$.
        Hence, since $(e_n(x_1),\ldots, e_n(x_m),B_n)'$ converges in distribution to $(e(x_1),\ldots,e(x_m),B)'$ for any $x_1, \dots, x_m >0$ and $m\in\N$  by Lemma~\ref{lem:cfg-fidis}, we can apply the extended continuous mapping theorem
        (Theorem 18.11 in \citealp{Van98}) to obtain 
        $\left( W_n(x_1),\ldots,W_n(x_m),B_n \right) '
         \to \left( e(x_1)/x_1,\ldots,e(x_m)/x_m,B \right)'$ in distribution. This is the fidis-convergence needed to prove Lemma \ref{lem:cfg-process}. 
         
Asymptotic tightness of the tail empirical process $e_n$ follows from Theorem 4.1 in \cite{Rob09}. Asymptotic tightness of $B_n$ follows from its weak convergence. This implies asymptotic tightness of the vector $(e_n,B_n)$, for instance by a simple adaptation of Lemma 1.4.3 in \cite{VanWel96}.
\end{proof}

	\begin{proof}[Proof of Lemma \ref{lem:cfg-EE'}.] 
	Let $H(x)= 1- e^{-\theta x}$ be the c.d.f. of the $\Ed(\theta)$-distribution. From the proof of Lemma 9.2 in \cite{BerBuc18}, we have, for any $m\in\N$,
	\[
	\sup_{x \in [1/m, m]} |\hat H_{k_n}(x) - H(x) | = o_\Prob(1), \quad n \to \infty.
	\]
	Since	
	\[ 
	E_{n,m} -E_{n,m}' = \int_{1/m}^{m} W_n(x) \ 
	\mathrm{d} ( \hat{H}_{k_n}-H)(x),
	\] 
	the assertion follows from Lemma~\ref{lem:cfg-process}, Lemma C.8 in \cite{BerBuc17} and the continuous mapping theorem.	 	
	\end{proof}

\begin{proof}[Proof of Lemma \ref{lem:cfg-enmbn}] As a consequence of Lemma \ref{lem:cfg-EE'}, Lemma \ref{lem:cfg-process}  and the continuous mapping theorem, we have 
\begin{multline*} 
E_{n,m} + B_n =  \int_{1/m}\ho{m} W_n(x) \ \theta e\ho{-\theta x} \ \mathrm{d}x + B_n + o_\Prob(1) \\
 \wto \int_{1/m}\ho{m} \frac{e(x)}{x} \theta e\ho{-\theta x} \ \mathrm{d}x  +B = E_m' + B  \sim \Nor(0,\tau_m\ho{2}),
\end{multline*}
where the variance $\tau_m^2$ is given by
        \begin{align*}
         \tau_m\ho{2} &=  \Var \Big\{ \ \int_{1/m}\ho{m} e(x) \frac{1}{x} \theta e\ho{-\theta x} \ \mathrm{d}x \Big\} + 2 \Cov \Big\{ \ \int_{1/m}\ho{m} e(x) \frac{1}{x} \theta e\ho{-\theta x} \ \mathrm{d}x, B \Big\} + \Var(B)  \nonumber\\
         &= \theta\ho{2} \int_{1/m}\ho{m} \int_{1/m}\ho{m} r(x,y) \frac{1}{xy} e\ho{-\theta(x+y)} \ \mathrm{d}x \ \mathrm{d}y + 2 \theta \int_{1/m}\ho{m} f(x) \frac{1}{x} e\ho{-\theta x} \ \mathrm{d}x + \frac{\pi\ho{2}}{6} \nonumber
        \end{align*}
        as asserted.
\end{proof}

        \begin{proof}[Proof of Lemma \ref{lem:cfg-sigcon}] By the definition of $\tau_m\ho{2}$ in Lemma \ref{lem:cfg-enmbn}
        \begin{align}
        \label{eq:tauml}
          \lim_{m \to \infty} \tau_m\ho{2} 
         = \theta\ho{2} \int_{0}\ho{\infty} \int_{0}\ho{\infty} r(x,y) \frac{1}{xy} e\ho{-\theta(x+y)}\diff x \diff y + 2 \theta \int_{0}\ho{\infty} f(x) \frac{1}{x} e\ho{-\theta x} \ \mathrm{d}x + \frac{\pi\ho{2}}{6}.
        \end{align}
        For  $x>y$, we have $r(x,y) = \theta x \Exp \big[ \xi_1^{(y/x)} \xi_2^{(y/x)} \big]$ with $(\xi_1^{(y/x)},\xi_2^{(y/x)}) \sim \pi_2^{(y/x)}$.  Hence, applying the transformation  $z=y/x$, the first summand on the right-hand side of \eqref{eq:tauml} can be written as         \begin{align}
        	\theta^2 \int_{0}^{\infty} \int_{0}^{\infty} r(x,y) \frac{1}{xy} e^{-\theta(x+y)} \diff x \diff y
        	&= 2 \theta^3 \int_{0}^{\infty} \int_{0}^{x} \Exp \big[ \xi_1^{(y/x)} \xi_2^{(y/x)} \big] \frac{1}{y} e^{-\theta (x+y)} \diff y \diff x \nonumber\\
        	&= 2 \theta^3 \int_{0}^{\infty} \int_{0}^{1} \frac{\Exp \big[ \xi_1^{(z)} \xi_2^{(z)} \big]}{z} e^{-\theta x(1+z)} \diff z \diff x \nonumber \\
        	&= 2 \theta^2 \int_{0}^{1} \frac{\Exp \big[ \xi_1^{(z)} \xi_2^{(z)} \big]}{z(z+1)} \ \mathrm{d}z. \label{eq:int1}
       \end{align}      
       For the second summand on the right-hand side of \eqref{eq:tauml}, note that 
       \begin{align} \label{eq:p2sum1}
       \sum_{i=1}^{\infty} i p_2^{(x,e^y)}(i,0) = \Exp \big[ \xi_1^{(e^y/x)} \I(\xi_2^{(e^y/x)}=0) \big] \theta x e^{-\theta e^y},
       \end{align} 
       see  Formula (A.7) in the proof of Lemma 9.6 in \cite{BerBuc18} and 
       $\sum_{i=1}^{\infty} i p^{(x)}(i) = \Exp[ N_E^{\scs (x)} ] = x$, see \cite{Rob09}. Therefore, we can rewrite $h$ from Lemma~\ref{lem:cfg-fidis} as follows
       \begin{align*}
       	h(x) &= \int_{0}^{\infty} \I(e^y \leq x)  \Exp \big[ \xi_1^{(e^y/x)} \I(\xi_2^{(e^y/x)}=0) \big] \theta x e^{-\theta e^y} \ \mathrm{d}y \nonumber\\
        &\hspace{4cm}- \int_{-\infty}^{0} x - \I(e^y \leq x) \Exp \big[ \xi_1^{(e^y/x)} \I(\xi_2^{(e^y/x)}=0) \big] \theta x e^{-\theta e^y} \ \mathrm{d}y. \\
        &= x \int_{1/x}^{\infty} \I(z \leq 1)  \Exp \big[ \xi_1^{(z)} \I(\xi_2^{(z)}=0) \big]\theta   \frac{e^{-\theta zx}}{z} \ \mathrm{d}z \nonumber\\
        &\hspace{4cm}- x \int_{0}^{1/x} \frac1z - \I(z \leq 1) \Exp \big[ \xi_1^{(z)} \I(\xi_2^{(z)}=0) \big] \theta  \frac{e^{-\theta zx}}{z} \ \mathrm{d}z,
       \end{align*}   
       where we have used the transformation $z=e^y/x$.
       For $0 < x \leq 1$, the first integral is zero and we obtain
       \begin{align*}
       	h(x) &= -  x \int_{0}^{1} \frac1z - \Exp \big[ \xi_1^{(z)} \I(\xi_2^{(z)}=0) \big] \theta  \frac{e^{-\theta zx}}{z} \ \mathrm{d}z - x \int_1^{1/x} \frac1z \diff z \\     
	 &= -  x \int_{0}^{1} \frac1z - \Exp \big[ \xi_1^{(z)} \I(\xi_2^{(z)}=0) \big] \theta  \frac{e^{-\theta zx}}{z} \ \mathrm{d}z + x \log(x),  \nonumber  
	 \end{align*}
       while for $x >1$, 
       \begin{align*}
       	h(x) 
	= 
	x \int_{1/x}^{1} \Exp \big[ \xi_1^{(z)} \I(\xi_2^{(z)}=0) \big] \theta \frac{e^{-\theta zx}}{z} \diff z - 
	x \int_{0}^{1/x} \frac1z - \Exp \big[ \xi_1^{(z)} \I(\xi_2^{(z)}=0) \big] \theta \frac{e^{-\theta zx}}{z}  \ \mathrm{d}z. \nonumber
       \end{align*}
       As a consequence, writing $g(z) = \Exp \big[ \xi_1^{(z)} \I(\xi_2^{(z)}=0) \big]$, we obtain
       \begin{align*}
       \int_0^\infty h(x) \frac1x e^{-\theta x} \diff x
       =
         \int_0^1  \log (x) e^{-\theta x}  \diff x  
        & -  \int_0^1 e^{-\theta x} \int_0^1  \frac1z -g(z) \theta  \frac{e^{-\theta zx}}{z}  \diff z   \diff x \\
        &  + \int_1^\infty e^{-\theta x}  \int_{1/x}^{1}g(z) \theta \frac{e^{-\theta zx}}{z} \diff z \diff x \\
        &- \int_1^\infty  e^{-\theta x}  \int_{0}^{1/x} \frac1z -g(z) \theta \frac{e^{-\theta zx}}{z}  \ \mathrm{d}z \diff x.
       \end{align*}
       Next, some tedious calculations based on Fubini's theorem allow to rewrite the  sum of the last three double integrals as
       \begin{align*}
       s = \int_0^1 \frac{e^{-\theta /z} - 1}{\theta z} + \frac{g(z)}{z(1+z)} \diff z.
       \end{align*}
       Using the fact that $g(z) = \frac1\theta -  \Exp \big[ \xi_1^{(z)} \I(\xi_2^{(z)}>0) \big]$, we thus obtain
        \begin{align*}
        \int_0^\infty h(x) \frac1x e^{-\theta x} \diff x 
               &= \int_0^1  \log (z) e^{-\theta z} + \frac{e^{-\theta / z} - 1}{\theta z} + \frac1{\theta z(1+z)} -  \frac{\Exp \big[ \xi_1^{(z)} \I(\xi_2^{(z)}>0) \big]}{z(1+z)} \diff z \\
       &= \int_0^1\log (z) e^{-\theta z}  +  \frac{e^{-\theta / z}}{\theta z} - \frac1{\theta(1+z)} -  \frac{\Exp \big[ \xi_1^{(z)} \I(\xi_2^{(z)}>0) \big]}{z(1+z)} \diff z .
       \end{align*}
       Finally, one can show 
       \[
       \int_0^1 \log (z) e^{-\theta z} + \frac{e^{-\theta/z}}{\theta z} \diff z= - (\log \theta + \gamma) /\theta = \varphi_{(\rm C)}(\theta) / \theta,
       \] such that, assembling terms and recalling $f(x) = h(x) - x \varphi_{(\rm C)}(\theta)$,
       \begin{align} \label{eq:int2}
       \int_0^1 f(x) \frac1x e^{-\theta x} \diff x 
       &=
         \int_0^1 h(x) \frac1x e^{-\theta x} \diff x -  \varphi_{(\rm C)}(\theta)  \int_0^\infty e^{-\theta x} \diff x \nonumber \\
        &= - \log(2) / \theta -  \int_0^1 \frac{\Exp \big[ \xi_1^{(z)} \I(\xi_2^{(z)}>0) \big]}{z(1+z)} \diff z.
       \end{align}
       The lemma is now an immediate consequence of \eqref{eq:tauml}, \eqref{eq:int1} and \eqref{eq:int2}.
       \end{proof}

       \begin{proof}[Proof of Lemma \ref{lem:cfg-limsup}] 
By Lemma~\ref{lem:cfg-EE'}, it suffices to show the assertion with $E_{n,m}$ replaced by $E_{n,m}'$.
Define $\tilde{e}_{n}(x) := e_n(x)+k_n\ho{\scs-1/2}$, such that, by Condition \ref{cond_max}, we have 
\[
\max_{Z_{ni} \geq c} \left| \frac{\tilde{e}_n(Z_{ni})}{Z_{ni}\sqrt{k_n}} \right| = o_\Prob(1)
\]
for any constant $c >0$. Fix $m\in\N$. By the previous display, for any $\varepsilon >0$, the event 
\[ 
B_n = B_n(m,\varepsilon) = 
\Big\{ \max_{Z_{ni} \geq m} \Big| \frac{\tilde{e}_n(Z_{ni})}{Z_{ni}\sqrt{k_n}} \Big| \leq \varepsilon \Big\} 
\]
satisfies $\Pro(B_n) \to 1$. Next,         
\begin{align}
          |E_{n,m}-E_n| 
          & \leq \Big| \intnu \log \Big( 1+ \frac{\tilde{e}_n(x)}{x\sqrt{k_n}} \Big) \sqrt{k_n} \ \I_{(0,1/m]}(x) \ 
          \mathrm{d}\hat{H}_{k_n}(x) \Big| \nonumber\\
          & \hspace{2cm} + \Big| \intnu \log \Big( 1+ \frac{\tilde{e}_n(x)}{x\sqrt{k_n}} \Big) \sqrt{k_n} \ \I_{[m,\infty)}(x) \ 
          \mathrm{d}\hat{H}_{k_n}(x) \Big| \nonumber\\
          &=: |V_{n1}| + |V_{n2}|, \nonumber
         \end{align}
         such that 
         \begin{equation} 
         |E_{n,m}-E_n| = |E_{n,m}-E_n| \I_{B_n} + o_\Prob(1) \leq |V_{n1}| + |V_{n2}| \I_{B_n} + \op. \label{eq:enmen}
         \end{equation}
         We begin by treating the term $|V_{n2}| \I_{B_n}$. Since $\log(1+x) = \int_0^1 x/(1+sx) \ \mathrm{d}s$ for any $x >-1$,  we have
         \begin{align}
          V_{n2} \I_{B_n} &= \intnu \frac{\tilde{e}_n(x)}{x} \intne \frac{1}{1+s \frac{\tilde{e}_n(x)}{x\sqrt{k_n}}}
          \ \mathrm{d}s \ 
          \ \I(x \geq m) \ \mathrm{d}\hat{H}_{k_n}(x) \, \I_{B_n} \nonumber\\
          &= \frac{1}{k_n}\sum_{i=1}\ho{k_n} \frac{\tilde{e}_n(Z_{ni})}{Z_{ni}} \I(Z_{ni} \geq m) 
          \intne \frac{1}{1+s \frac{\tilde{e}_n(Z_{ni})}{Z_{ni}\sqrt{k_n}}} \ \mathrm{d}s \, \I_{B_n} \nonumber\\
          &= k_n\ho{-3/2} \sum_{i=1}\ho{k_n} \frac{1}{Z_{ni}} \I(Z_{ni} \geq m) 
          \intne \frac{1}{1+s \frac{\tilde{e}_n(Z_{ni})}{Z_{ni}\sqrt{k_n}}} \ \mathrm{d}s \
          \Big\{\sum_{t=1}^n f(U_t,Z_{ni}) + 1 \Big\}  \, \I_{B_n}, \nonumber
         \end{align}
         where 
         \begin{align} \label{eq:fuz}
         f(U_t,Z_{ni}) = \I(U_t > 1-Z_{ni}/{b_n} ) - Z_{ni}/{b_n}. 
          \end{align}
         For given $\varepsilon \in (0,c_1 \wedge c_2)$ with $c_j$ as in Condition~\ref{cond:BerBuc}, let $C_n =C_n(\varepsilon)$ denote the event $\{ \min_{i=1,\ldots,k_n} N_{ni} > 1- \varepsilon/2\}
         = \{ \max_{i=1,\ldots,k_n}Z_{ni} < \varepsilon b_n/{2} \}$, which satisfies $\Pro(C_n) \to 1$ 
         by Condition \ref{cond:BerBuc}(v). 
         Hence, we can write 
         $V_{n2} \I_{B_n} = \bar{V}_{n2} \I_{C_n} + \op$, where 
         \[ 
         \bar{V}_{n2} = k_n\ho{-3/2} \sum_{i=1}\ho{k_n} \frac{1}{Z_{ni}} \I(\varepsilon b_n/2 > Z_{ni} \geq m) 
          \intne \frac{1}{1+s \frac{\tilde{e}_n(Z_{ni})}{Z_{ni}\sqrt{k_n}}} \ \mathrm{d}s  
          \Big\{ \sum_{t=1}^n f(U_t,Z_{ni}) + 1 \Big\}\, \I_{B_n}. 
          \]
          We obtain 
          \begin{align*}
           |\bar{V}_{n2}| 
          & \leq \frac{1}{m} k_n\ho{-3/2} \sum_{i=1}\ho{k_n} \I(\varepsilon b_n/2 > Z_{ni} \geq m) 
          \intne \frac{1}{\left| 1+s \frac{\tilde{e}_n(Z_{ni})}{Z_{ni}\sqrt{k_n}}\right|} \ \mathrm{d}s \,
          \Big\{ \Big| \sum_{t=1}^n f(U_t,Z_{ni}) \Big| + 1 \Big\} \, \I_{B_n}. 
          \end{align*}
          On the event $B_n$ the integral over $s$ can be bounded as follows
          \begin{align}
           & \intne \frac{1}{\left| 1+s \frac{\tilde{e}_n(Z_{ni})}{Z_{ni}\sqrt{k_n}}\right|} \ \mathrm{d}s \ \I_{B_n}
           \leq \intne \frac{1}{1-s\varepsilon} \
           \mathrm{d}s  \ \I_{B_n} \leq \frac{1}{1-\varepsilon}. \nonumber
          \end{align}
          The previous two displays imply that $|\bar{V}_{n2}|$ is bounded by 
          \begin{align} 
          & \frac{1}{m} \frac{1}{1-\varepsilon} \ k_n\ho{-3/2} \sum_{i=1}\ho{k_n} \I(\varepsilon b_n/2 > Z_{ni} \geq m) 
         \Big\{ \Big| \sum_{t=1}^n f(U_t,Z_{ni}) \Big| + 1 \Big\} \nonumber\\
          = & \frac{1}{m} \frac{1}{1-\varepsilon} \ k_n\ho{-3/2} \sum_{i=1}\ho{k_n} \I(\varepsilon b_n/2 > Z_{ni} \geq m) 
          \Big| \sum_{t=1}^n f(U_t,Z_{ni}) \Big| + O_\Prob(k_n^{-1/2}). \nonumber
          \end{align} 
          The upper bound  can now be treated exactly as in the proof of Lemma 9.1 in \cite{BerBuc18}, finally yielding
          \begin{align} 
          \label{eq:vn2}
          \lim_{m \to \infty} \limsup_{n \to \infty} 
          \Pro(|{V}_{n2} \I_{B_n} | > \delta)=0.
          \end{align}
          
          It remains to treat $|V_{n1}|$. 
Write
         \begin{align}
         |V_{n1}| &\le T_{n}(0,dk_n^{-1}) + T_n(dk_n^{-1}, dk_n^{-\mu}) + T_n(dk_n^{-\mu}, 1/m) \nonumber \\
         &=: T_{n1} + T_{n2} + T_{n3}, \label{eq:vn1}
         \end{align}
         where, for some constant $d>0$ and $\mu=\mu_d$ determined below,       
                  \begin{align*}
         	T_{n}(a,b) &= \sqrt{k_n} \int_0^\infty \ind(x \in (a,b]) \Big| \log \Big( 1+ \frac{\tilde{e}_n(x)}{x\sqrt{k_n}} 
         	\Big) \Big| \ \mathrm{d}\hat{H}_{k_n}(x).
         \end{align*}  
         We start by covering the term $T_{n1}=T_{n}(0,dk_n^{-1})$ and determining the constants $d$ and $\mu$. Note that for the event
         $J_n = \{ \min_{i=1,\ldots,k_n} Z_{ni} > d k_n\ho{-1} \}$ one has 
         \begin{align*}
         \Pro(J_n) & = \Pro \big(k_n \min_{i=1,\ldots,k_n}Z_{ni}  > d \big) = \Pro\big(n \ \big(  1- \max_{i=1,\ldots,k_n} N_{ni} \big) > d\big) 
          = \Pro(Z_{1:n} > d) \to e^{-d\theta}.
         \end{align*}
         Then, 
         \begin{align}
          \Pro(T_{n1}>\delta) &= \Pro (T_{n1} \I_{J_n} + T_{n1} \I_{J_n\ho{c}} > \delta) \nonumber\\
          &\leq \Pro (T_{n1} \I_{J_n} > \delta/2) + \Pro (T_{n1} \I_{J_n\ho{c}} > \delta/2) \nonumber\\ 
          &\leq \Pro(J_n\ho{c}) \to 1-\exp(-d \theta). \nonumber
         \end{align}
         Hence, for any given $\eps >0$ we can choose $d = d(\eps) < -\log(1-\eps)/\theta$, such that 
         \begin{align} \label{eq:tn1}
         \limsup_{n \to \infty} \Pro(T_{n1}>\delta) \leq \limsup_{n \to \infty} \Pro(J_n\ho{c}) = 1- \exp(-d \theta) < \eps.
         \end{align}
         
        Now, choose $\mu=\mu_d \in (1/2, 1/\{ 2(1-\tau) \})$ from Condition~\ref{cond_ewerte}, where $\tau\in(0,1/2)$ is from Condition~\ref{cond_gewKonv}. Next, consider $T_{n3} = T_n(dk_n^{-\mu}, 1/m)$ and note that, for $x \in (d k_n^{-\mu},1/m]$, we have 
         \begin{equation}
         	\left| \frac{\tilde{e}_n(x)}{x\sqrt{k_n}} \right| = 
         	\left| \frac{\tilde{e}_n(x)}{x^{\tau}} \right| \ \frac{1}{x^{1-\tau}\sqrt{k_n}} \leq 
         	\frac{1}{d\ho{1-\tau}} \ \left| \frac{\tilde{e}_n(x)}{x^{\tau}} \right| \ k_n^{\mu(1-\tau)-1/2} =o_\Prob(1) \nonumber
         \end{equation} 
         uniformly in $x$, by Condition \ref{cond_gewKonv}. As a consequence, the event 
         \[
         D_n = \Big\{ \Big|
         \frac{\tilde{e}_n(x)}{x \sqrt{k_n}} \Big| \leq \frac{1}{2}  \Big\}
         \] 
         satisfies $\ind_{D_n^c} = o_\Prob(1)$, whence, recalling that $x/(1+x) \le \log(1+x) \le x$ for any $x>-1$, we have
         \begin{align}
         	T_{n3} &= \sqrt{k_n} \int_{(d k_n^{-\mu},1/m]} \Big| \log \Big( 1+ \frac{\tilde{e}_n(x)}{x\sqrt{k_n}} 
         	\Big) \Big| \I_{D_n} 
         	\ \mathrm{d}\hat{H}_{k_n}(x) + \op \nonumber\\
         	& \leq 
	\int_{(d k_n^{-\mu},1/m]} \max\Big\{ \Big| \frac{\tilde{e}_n(x)}{x} \Big|,
         	\Big| \frac{\tilde{e}_n(x)}{x} \Big| \Big(1+ \frac{\tilde{e}_n(x)}{x\sqrt{k_n}}\Big)\ho{-1} \Big\} 
         	 \I_{D_n} \ \mathrm{d}\hat{H}_{k_n}(x) +\op \nonumber\\             
                & \leq 2 \int_{(d k_n^{-\mu},1/m]}  \left| \frac{\tilde{e}_n(x)}{x^{\tau}} \right| 
         	\frac{1}{x^{1-\tau}} \I_{D_n} \ \mathrm{d}\hat{H}_{k_n}(x) +\op. \nonumber         
         \end{align}
         By Lemma \ref{hilfslem Konv}, Condition \ref{cond_gewKonv} and the continuous mapping theorem, the last expression converges weakly to
         \begin{align*}
          T_3 (m) = 2 \ \int_{0}^{1/m} \ \left| \frac{e(x)}{x^{\tau}} \right| \frac{1}{x^{1-\tau}} \ \mathrm{d}H(x). 
         \end{align*}
         As a consequence,
         \begin{align} \label{eq:tn3}
         \lim_{m\to\infty} \limsup_{n\to\infty} \Prob (T_{n3} \ge \delta) \le \lim_{m\to\infty} \Prob(T_3(m) >\delta) = 0.
         \end{align}

         Finally, regarding $T_{n2}$, note that,
         for $x \in (dk_n\ho{-1},dk_n\ho{-\mu})$, 
         \begin{align*}
         \frac{\tilde{e}_n(x)}{x\sqrt{k_n}} &\leq 
         \frac1x \Big( \frac1{k_n} + \frac1{k_n} \sum_{i=1}^n \ind(U_i > 1-x/b_n) \Big) \le \frac{n+1}{d}, \\
         \frac{\tilde{e}_n(x)}{x\sqrt{k_n}}  &\ge 
         \frac1x \Big( \frac1{k_n} -\frac1{k_n} \sum_{i=1}^n x/b_n\Big) \ge \frac{1}{dk_n^{1-\mu}} - 1,
         \end{align*}
         which implies
         \begin{align*}
           \Big| \log \Big( 1+ \frac{\tilde{e}_n(x)}{x\sqrt{k_n}} \Big) \Big| 
          &= \log \Big( 1+ \frac{\tilde{e}_n(x)}{x\sqrt{k_n}}\Big) \I\Big(\frac{\tilde{e}_n(x)}{x\sqrt{k_n}} >0\Big) - \log \Big( 1+ \frac{\tilde{e}_n(x)}{x\sqrt{k_n}}\Big) \I\Big(\frac{\tilde{e}_n(x)}{x\sqrt{k_n}} <0\Big) \\
          & \leq \log \left( (n+1)d\ho{-1}+1 \right) + \log \left( d k_n\ho{1-\mu} \right) \\
          & \lesssim \log(n).
         \end{align*}
          As a consequence, 
         the term $T_{n2}=T_n(dk_n^{-1}, dk_n^{-\mu}) $ can be bounded as follows
         \begin{align}
          T_{n2} 
          & \lesssim \log(n) \ \sqrt{k_n} \int_{(d k_n^{-1},d k_n^{-\mu}]} \ \mathrm{d}\hat{H}_{k_n}(x) 
          \nonumber
          = \frac{\log(n)}{\sqrt{k_n}} \sum_{i=1}\ho{k_n} \I \big(Z_{ni} \in (d k_n\ho{-1},d k_n\ho{-\mu}] \big). \nonumber
         \end{align}
         Hence, by Condition \ref{cond_ewerte}, 
         \begin{align}
          E\left[ T_{n2} \right] 
          & \lesssim 
          \log(n) \ \sqrt{k_n}  \Pro(Z_{n1} < d k_n\ho{-\mu}) \nonumber\\
          &= \log(n) \ \sqrt{k_n}  \{1-\exp(-\theta d k_n\ho{-\mu})\} + o(1) \nonumber\\
          &= \theta d \log(n) \ k_n\ho{1/2-\mu} \{ 1+o(1) \} + o(1) \nonumber\\
          & = O( \log(k_n) k_n^{1/2-\mu}) =o(1) 
          \label{eq:tn2},
         \end{align}
         where the last line follows from $\log n =\log k_n + \log b_n \lesssim (1+q) \log k_n$ by Condition~\ref{cond:bnkn2}.            

The assertion follows from \eqref{eq:enmen}, combined with \eqref{eq:vn2}, \eqref{eq:vn1}, \eqref{eq:tn1}, \eqref{eq:tn3} and \eqref{eq:tn2}.
\end{proof}

		\subsection{Auxiliary lemmas -- Sliding blocks}
\label{subsec:psb}

Throughout this section, we assume that Condition \ref{cond:BerBuc}, \ref{cond2}(i) and \ref{cond3}(i) are met.

	\begin{lem} \label{lem:cfg-fidis2}
	For any $x_1,\ldots,x_m \in [0,\infty)$ and $m \in \N$, we have
	 \[
	 (e_n(x_1),\ldots,e_n(x_m),B_n^\slb)' 
			\wto (e(x_1),\ldots,e(x_m),B^\slb)', 
			\] where $ (e(x_1),\ldots,e(x_m),B^\slb)' \sim \Nor_{m+1} (0,\Sigma_{m+1}^\slb)$ with 
			\[
			\Sigma_{m+1}^\slb= 
			\begin{pmatrix}
		r(x_1,x_1) & \dots & r(x_1,x_m) & f(x_1) \\
		\vdots & \ddots& \vdots & \vdots \\
		r(x_m,x_1) & \dots & r(x_m,x_m) & f(x_m) \\
		f(x_1) & \dots & f(x_m) & 8 \log(2)-4
		\end{pmatrix}.
			\] 
			Here, the functions $r$ and $f$ are defined as in Lemma \ref{lem:cfg-fidis}.
	\end{lem}
	
	\begin{lem}\label{lem:cfg-process2}
		For any $m \in \N$, we have
		\[ 
		\left\{ (W_n(x),B_n^\slb)' \right\}_{x \in [1/m,m]}  
		\wto \left\{ \left( \frac{e(x)}{x},B^\slb \right)' \right\}_{x \in [1/m,m]} \ \textrm{ in } D([1/m,m]) \times \R, 
		\]
		where $(e,B^\slb)'$ is a centered Gaussian process with continuous sample paths and with covariance
		functional as specified in Lemma \ref{lem:cfg-fidis2}. 
	\end{lem}

		\begin{lem} \label{lem:cfg-EE'2}
		For any $m\in\N$, we have
		\[ 
		E_{n,m}^{\slb} = E_{n,m}'+ \op \quad \text{ as } n \to \infty, 
		\] 
		where 
		$
		E_{n,m}'= \int_{1/m}^{m} W_n(x) \theta e^{-\theta x} \ \mathrm{d}x
		$ is as in Lemma~\ref{lem:cfg-EE'}.
		\end{lem}

	\begin{lem} \label{lem:cfg-enmbn2}
		For any $m \in \N$,  we have
		\[ 
		E_{n,m}^\slb + B_n^\slb \wto E_m'+B^\slb \sim \Nor(0,\tau_{\slb,m}^2) \ \textrm{ as } n \to \infty, 
		\] 
		where, with $r$ and $f$ defined as in Lemma~\ref{lem:cfg-fidis}, 
		\[ 
		\tau_{\slb,m}^2= \theta\ho{2} \int_{1/m}\ho{m} 
		\int_{1/m}\ho{m} r(x,y) \frac{1}{xy} e\ho{-\theta(x+y)} \ \mathrm{d}x \mathrm{d}y + 2 \theta \int_{1/m}\ho{m} f(x) \frac{1}{x} e\ho{-\theta x} \ \mathrm{d}x +  8\log(2)-4. \]
	\end{lem}

	\begin{lem} \label{lem:cfg-sigcon2}
		As $m \to \infty$, $\tau_{\slb,m}^2 \to \sigma_{\slb, (\rm C)}^2/\theta^2$, where $ \sigma_{\slb, (\rm C)}^2$ is specified in Theorem~\ref{theo:cfg}.
		\end{lem}

	\begin{lem} \label{lem:cfg-limsup2}
		If in addition, Condition \ref{cond_CFG} holds, then, for all $\delta > 0$, 
		\[ 
		\lim_{m \to \infty} \limsup_{n \to \infty} P \! \left( |E_{n,m}^{\slb}-E_n^\slb | > \delta \right) = 0. 
		\]
	\end{lem}

\begin{proof}[Proof of Lemma \ref{lem:cfg-fidis2}] 
As in the proof of Lemma \ref{lem:cfg-fidis} we only show joint weak convergence of $(e_n(x),B_n\ho{\slb})$ for some fixed $x>0$;
the general case can be shown analogously.
For given $\varepsilon \in (0,c_1 \wedge c_2)$ let $A_n' = \{\min_{t=1,\ldots,n-b_n+1} N_{nt}^{\slb} > 1- \varepsilon\}$, 
such that $\Pro(A_n) \to 1$ by Condition \ref{cond:BerBuc}(v). 
By the Cram\'{e}r-Wold device, it suffices to prove weak convergence of
\begin{multline*}
	\lambda_1 e_n(x) + \lambda_2 B_{n}^{\slb} 
	= \sum_{j=1}^{k_n-1} \sum_{s \in I_j} \Big[ \frac{\lambda_1}{\sqrt{k_n}} \Big\{ \iu - \frac{x}{b_n} \Big\} \\
	+ \frac{\lambda_2 \sqrt{k_n}}{n-b_n+1} \Big\{ \log(Z_{ns}^{\slb})-E[\log(Z_{ns}^{\slb})] \Big\}  \Big] + \op, \nonumber
\end{multline*}
 for some arbitrary $\lambda_1,\lambda_2 \in \R$, where the negligible term stems from omitting a negligible number of summands. 
 
We are going to apply a big block-small block argument, based on a suitable `blocking of blocks' to take care of the serial dependence introduced through the use of sliding blocks. For that purpose, let $k_n\ho{\ast} < k_n$ be an integer sequence with $k_n\ho{\ast} \to \infty$ and $k_n\ho{\ast} = 
o(k_n\ho{\scs \delta/(2(1+\delta))})$,
where $\delta$ is from Condition \ref{cond:BerBuc}(ii). For $q_n\ho{\ast} = \lfloor k_n/(k_n\ho{\ast}+2)\rfloor$ and 
$j=1,\ldots,q_n\ho{\ast}$, define
\[ J_j\ho{+} = \bigcup_{i=(j-1)(k_n\ho{\ast}+2)+1}\ho{j(k_n\ho{\ast}+2)-2} I_i \ \ \textrm{ and } \ \ 
J_j\ho{-} = I_{j(k_n\ho{\ast}+2)-1} \cup I_{j(k_n\ho{\ast}+2)}. \]
Thus we have $q_n\ho{\ast}$ big blocks $J_j\ho{+}$ of size $k_n\ho{\ast}b_n$, which are separated by a small block 
$J_j\ho{-}$ of size $2b_n$, just as in the construction in the proof of Lemma 10.3  in \cite{BerBuc18}.
Consequently, we have $\lambda_1 e_n(x) + \lambda_2 B_{n}^{\slb} = L_n^{+}+ L_n^{-}+ \op$, where 
\[
L_n^{\pm} = \frac{1}{\sqrt{q_n^{\ast}}} \sum_{j=1}^{q_n^{\ast}} W_{nj}^{\pm} 
\] 
with 
\[ 
W_{nj}^{\pm} = \sqrt{\frac{q_n^{\ast}}{k_n}} \sum_{s \in J_j^{\pm}} \lambda_1 \Big\{ \iu - \frac{x}{b_n}
\Big\} + \frac{\lambda_2 n}{n-b_n+1} \frac{1}{b_n} \Big\{ \log(Z_{ns}^{\slb})-E[\log(Z_{ns}^{\slb})] \Big\} 
\] 
for $j=1,\ldots,q_n^{\ast}$. In the following, we show that, on the one hand, $L_n^-\I_{A_n'} = \op$ and that, on
the other hand, $L_n^+ \I_{A_n'}$ converges to the claimed normal distribution. First, we cover $L_n^- \I_{A_n'}$. 
As in the proof of Lemma \ref{lem:cfg-fidis}, we have 
\[
Z_{ns}^{\slb} = b_n \Big( 1- \max_{t=s,\ldots,s+b_n-1}U_t \Big) = b_n 
\Big( 1- \max_{t=s,\ldots,s+b_n-1}U_t^{\varepsilon} \Big) =: Z_{ns}^{\varepsilon,\slb} 
\] 
on the event $A_n'$, where $U_t^{\varepsilon} =U_t \I(U_t >1- \varepsilon)$. Hence, we can write $L_n^- \I_{A_n'} = \tilde{L}_n^- \I_{A_n'} 
+ \op = \tilde{L}_n^- + \op$ with 
\[ 
\tilde{L}_n^- = \frac{1}{\sqrt{q_n^{\ast}}} \sum_{j=1}^{q_n^{\ast}} 
W_{nj}^{\varepsilon -}, 
\]
where 
\[ 
W_{nj}^{\varepsilon -} =  \sqrt{\frac{q_n^{\ast}}{k_n}} \sum_{s \in J_j^{-}} \lambda_1 \Big\{ \iue - \frac{x}{b_n} \Big\} + \frac{\lambda_2 n}{n-b_n+1} \frac{1}{b_n} \Big\{ \log(Z_{ns}^{\varepsilon,\slb})-E[\log(Z_{ns}^{\varepsilon,\slb})] \Big\}.
\]
We proceed by showing that $\Var[\tilde{L}_n^-] = o(1)$. By stationarity, one has
\begin{align}
	\Var[\tilde{L}_n^-] 
	&= \Var[W_{n1}^{\varepsilon -}] + \frac{2}{q_n^{\ast}} \sum_{j=1}^{q_n^{\ast}} (q_n^{\ast}-j) 
	\Cov\big(W_{n1}^{\varepsilon -},W_{n,j+1}^{\varepsilon -}\big), \nonumber
\end{align}
which is bounded by $3 \Var[W_{n1}^{\varepsilon -}] + 2 \sum_{j=2}^{q_n^{\ast}} |\Cov\big(W_{n1}^{\varepsilon -},W_{n,j+1}^{\varepsilon -}\big)|$
in absolute value.
First, we show $\Var[W_{n1}\ho{\varepsilon -}]=o(1)$, for which it suffices to show that $||W_{n1}\ho{\varepsilon -}||_p=o(1)$ for some $p \in (2,2+\delta)$.
By Minkowski's inequality, one has 
\begin{align} \label{eq:wn1}
 ||W_{n1}\ho{\varepsilon -}||_p 
 & \leq 
 2 \sqrt{\frac{q_n\ho{\ast}}{k_n}}\left[ |\lambda_1|\ ||N_{b_n}\ho{(x)}(E)||_p + |\lambda_2|
 \ ||\log(Z_{n1}\ho{\varepsilon,\slb}) -\Exp[ \log(Z_{n1}\ho{\varepsilon,\slb})]||_p \right] \\
 &=O(\sqrt{q_n^*/k_n}) = o(1)  \nonumber 
\end{align}
by Condition \ref{cond:BerBuc}(ii) and \ref{cond2}(i).
It remains to treat the sum over the covariances. Since $W_{nj}\ho{\varepsilon -}$ is $\mathcal{B}\ho{\scs \varepsilon}_{ \{(j(k_n\ho{\ast}+2)-2)b_n+1\}:\{ j(k_n\ho{\ast} +2 )b_n \} }$- measurable, we may apply Lemma 3.11
in \cite{DehPhi02} to obtain
\begin{align}
 |\Cov(W_{n1}\ho{\varepsilon -},W_{n,j+1}\ho{\varepsilon -})| 
 & \leq 10 \ ||W_{n1}\ho{\varepsilon -}||_p\ho{2} \ \alpha_{c_2} (jk_n\ho{\ast} b_n)\ho{1-2/p}. \nonumber
\end{align}
By Condition \ref{cond:BerBuc}(iii), the sum $\sum_{j=2}\ho{q_n\ho{\ast}} \alpha_{c_2} 
(jk_n\ho{\ast}b_n)\ho{1-2/p}$ converges to zero, hence $||W_{n1}\ho{\varepsilon -}||_p=o(1)$ as asserted. 

Let us now treat the term $L_n\ho{+}\I_{A_n'}$ and show weak convergence to the asserted normal distribution. 
One can write 
\[ 
L_n^+ \I_{A_n'} = \frac{1}{\sqrt{q_n^{\ast}}} \sum_{j=1} ^{q_n^{\ast}} \tilde{W}_{nj}^{+} + \op,  \qquad \tilde{W}_{nj}^{+} = W_{nj}^{+} \I\big( \max_{t \in J_j^{+}} Z_{nt}^{\slb} < \varepsilon b_n \big).
\] 
A standard argument based on characteristic functions shows that the weak limit of ${q_n^{\ast}}^{\scs -1/2} \sum_{j=1} ^{q_n^{\ast}} \tilde{W}_{nj}^{+}$ is the same as if the summands were independent. By arguments as before, we may also pass back to an independent sample  ${W}_{nj}^{+}$, $j=1, \dots, q_n^\ast$. The assertion then follows from Ljapunov's central limit theorem, once we have shown the Ljapunov condition.

For that  purpose, note that
$
||W_{nj}\ho{+}||_{2+\delta}  = O(\sqrt{q_n^* k_n}) = O(\sqrt {k_n^*})
$ by similar arguments as in \eqref{eq:wn1}
such that $E[|W_{nj}\ho{+}|\ho{2+\delta}] =O({k_n^*}^{\scs (2+\delta)/2})$.
As a consequence, 
\[ 
\frac{\sum_{j=1}\ho{q_n\ho{\ast}}E[|W_{nj}\ho{+}|\ho{2+\delta}]}
{\Var\left[ \sum_{j=1}\ho{q_n\ho{\ast}} W_{nj}\ho{+} \right]\ho{\frac{2+\delta}{2}} } 
 ={q_n\ho{\ast}}^{ - \frac{\delta}{2}}  \frac{ E[|W_{n1}\ho{+}|\ho{2+\delta}] }{E[|W_{n1}\ho{+}|\ho{2}]\ho{\frac{2+\delta}{2}}} 
 = O(k_n\ho{-\delta/2} {k_n\ho{\ast}}^{1+\delta}) = o(1), \]
since $k_n\ho{\ast} = o(k_n\ho{\delta/(2(1+\delta))})$ by construction and provided that the limit of $E[|W_{n1}\ho{+}|\ho{2}]$ exists. If it does, we can conclude 
that $L_n\ho{+} \wto \Nor(0,\lim_{n\to \infty}E[|W_{n1}\ho{+}|\ho{2}])$. 
and it suffices to show that
\[ 
\lim_{n \to \infty} E[|W_{n1}\ho{+}|\ho{2}] = \lambda_1\ho{2} r(x,x) + 2 \lambda_1 \lambda_2 f(x) + 
\lambda_2\ho{2} \{8 \log(2) -4\}.  
\]
To this, note that $W_{n1}\ho{+} = \lambda_1 
e_{n\ho{\ast}}(x) + \lambda_2 B_{n\ho{\ast}}\ho{\slb} + \op$, where $e_{n\ho{\ast}}$ and $B_{n\ho{\ast}}\ho{\slb}$ are defined
as $e_n$ and $B_n\ho{\slb}$ with $n$ replaced by $n\ho{\ast} = k_n\ho{\ast}b_n$ and $k_n$ by $k_n\ho{\ast}$; and 
our general conditions still hold with this replacement. The result follows from Lemma \ref{lem:cov-cfg1} and Lemma~\ref{lem:varcfg1}  and the proof of Theorem 4.1 in \cite{Rob09}.
\end{proof}

\begin{proof}[Proof of Lemma~\ref{lem:cfg-process2}] Up to notation, the proof is exactly the same as the one of Lemma~\ref{lem:cfg-process} in the disjoint blocks case.    
\end{proof}

	\begin{proof}[Proof of Lemma \ref{lem:cfg-EE'2}] 
The result follows immediately from the argument in the proof of Lemma \ref{lem:cfg-EE'} and the proof of Lemma 10.2 in \cite{BerBuc18}.
\end{proof}

    \begin{proof}[Proof of Lemma \ref{lem:cfg-enmbn2}] Up to notation, the proof is exactly the same as the one of Lemma~\ref{lem:cfg-enmbn} in the disjoint blocks case.    
\end{proof}

    \begin{proof}[Proof of Lemma \ref{lem:cfg-sigcon2}] 
    By the definition of $\tau_m^2$ and $\tau_{\slb,m}\ho{2}$ in Lemma \ref{lem:cfg-enmbn} and \ref{lem:cfg-enmbn2}, we have
        \begin{align*}
        \tau_{\slb, m}\ho{2} = \tau_m^{2} - \pi^2/6 + 8 \log(2) -4.
        \end{align*}
        Hence, by the proof of Lemma~\ref{lem:cfg-sigcon} and the definition of $\sigma_{\slb, \rm C}^2$ in Theorem~\ref{theo:cfg}, 
        \[
        \lim_{m\to\infty}\tau_{\slb, m}\ho{2} =  \sigma_{\djb, \rm C}^2 / \theta^2 - \pi^2/6 + 8 \log 2 - 4 =  \sigma_{\slb, \rm C}^2 / \theta^2. \qedhere         
        \]
\end{proof}

     \begin{proof}[Proof of Lemma \ref{lem:cfg-limsup2}] 
     The proof is similar to the one of Lemma~\ref{lem:cfg-limsup}, which is why we keep it short. Write $|E_{n,m}\ho{\slb}-E_n\ho{\slb}| \leq |V_{n1}|+|V_{n2}|$ with 
     \begin{align}
      V_{n1} &= \intnu \log \bigg( 1+ \frac{\tilde{e}_n(x)}{x \sqrt{k_n}} \bigg)
      \sqrt{k_n} \ \I_{(0,1/m]}(x) \ \mathrm{d}\hat{H}_{n}\ho{\slb}(x), \nonumber\\
      V_{n2} &= \intnu \log \bigg( 1+ \frac{\tilde{e}_n(x)}{x \sqrt{k_n}} \bigg)
      \sqrt{k_n} \ \I_{[m,\infty)}(x) \ \mathrm{d}\hat{H}_{n}\ho{\slb}(x), \nonumber
     \end{align}
     where $\tilde{e}_n(x) = e_n(x)+k_n\ho{-1/2}$. For some $\varepsilon >0$ define the event 
     \[ 
     B_n = \Big\{ \max_{Z_{ni}\ho{\slb} \geq m} 
     \Big| \frac{\tilde{e}_n(Z_{ni}\ho{\slb})}{Z_{ni}\ho{\slb} \sqrt{k_n}} \Big| \leq \varepsilon \Big\}, \]
     such that $\Pro(B_n) \to 1$ by Condition \ref{cond_max}. As in the proof of Lemma \ref{lem:cfg-limsup}, with $f$ defined in \eqref{eq:fuz},
     we can write 
     \begin{multline*}
     V_{n2} \I_{B_n} 
     = k_n\ho{-3/2} \sum_{i=1}\ho{k_n-1} \sum_{w \in I_i} 
     \frac{1}{Z_{nw}\ho{\slb}} \I(Z_{nw}\ho{\slb} \geq m)
     \intne \frac{1}{1+s \ \frac{\tilde{e}_n(Z_{nw}\ho{\slb})}{Z_{nw}\ho{\slb} \sqrt{k_n}}} \ \mathrm{d}s  \\
     \times  b_n\ho{-1} 
     \Big\{ \sum_{j=1}\ho{k_n} \sum_{t \in I_j} f(U_t,Z_{nw}\ho{\slb})+1 \Big\} \, \I_{B_n} + \op \nonumber.
     \end{multline*}
     By Condition \ref{cond:BerBuc}(v), $\Pro(C_n) \to 1$ where $C_n = \big\{ \min_{i=1,\ldots,n-b_n+1} N_{ni}^\slb >1-\varepsilon/2 \big\}$. Hence, $V_{n2} \I_{B_n} = \bar{V}_{n2} \I_{B_n} \I_{C_n} + \op$, where 
     \begin{multline*}
     \bar{V}_{n2} = k_n\ho{-3/2} \sum_{i=1}\ho{k_n-1} \sum_{w \in I_i} 
     \frac{1}{Z_{nw}\ho{\slb}} \I(\varepsilon b_n/2 >Z_{nw}\ho{\slb} \geq m)
     \intne \frac{1}{1+s \ \frac{\tilde{e}_n(Z_{nw}\ho{\slb})}{Z_{nw}\ho{\slb} \sqrt{k_n}}} \ \mathrm{d}s \\
     \times  b_n\ho{-1} 
     \Big\{ \sum_{j=1}\ho{k_n} \sum_{t \in I_j} f(U_t,Z_{nw}\ho{\slb})+1 \Big\}, 
     \end{multline*}
     such that $\bar{V}_{n2}$ can be bounded as in the proof of Lemma \ref{lem:cfg-limsup} as follows
     \begin{align}
      |\bar{V}_{n2} \I_{B_n} | & \leq \frac{1}{m} \frac{1}{1-\varepsilon} k_n\ho{-3/2} \sum_{i=1}\ho{k_n-1} \sum_{ w \in I_i}
      \I(\varepsilon b_n/2 > Z_{nw}\ho{\slb} \geq m) \ b_n\ho{-1} \Big| 
      \sum_{j=1}\ho{k_n} \sum_{t \in I_j} f(U_t,Z_{nw}\ho{\slb}) \Big| + \op. \nonumber
     \end{align}
     This expression can be handled as in the proof of Lemma 10.1 in \cite{BerBuc18}, such that 
     \[ \lim_{m \to \infty} \limsup_{n \to \infty} \Pro (| \bar{V}_{n2} \I_{B_n} \I_{C_n} |>\delta)=0. \]
     The remaining term $|V_{n1}|$ can be treated analogously to the eponymous term in the proof of Lemma \ref{lem:cfg-limsup}.
\end{proof}

	\begin{lem} \label{lem:cov-cfg1} 		
	(a) For $x \geq 0$, as $n \to \infty$, 
	\[ \Cov(e_n(x),B_n^{\slb}) \to 2 \int_{0}^{1} h_{\slb,x}(\xi) \ 
	\mathrm{d}\xi - 2 x \varphi_{(\rm C)}(\theta), \] where 
	\begin{align*}
		h_{\slb,x}(\xi) = & \sum_{i=1}^{\infty} i  \int_{0}^{\infty} \I(y \leq \log(x)) \ \sum_{l=0}^{i} p^{(\xi x)}(l) \ p_2^{((1-\xi)x,(1-\xi)e^y)}(i-l,0) \ e^{-\theta \xi e^y} \\
		& \hspace{3cm} + \I(y > \log(x)) \ p^{(\xi x)}(i) \ e^{-\theta e^y} \ \mathrm{d}y \\ \nonumber
		- & \sum_{i=1}^{\infty} i \int_{- \infty}^{0} p^{(x)}(i) - \I(y \leq \log(x)) \sum_{l=0}^{i} p^{(\xi x)}(l) \ p_2^{((1-\xi)x,(1-\xi)e^y)}(i-l,0) \ e^{- \theta \xi e^y} \\ \nonumber
		& \ \ \ \ \ \ \ \ \ \ \ \ \ \ \ \ \ \ \ \   -\I(y > \log(x)) \ p^{(\xi x)}(i) \ e^{- \theta e^y} \ \mathrm{d}y. \nonumber
	\end{align*}
	
	\smallskip
	\noindent
	(b) We have
        \[
        2\int_{0}\ho{1} h_{\slb,x}(\xi) \ \mathrm{d}\xi  = h(x)  + x \varphi_{\rm (C)}(\theta),
        \] 
        where $h$ is defined in Lemma \ref{lem:cfg-fidis}.
\end{lem}
\begin{proof}
(a)
	We assume that both $U_s$ and $Z_{nt}^{\slb}$ are measurable with respect to the appropriate
	$\mathcal{B}\ho{\varepsilon}_{\cdot: \cdot}$ sigma-algebra; 
	the general case can be treated by multiplying with suitable indicator functions as in the proof of Lemma \ref{lem:cfg-fidis2}. 
	Let
	$
	A_j = \sum_{s \in I_j} \I\big( U_s > 1- \frac{x}{b_n} \big)$ and $D_j = \sum_{s \in I_j} \log(Z_{ns}^{\slb}).$ Then
\begin{multline*}
		\Cov(e_n(x),B_n^{\slb})
		= \frac{1}{n-b_n+1} \sum_{i=1}^{k_n} \sum_{j=1}^{k_n-1} \Cov(A_i,D_j) \\
		+ \frac{1}{n-b_n+1} \sum_{i=1}^{k_n} \Cov(A_i,  \log(Z_{n,n-b_n+1}^{\slb})).
	\end{multline*}
	The second sum is asymptotically negligible, since
	$||A_j||_2 = ||N_{b_n}^{\scs (x)}(E)||_2= O(1)$ and $||\log(Z_{n,n-b_n+1}^{\slb})||_2 =O(1)$
	by Condition \ref{cond:BerBuc}(ii) and \ref{cond2}(i). 
	Next, following the argument in the proof of Lemma B.1 in \cite{BerBuc18}, 
	we may write	
	\begin{align*}
		 \Cov(e_n(x),B_n^{\slb})
		&= \frac{1}{b_n} \Cov(A_2,D_1+D_2) + o(1) \nonumber\\
		&= \frac{1}{b_n} \sum_{t=1}^{2b_n} \Cov \Big\{ \sum_{s \in I_2} \I\big(U_s > 1- \frac{x}{b_n} \big),  \log(Z_{nt}^{\slb}) \Big\} + o(1) \nonumber\\
		&= \int_{0}^{1} f_n(\xi) + g_n(\xi) \diff\xi - 2 x \Exp \big[ \log(Z_{n1}^{\slb}) \big] + o(1), 	\end{align*}
	where 
	\begin{align*}
	f_n(\xi) &= \sum_{t=1}^{b_n} \Exp \Big[\sum_{s \in I_2} \I\big(U_s > 1- \tfrac{x}{b_n} \big) 
	\log(Z_{nt}^{\slb}) \Big] \I\big\{\xi \in [\tfrac{t-1}{b_n},\tfrac{t}{b_n}) \big\}, \\
	g_n(\xi) &= \sum_{t=b_n+1}^{2b_n} \Exp \Big[\sum_{s \in I_2} \I\big(U_s > 1- \tfrac{x}{b_n} \big) 
	\log(Z_{nt}^{\slb}) \Big] \I\big\{\xi \in [\tfrac{t-b_n-1}{b_n},\tfrac{t-b_n}{b_n}) \big\}. 
	\end{align*}
	Note that $\lim_{n\to\infty}\Exp[\log(Z_{n1}\ho{\slb})] = \varphi_{\rm (C)}(\theta)$ by uniform integrability of $\log (Z_{1:n})$, and that $\sup_{n \in \N} ||f_n||_{\infty} + ||g_n||_{\infty} < \infty$ as a consequence of  $\|\sum_{s \in I_1} \I\big( U_s > 1- \frac{x}{b_n} \big)\|_2 \times \|\log(Z_{n1}\ho{\slb})\|_2<\infty$ by Condition \ref{cond:BerBuc}(ii) and \ref{cond2}(i). Hence, the lemma is proven if we show that, for any $\xi \in (0,1)$,
	\[ 
	\lim_{n \to \infty} f_n(1-\xi) = \lim_{n \to \infty} g_n(\xi)= h_{\slb,x}(\xi). 
	\] 
	Since the proof for $f_n(1-\xi)$ is similar, we only treat $g_n(\xi)$, which can be written as 
	\[ 
	g_n(\xi) = \Exp \Big[ \sum_{s \in I_2} \ind \big(U_s > 1- \tfrac{x}{b_n} \big)  
	\log(Z_{n,\lfloor (1+\xi)b_n \rfloor+1}^{\slb}) \Big]. 
	\] 
	Let us next show joint weak convergence of $\sum_{s \in I_2} \I (U_s > 1- \frac{x}{b_n})$ and 
	$\log(Z_{n,\lfloor (1+\xi)b_n \rfloor+1}^{\slb})$. For that purpose, note that
	\begin{align}
		G_n(i,y) :=&\, \Pro \Big( \sum_{s \in I_2} \I\big(U_s > 1- \tfrac{x}{b_n} \big)=i, \ 
		\log(Z_{n,\lfloor (1+\xi)b_n \rfloor+1}^{\slb}) \geq y \Big) \nonumber\\
		=&\, \Pro \Big( \sum_{s \in I_2} \I\big(U_s > 1- \tfrac{x}{b_n} \big)=i,
		\ Z_{n,\lfloor (1+\xi)b_n \rfloor+1}^{\slb} \geq e^y \Big), \nonumber
	\end{align}
	coincides with $F_n(i,e^y)$ in the proof of Lemma B.1 in \cite{BerBuc18}. Hence, by that proof, 
	we have
	\begin{align} 
	\lim_{n \to \infty } G_n(i,y) 
	&= \sum_{l=0}^{i} p^{(\xi x)}(l)  p_2^{((1-\xi)x,(1-\xi)e^y)}(i-l,0) e^{-\theta \xi e^y} \nonumber
	\end{align} 
	for $y \leq \log x$ and
	\begin{align*}
		\lim_{n \to \infty} G_n(i,y) 
		= p^{(\xi x)}(i) \ e^{-\theta e^y}
	\end{align*}
	 for $y > \log x$. Further, note that 
	 \[
	 \lim_{n\to\infty} \Prob \big( N_{b_n}^{(x)}(E)= i \big) =p^{(x)}(i).
	 \]
	As a consequence of the previous three displays, and since weak convergence and uniform integrability implie convergence of moments, we have 
	\begin{align}
		g_n(\xi) 
		&= \sum_{i=1}\ho{\infty} i  \int_{0}\ho{\infty} \Prob\Big(  \sum_{s=b_n+1}\ho{2b_n} \I\big(U_s > 1- \frac{x}{b_n}  = i  \big) , \log(Z_{n,\lfloor (1+\xi)b_n \rfloor +1}\ho{\slb} )
		\geq y \Big) \ \mathrm{d}y \nonumber\\ 
		& \ \ \ \ \ \  \ \ \ \ \ \ \ - i \int_{-\infty}\ho{0} 
		\Prob \Big(  \sum_{s=b_n+1}\ho{2b_n} \I\big(U_s > 1- \frac{x}{b_n} \big) = i  , \log(Z_{n,\lfloor (1+\xi)b_n \rfloor +1}\ho{\slb} ) \leq y\Big) \ \mathrm{d}y \nonumber\\
		&= \sum_{i=1}\ho{\infty} i \int_{0}\ho{\infty} G_n(i,y) \ \mathrm{d}y - i \int_{-\infty}\ho{0} 
		\Prob\big( N_{b_n}^{(x)}(E)= i \big) - G_n(i,y) \ \mathrm{d}y \nonumber\\
		& \to h_{\slb,x}(\xi) \nonumber
	\end{align}
	as asserted, which implies part (a)  of the lemma.

 (b) 
  In the proof of Lemma B.3 in \cite{BerBuc18} it is shown that, for $y\le\log(x)$,
    \begin{multline*}
     S(x,y,\xi) = 
     e^{-\theta\xi e^y} \sum_{i=1}\ho{\infty} i \ \sum_{l=0}\ho{i} p\ho{(\xi x)}(l) \ p_2\ho{((1-\xi)x,(1-\xi)e\ho{y})}(i-l,0) \nonumber\\
     =  \xi x e\ho{- \theta e\ho{y}} + \Exp\left[ \xi_{11}\ho{(e\ho{y}/x)} \I(\xi_{12}\ho{(e\ho{y}/x)}=0) \right]
     \theta (1-\xi)x e\ho{-\theta e\ho{y}}, \nonumber
    \end{multline*}
    where $(\xi_{11}\ho{(y/x)},\xi_{12}\ho{(y/x)}) \sim \pi_2\ho{(y/x)}.$ Equation \eqref{eq:p2sum1} then allows to rewrite
   \begin{align*}
   S(x,y,\xi) =  \xi x e\ho{- \theta e\ho{y}} + (1-\xi) \sum_{i=1}^\infty  i p_2^{(x,e^y)}(i,0)  
   &\equiv  \xi x e\ho{- \theta e\ho{y}} + (1-\xi) T(x,y).
   \end{align*}
    As a consequence, further noting that $\sum_{i=1}\ho{\infty} i \ p\ho{(\xi x)}(i) = \xi x$, we obtain 
    \begin{multline*}
    h_{\slb,x}(\xi) = \int_0^\infty    \xi x e\ho{- \theta e\ho{y}} + \ind(y \le \log (x) )  (1-\xi) T(x,y) \diff y \\
     \qquad - \int_{-\infty}^0 x - \xi x e\ho{- \theta e\ho{y}} - \ind(y  \le \log(x)) (1-\xi) T(x,y) \diff y.
    \end{multline*}
    Then, by Fubinbi's theorem, 
    \begin{multline*}
     2 \int_{0}\ho{1} h_{\slb,x}(\xi) \ \mathrm{d}\xi 
     =
     \int_0^\infty x e^{-\theta e^y} + \ind ( y \le \log x) T(x,y) \diff y \\
      \qquad - \int_{-\infty}^0 x(1-e^{-\theta e^y}) + x -  \ind(y \le \log (x)) T(x,y) \diff y.
    \end{multline*}
    The assertion now follows from the fact that 
    \[\int_0^\infty e^{-\theta e^y} \diff y= \int_\theta^\infty \frac{e^{-z}}z \diff z = - \Ei(-\theta)\] 
    and 
    \begin{align*}
     \int_{-\infty}^0 1-e^{-\theta e^y} \diff y 
     &= 
     \int_0^\theta \frac{1-e^{-z}}{z} \diff z = (1-e^{-z})\log(z) \big|_0^\theta - \int_0^\theta e^{-z}\log(z) \diff z \\
     &=
    \log(\theta) - e^{-\theta} \log(\theta) - \Big\{ \gamma - \int_\theta^\infty e^{-z} \log(z) \diff z\Big\} \\
     &=
    \log(\theta) - e^{-\theta} \log(\theta) - \gamma + \Big\{ -e^{-z} \log(z) \big|_\theta^\infty + \int_\theta^\infty \frac{e^{-z}}z \diff z\Big\} \\
    &= \log (\theta) + \gamma - \Ei(-\theta) = - \varphi_{(\rm C)}(\theta) - \Ei(-\theta)
     \end{align*}
    after assembling terms, where $\Ei(x) = - \int_{-x}^{\infty} e^{-t}/t \ \mathrm{d}t$ for $x>0$ is the exponential integral.
        \end{proof}

    \begin{lem} \label{lem:varcfg1} One has
     \begin{align}
      \lim_{n \to \infty} \Var(B_n\ho{\slb}) &= 8 \log(2)-4 \approx 1.545. \nonumber
     \end{align}
    \end{lem}
    
    \begin{proof}
     As in the proof of Lemma \ref{lem:cov-cfg1}, we assume that the $Z_{nt}^{\slb}$ are measurable with respect to the 
     appropriate $\mathcal{B}\ho{\varepsilon}_{\cdot,\cdot}$ sigma-algebra. We may then argue as in that proof to obtain
     \begin{align}
     	\Var(B_n^{\slb})
     	&= \frac{2}{b_n} \sum_{t=1}^{b_n} \Exp \big[ \log(Z_{n1}^{\slb}) \log(Z_{n,1+t}^{\slb}) \big] - 2 \Exp[\log(Z_{n1}^{\slb})]^2 + o(1) \nonumber\\
     	&= 2 \int_{0}^{1} f_n(\xi) \ \mathrm{d}\xi - 2 \Exp[\log(Z_{n1}^{\slb})]^2 +o(1), \label{eq:stern1}
     \end{align}
     where $f_n : [0,1] \to \R$ is defined as
     \begin{align} 
     f_n(\xi) 
     = \sum_{t=1}^{b_n} \Exp[\log(Z_{n1}^{\slb})  \log(Z_{n,1+t}^{\slb})] \ \I\big( \xi \in \big[ \tfrac{t-1}{b_n},\tfrac{t}{b_n} \big) \big) \nonumber
     &= \Exp [\log(Z_{n1}^{\slb})  \log(Z_{n, \lfloor b_n \xi \rfloor+1}^{\slb})]. \nonumber
     \end{align}
     By Condition \ref{cond2}(i), we have $\Exp[\log(Z_{n1}^{\slb})] \to \varphi_{(\rm C)}(\theta)$. Further,
     \[\sup_{n \in \N} ||f_n||_{\infty} \leq \sup_{n \in \N} E[\log(Z_{n1}^{\slb})^2] < \infty,\]
    whence convergence of the integral over $f_n$ in \eqref{eq:stern1} may be concluded from the dominated convergence theorem, once we have shown pointwise 
     convergence of $f_n$. To this end we show that, for any fixed $\xi \in (0,1)$, 
          \begin{equation}
     	\big( \log(Z_{n1}^{\slb}), \ \log(Z_{n,\lfloor b_n \xi \rfloor +1}^{\slb}) \wto  \big( X^{(\xi)},Y^{(\xi)}\big) \label{weakcon}
     \end{equation}
     for some random vector $\big( X^{(\xi)},Y^{(\xi)}\big)$. This in turn 
     will imply 
     \begin{align*}
     	\lim_{n \to \infty} f_n(\xi) &= \lim_{n \to \infty} \Exp[\log(Z_{n1}^{\slb}) \ \log(Z_{n, \lfloor b_n \xi \rfloor +1}^{\slb})] 
     	= \Exp[X^{(\xi)}Y^{(\xi)}] \nonumber
     \end{align*}    
    by Condition \ref{cond2}(i) and therefore
         \begin{equation} \label{eq:varl}
      \lim_{n \to \infty} \Var(B_n\ho{\slb}) = 2 \int_{0}\ho{1} \Exp [X\ho{(\xi)} Y\ho{(\xi)}] \ \mathrm{d}\xi - 2 \varphi_{(\rm C)}(\theta)\ho{2} = 
      2\int_{0}\ho{1} \Cov (X\ho{(\xi)}, Y\ho{(\xi)}) \ \mathrm{d}\xi.
     \end{equation}
   For the proof of \eqref{weakcon}, define, for $x,y\in \R$,
     \begin{align*}
     	G_{n,\xi}(x,y) &= \Pro \big( \log(Z_{n1}^{\slb}) > x, \ \log(Z_{n,\lfloor b_n \xi \rfloor +1}^{\slb}) > y \big)
     	= \Pro \big( Z_{n1}^{\slb} > e^x, \ Z_{n,\lfloor b_n \xi \rfloor +1}^{\slb} > e^y \big), \nonumber 
     \end{align*}
     which converges to
     \begin{align*}
     	G_{\xi}(x,y) = & \exp \left( - \theta \big[\xi(e^x \wedge e^y) + (e^x \vee e^y) \big] \right)      \end{align*}  
    by the proof of Lemma B.2 in \cite{BerBuc18}. Hence, \eqref{weakcon}, where the random vector $(X^{(\xi)},Y^{(\xi)})$ has joint c.d.f.
     \begin{align*}
     	F_{\xi}(x,y) = \Pro \big( X^{(\xi)} \leq x, \ Y^{(\xi)} \leq y \big) &= 1- \Pro \big( X^{(\xi)}>x \big) - \Pro\big( Y^{(\xi)}>y \big) + G_{\xi}(x,y), \\
	&=1- \exp(- \theta e^x) - \exp(-\theta e^y) + G_{\xi}(x,y). 
     \end{align*}  
     
      We are left with calculating the right-hand side of \eqref{eq:varl}. By Lemma \ref{hilfslemEW}, we have
      \begin{align}
       V &\equiv  \int_{0}\ho{1} \Cov (X\ho{(\xi)}, Y\ho{(\xi)})  \ \mathrm{d}\xi  \nonumber \\
      &=  \int_{0}\ho{1} \int_{0}\ho{\infty} \int_{0}\ho{\infty} G_{\xi}(x,y) - e^{-\theta e^x} e^{-\theta e^y} \ \mathrm{d}x \ \mathrm{d}y \ \mathrm{d}\xi \nonumber  \\
       &\hspace{1cm} + \int_{0}\ho{1} \int_{- \infty}\ho{0} \int_{- \infty}\ho{0} F_{\xi}(x,y) - (1-e^{-\theta e^x} ) (1-e^{-\theta e^y} )\ \mathrm{d}x \ \mathrm{d}y \ \mathrm{d}\xi \nonumber\\
      & \hspace{1cm}- 2 \int_{0}\ho{1} \int_{- \infty}\ho{0} \int_{0}\ho{\infty} \Pro(X\ho{(\xi)}>x,Y\ho{(\xi)}\leq y) - e^{-\theta e^x}(1-e^{-\theta e^y}) \ \mathrm{d}x \ \mathrm{d}y \ \mathrm{d}\xi, \nonumber   \\
      & \equiv A+B-2\cdot C. \label{IntEWXi} 
     \end{align}
     We start with the first summand $A$. Recall the exponential integral $\Ei(x)= - \int_{-x}^\infty e^{-t}/t \diff t$ for $x>0$, and note that $\int_{y}^\infty e^{-\theta e^x}\diff x = -\Ei(-\theta e^y)$ for $y\in\R$ and $\int_0^1 e^{-a \xi } \diff \xi = (1-e^{-a})/a$ for $a>0$. Fubini's theorem allows to write
     \begin{align}
     A =&\, 
      \intnu \int_0^y e\ho{- \theta e\ho{y}}\Big\{  \intne e\ho{- \theta \xi e\ho{x}} \ \mathrm{d}\xi - e^{-\theta e^x} \Big\}  \diff x 
      +  
      \int_y^\infty e\ho{- \theta e\ho{x}} \Big\{ \intne e\ho{- \theta \xi e\ho{y}} \ \mathrm{d}\xi - e^{-\theta e^y} \Big\} \ \mathrm{d}x \ \mathrm{d}y \nonumber\\
      =&\,
       \int_0^\infty e\ho{- \theta e\ho{y}} \int_0^y \frac{1-e^{-\theta e^x}}{\theta e^x} - e^{-\theta e^x} \diff x 
       + 
      \int_{y}\ho{\infty}  e^{-\theta e^x} \diff x \, \Big\{ \frac{1-e\ho{- \theta e\ho{y}}}{\theta e^y}-  e^{-\theta e^y} \Big\} \diff y \nonumber\\
      =&\, 
      \intnu e\ho{- \theta e\ho{y}} \Big\{ \frac{e^{-\theta e^y}-1}{\theta e^y} - \frac{e^{-\theta}-1}{\theta}  \Big\}  +  \{ -\Ei(-\theta e^y) \} \Big\{ \frac{1-e^{-\theta e^y}}{\theta e^y}  - e^{-\theta e^y} \Big\}\diff y.\nonumber
     \end{align}
     Next, invoke the substitution $z=\theta e^y$ to obtain that
     \begin{align} \label{eq:av}
     A=\int_\theta^\infty \Big\{ \frac{e^{-z}}z - \frac1z + \frac{1-e^{-\theta}}{\theta} \Big\} \frac{e^{-z}}z - \Ei(-z) \Big\{\frac1z- \frac{e^{-z}}z  -e^{-z} \Big\}\frac1z \diff z.
     \end{align}
     A similar calculation allows to rewrite
     \begin{align*}
     B 
     &= 
     \int_0^1 \int_{-\infty}^0 \int_{-\infty}^{0}  G_\xi(x,y)  - e^{-\theta e^x} e^{-\theta e^y} \diff x \diff y \diff \xi \\
     &=
     \int_{-\infty}^0 \int_{-\infty}^y e^{-\theta e^y} \Big\{ \int_0^1 e^{-\theta \xi e^x} \diff \xi - e^{-\theta e^x} \Big\}  \diff x
     +
     \int_y^0 e^{-\theta e^x} \Big\{ \int_0^1 e^{-\theta \xi e^y} \diff \xi - e^{-\theta e^y} \Big\} \diff x \diff y \\
     &=
     \int_{-\infty}^0 e^{-\theta e^y} \int_{-\infty}^y \frac{1-e^{-\theta e^x}}{\theta e^x} - e^{-\theta e^x} \diff x
     +
     \int_y^0 e^{-\theta e^x}\diff x \,  \Big\{ \frac{1-e^{-\theta e^y}}{\theta e^y} - e^{-\theta e^y} \Big\} \diff y \\
     &=
     \int_{-\infty}^0 e^{-\theta e^y} \Big\{ \frac{e^{-\theta e^y} - 1}{\theta e^y} + 1\Big\} 
     +
    \Big\{ \Ei(-\theta) - \Ei(-\theta e^y)\Big\}  \Big\{ \frac{1-e^{-\theta e^y}}{\theta e^y} - e^{-\theta e^y} \Big\} \diff y,
    \end{align*}
    and the substitution $z=\theta e^y$ yields
    \begin{align} \label{eq:bv}
    B&=
    \int_0^\theta \Big\{ \frac{e^{-z}}{z} - \frac1z +1\Big\} \frac{e^{-z}}{z} 
    + 
    \Big\{ \Ei(-\theta) - \Ei(-z)\Big\} \Big\{ \frac{1}z - \frac{e^{-z}}z - e^{-z} \Big\} \frac1z \diff z.
     \end{align}
     Finally, regarding the term $C$, we have
     \begin{align}
     C 
     &= \nonumber
     \int_0^1 \int_{-\infty}^0 \int_{0}^{\infty} e^{-\theta e^x} e^{-\theta e^y} - G_\xi(x,y) \diff x \diff y \diff \xi \\
     &= \nonumber
     \int_{-\infty}^0 \int_0^\infty e^{-\theta e^x} \Big\{ e^{-\theta e^y}  - \int_0^1 e^{-\theta \xi e^y} \diff \xi  \Big\} \diff x \diff y\\
     &=
     \{ -\Ei(-\theta)\} \int_{-\infty}^0 e^{-\theta e^y}  -  \frac{1-e^{-\theta e^y}}{\theta e^y} \diff y
     =\label{eq:cv}
     \Ei(-\theta) \int_0^\theta \Big\{ \frac{1}z - \frac{e^{-z}}{z}- e^{-z}\Big\} \frac1z \diff z.
    \end{align}
Next, the expressions in \eqref{eq:av}, \eqref{eq:bv} and \eqref{eq:cv} may be plugged-into \eqref{IntEWXi}. Using the notations
\[
g(z) = \Big\{ \frac{1}z - \frac{e^{-z}}z - e^{-z} \Big\} \frac1z, \qquad h(z)=  \Big\{ \frac{e^{-z}}z - \frac1z + 1 \Big\} \frac{e^{-z}}z,
\]
we obtain that
\begin{align*}
V &= 
\int_\theta^\infty \Big\{ \frac{1-e^{-\theta}}{\theta} - 1\Big\}  \frac{e^{-z}}z + h(z)  + \{-\Ei(-z)\} g(z)  \diff z \\
& \hspace{1.5cm} + \int_0^\theta \Big\{ \Ei(-\theta) - \Ei(-z)\Big\} g(z) + h(z) - 2 \Ei(-\theta) g(z) \diff z\\
&= \int_0^\infty h(z) +  \{-\Ei(-z)\} g(z)  \diff z +  \frac{1-e^{-\theta}-\theta}{\theta}  \{ - \Ei(-\theta)\} - \Ei(-\theta) \int_0^\theta g(z) \diff z
\end{align*}
The first integral is independent of $\theta$, and can be seen to be equal to $4 \log 2 - 2$. Further, $\int_0^\theta g(z) \diff z=(e^{-\theta}-1+\theta)/\theta$, whence the last two summands cancel out. This proves the lemma.
\end{proof}


\subsection{Further auxiliary lemmas}
\label{subsec:pa}

          \begin{lem} \label{hilfslem Konv}
          Let $A$ be a continuous function on $[0,1]$ such that $\lim_{x\to0} A(x)/x\ho{\eta} =0$ for some $\eta\in(0,1/2)$. Further, let $H_n$ and $H$ be monotone and non-negative functions
          on $[0,1]$ with 
            \[
            \limsup_{n \to \infty} \int_{[0,1]} \frac{1}{x\ho{1-\eta}} \ \mathrm{d}H_n(x) < \infty \ \ \textrm{ and } \ \ 
          \int_{[0,1]} \frac{1}{x\ho{1-\eta}} \ \mathrm{d}H(x) < \infty.
          \]
          If $\lim_{n\to\infty} \sup_{x \in [0,1]} |B_n(x)| =0$, where $B_n := H_n-H$,
          and if there is a sequence of measurable functions $A_n$ such that \[  \lim_{n\to\infty}  \sup_{x \in [0,1]} \left| \frac{A_n(x)-A(x)}{x\ho{\eta}}
          \right| =0 ,\] then we have \[  \lim_{n\to\infty}  \int_{[0,1]} \frac{A_n(x)}{x} \ \mathrm{d}B_n(x) = 0.\]
         \end{lem}
         
         \begin{proof}
         For $r \in \N$ define the piecewise constant function 
         \[ 
         \tilde{A}_r(x) := \sum_{k=1}\ho{r} \I_{( \frac{k-1}{r},
         \frac{k}{r} ]} (x) \ \frac{A\big(k/r\big)}{k/r} 
         \] as an approximation of $A(x)/x$. We write 
         $\int_{[0,1]}A_n(x)/{x} \ \mathrm{d}B_n(x)
         = I_{n1}+I_{n2}+I_{n3},$
         where
         \begin{align*}
          I_{n1} &= \int_{[0,1]} \frac{A_n(x)-A(x)}{x} \ \mathrm{d}B_n(x), &  
          I_{n2} &= \int_{[0,1]} \frac{A(x)}{x}-\tilde{A}_r(x)  \ \mathrm{d}B_n(x),   \\
          I_{n3} &= \int_{[0,1]} \tilde{A}_r(x) \ \mathrm{d}B_n(x). \nonumber
         \end{align*}
         The first integral is bounded by
         \begin{align*}
          \int_{[0,1]} \Big| \frac{A_n(x)-A(x)}{x} \Big| \ \mathrm{d}(H_n+H)(x) 
          \leq \sup_{x \in [0,1]} \Big| \frac{A_n(x)-A(x)}{x\ho{\eta}} \Big| \ \int_{[0,1]}
          \frac{1}{x\ho{1-\eta}} \ \mathrm{d}(H_n+H)(x),
         \end{align*}
         which converges to zero by assumption. Regarding $I_{n2}$, we obtain 
         \begin{align} 
         |I_{n2}| 
         &= 
         \Big| \int_{[0,1]} \frac{A(x)-\tilde{A}_r(x)x}{x\ho{\eta}} \ \frac{1}{x\ho{1-\eta}} \ \mathrm{d}B_n(x) \Big| \nonumber\\
         & \leq 
         \sup_{x \in [0,1]} \Big| \frac{A(x)-\tilde{A}_r(x)x}{x\ho{\eta}} \Big| \ 
         \int_{[0,1]} \frac{1}{x\ho{1-\eta}} 
         \ \mathrm{d}(H_n+H)(x). \label{supAtilde}
         \end{align}
         By uniform continuity of $x \mapsto A(x)/x\ho{\eta}$ on $[0,1]$, we have \[ \sup_{x \in [0,1]} \Big| 
         \frac{A(x)-\tilde{A}_r(x)x}{x\ho{\eta}} \Big| \to 0 \ \textrm{ for } r \to \infty. \] Thus, the limes superior (for $n\to\infty$) of the expression on the right-hand side of (\ref{supAtilde}) can be made 
         arbitrarily small by increasing $r$. 
         Finally, we can bound $|I_{n3}|$ as follows 
         \begin{align}
          |I_{n3}| 
          \leq 
         \sum_{k=1}\ho{r} \frac{|A(k/r)|}{k/r} \Big| \int_{[0,1]} 
         \I_{\big( \frac{k-1}{r}, \frac{k}{r} \big]} (x) \ \mathrm{d}B_n(x) \Big| 
         &= 
         \sum_{k=1}\ho{r} \frac{|A(k/r)|}{k/r} \Big| B_n\Big(\frac{k}{r}\Big)-B_n\Big( \frac{k-1}{r} \Big) 
         \Big| \nonumber\\
         &\leq 2r\ho{2} \sup_{x \in [0,1]} |A(x)| \sup_{x \in [0,1]} |B_n(x)|, \nonumber
         \end{align}
         which converges to zero by assumption. 
\end{proof}

    \begin{lem} \label{hilfslemEW} Let $X$ and $Y$ be real-valued random variables such that $XY$ is integrable. Then, 
     \begin{align}
       E[XY] = & \intnu \intnu \Pro (X>x, Y>y) \diff x\diff y + \intmn \intmn \Pro(X \leq x, Y \leq y) 
      \diff x\diff y \nonumber\\
       &- \intmn \intnu \Pro(X>x, Y \leq y) \diff x\diff y - \intnu \intmn \Pro(X \leq x, Y > y) \diff x\diff y. \nonumber
     \end{align}
    \end{lem}

\begin{proof} This is a standard calculation based on Fubini's theorem.
\end{proof}

\bibliographystyle{chicago}
\bibliography{biblio}

\end{document}